\documentclass[11pt]{amsart}
\usepackage{amsmath, amsthm, amssymb, enumerate}
\usepackage{graphicx, mathrsfs, verbatim, tikz}
\usepackage[all, cmtip]{xy}
\usepackage[vcentermath]{youngtab}
\usepackage{fullpage}
\usepackage{hyperref}

\usetikzlibrary{arrows,matrix}
\tikzset{tab/.style={matrix of math nodes,column sep=-.35, row sep=-.35,text height=7pt,text width=7pt,align=center,inner sep=2,font=\footnotesize}}

\newtheorem{thm}{Theorem}[section]
\newtheorem{lemma}[thm]{Lemma}
\newtheorem{prop}[thm]{Proposition}
\newtheorem{cor}[thm]{Corollary}
\newtheorem{conj}[thm]{Conjecture}

\theoremstyle{definition}
\newtheorem{dfn}[thm]{Definition}
\newtheorem{ex}[thm]{Example}
\newtheorem{remark}[thm]{Remark}

\numberwithin{equation}{section} 
\numberwithin{table}{section}

\newcommand{\inner}[2]{\left\langle #1, #2 \right\rangle}
\newcommand{\virtual}[1]{\widehat{#1}}
\newcommand{\iso}{\cong}
\newcommand{\absval}[1]{\left\lvert #1 \right\rvert}
\newcommand{\clfw}{\overline{\Lambda}} 
\newcommand{\column}[1]{\left[ #1 \right]} 
\newcommand{\case}[1]{\vspace{12pt}\noindent
\underline{#1}: \nopagebreak}
\newcommand{\hwRC}{\RC^*} 

\newcommand{\ZZ}{\mathbb{Z}}

\newcommand{\HH}{\mathcal{H}}
\newcommand{\g}{\mathfrak{g}}

\DeclareMathOperator{\RC}{RC} 
\DeclareMathOperator{\wt}{wt} 
\DeclareMathOperator{\ls}{ls} 
\DeclareMathOperator{\lt}{lt} 
\DeclareMathOperator{\cc}{cc} 
\DeclareMathOperator{\fillmap}{fill} 

\usepackage{listings}
\lstdefinelanguage{Sage}[]{Python}
{morekeywords={False,sage,True},sensitive=true}
\definecolor{dblackcolor}{rgb}{0.0,0.0,0.0}
\definecolor{dbluecolor}{rgb}{0.01,0.02,0.7}
\definecolor{dgreencolor}{rgb}{0.2,0.4,0.0}
\definecolor{dgraycolor}{rgb}{0.30,0.3,0.30}

\begin{document}
\title{\bf Crystal structure on rigged configurations and the filling map}

\author[A. Schilling]{Anne Schilling}
\address[Anne Schilling]{Department of Mathematics, UC Davis, One Shields Ave., Davis, CA 95616-8633, U.S.A.}
\email{anne@math.ucdavis.edu}

\author[T. Scrimshaw]{Travis Scrimshaw}
\address[Travis Scrimshaw]{Department of Mathematics, UC Davis, One Shields Ave., Davis, CA 95616-8633, U.S.A.}
\email{scrimsha@math.ucdavis.edu}

\thanks{A.S. was supported by NSF grants DMS--1001256, OCI--1147247 and a Simons Fellowship.}
\thanks{T.S. was supported by NSF grant OCI--1147247.}

\begin{abstract}
In this paper, we extend work of the first author on a crystal structure on rigged configurations of simply-laced type
to all non-exceptional affine types using the technology of virtual rigged configurations and crystals. Under the bijection 
between rigged configurations and tensor products of Kirillov--Reshetikhin crystals specialized to a single tensor factor, 
we obtain a new tableaux model for Kirillov--Reshetikhin crystals. This is related to the model in terms of 
Kashiwara--Nakashima tableaux via a filling map, generalizing the recently discovered filling map in type $D_n^{(1)}$.
\end{abstract}

\maketitle

\section{Introduction}

Rigged configurations index solutions of the Bethe Ansatz equations used to solve integrable
systems such as the XXX spin 1/2 Heisenberg chain. Despite their analytic origin, rigged configurations have fascinating combinatorial
properties. Kerov, Kirillov, and Reshetikhin~\cite{KKR86,KR86} introduced them in type $A_n^{(1)}$ and showed that they are in 
bijection with semi-standard tableaux. This bijection $\Phi$ is defined in a rather recursive manner which leaves many of its properties obscure.
For example, the bijection preserves certain statistics (cocharge and energy) and also maps the
very intricate combinatorial $R$-matrix to the identity map on rigged configurations.

Given the recursive nature of the bijection, it is desirable to explain it in more algebraic terms.
In~\cite{S06} a classical crystal structure was imposed on rigged configurations for simply-laced
types by verifying that the Stembridge~\cite{Stem.2003} local axioms hold. This led to the generalization~\cite{DS06}
of the Kerov, Kirillov, and Reshetikhin rigged configurations, which correspond to highest weight 
elements in a tensor product of single row representations in this setting. One of the main achievements in the current 
paper is the generalization of the classical crystal structure on rigged configurations to non-simply-laced types. 
Since the Stembridge local rules only characterize simply-laced highest weight crystals, we employ the method of virtual crystals initially
introduced in~\cite{OSS03III,OSS03II} to achieve this goal. The virtual crystal method realizes a crystal of non-simply-laced
type in terms of an embedding into a simply-laced crystal. In terms of the Kashiwara--Nakashima (KN) tableaux model~\cite{KN94}, it is 
not always easy to characterize the image of these embeddings in order to check that the virtual crystal is ``aligned" and hence in bijection 
with the expected non-simply-laced crystal. One of the major advantages of using rigged configurations is that the image is easy to 
compute and the alignedness property readily follows (see Section~\ref{sec:virtual_rc}).

The bijection $\Phi$ was generalized beyond type $A_n^{(1)}$ to arbitrary non-exceptional types in~\cite{OSS03} and to type $E_6^{(1)}$
in~\cite{OS12} for tensor products of Kirillov--Reshetikhin (KR) crystals indexed by the vector representation.
In the spirit of~\cite{KSS02} for type $A_n^{(1)}$, it is conjectured that $\Phi$ can be extended to arbitrary tensor products of 
KR crystals. This involves certain splitting maps (of the rectangles that index the KR crystals) and,
beyond type $A_n^{(1)}$, also a ``filling map'' as first pointed out in~\cite{S05} and fully established in type $D_n^{(1)}$
in~\cite{OSS13}. More precisely, the recursive algorithm for $\Phi$ yields rectangular tableaux (which are not necessarily semi-standard), 
which we coin Kirillov--Reshetikhin (KR)
tableaux following~\cite{OSS13}, that are similar to the KN tableaux appearing in the theory of crystal bases.
The filling map is a crystal isomorphism between these two versions of tableaux.
The second main result of the current paper is an explicit description of the filling map on classically highest weight elements for 
all non-exceptional types.

In order to utilize the algebraic structure of crystal bases to its fullest degree, it is necessary not only to define
a classical crystal structure, but also affine crystal operators. Since tensor products of affine KR crystals
are connected~\cite{Kashiwara.2002}, this would provide a description of the bijection between rigged configurations
and tensor products of KR crystals as an affine crystal isomorphism. For a single tensor factor in type $D_n^{(1)}$,
this was achieved in~\cite{OSS13}. In this paper, we extend the result of~\cite{OSS13} to any type that embeds into
type $D_n^{(1)}$. Note that semi-infinite tensor products of perfect KR crystals play an important role in the path realization~\cite{HK02}
of highest weight representations for $\g$. Also, the characters of KR crystals correspond to solutions of certain $Q$-systems, and their 
corresponding modules are solutions of certain $T$-systems~\cite{KNS11}. It is conjectured that KR crystals are universal objects in the 
category of finite-dimensional $U_q^{\prime}(\g)$-crystals. 
In addition, tensor products of KR crystals are related to Macdonald polynomials and $q$-deformed Whittaker functions~\cite{LNSSS14,ST12}.

This paper is organized as follows. We begin with background on crystals, rigged configurations, the Kleber algorithm, and 
their corresponding virtual counterparts in Section~\ref{sec:background}. 
We provide an explicit description of crystal operators on non-simply-laced rigged configurations in Section~\ref{sec:virtual_rc}.
In Section~\ref{sec:filling_map}, the filling map for all non-exceptional types is described and proved on highest weight elements for 
a single tensor factor, along with the existence of a statistics preserving bijection.
The affine crystal structure for the types that embed virtually into type $D_n^{(1)}$ is
given in Section~\ref{sec:affine_structure}. We conclude in Section~\ref{sec:single_factors} by showing that
the aforementioned bijection commutes with the virtualization map on highest weight elements for a single tensor factor.
The appendix provides a proof of a generalization of a result of Baker~\cite{baker2000}.

\subsection*{Acknowledgements}
Both authors would like to thank ICERM for holding the program ``Automorphic Forms, Combinatorial Representation Theory and 
Multiple Dirichlet Series" during the spring of 2013 where part of this work was done. We would also like to thank
Masato Okado for helpful discussions. 

This work benefitted from computations in Sage~\cite{sage, combinat}. Rigged configurations and the bijection to KR tableaux
were implemented by the second author.

\section{Background}
\label{sec:background}

In this section we provide background on crystals, rigged configurations, Kirillov--Reshetikhin crystals, virtual crystals, and
the Kleber and virtual Kleber algorithm.

Let $\g$ be an affine Kac--Moody Lie algebra with index set $I$, generalized Cartan matrix $A = (A_{ij})_{i,j \in I}$, 
fundamental weights $\{\Lambda_i \mid i \in I\}$, weight lattice $P$, root lattice $Q$, simple roots $\{\alpha_i \mid i \in I\}$, 
and simple coroots $\{\alpha_i^{\vee} \mid i \in I\}$. Denote by $P^{\vee}$ and $Q^{\vee}$ the coweight and coroot lattice, respectively.
Our conventions for the Dynkin diagrams and Cartan matrices follow Kac~\cite{kac90} (see also Figure~\ref{fig:Dynkin}).
Let $\inner{\cdot}{\cdot} \colon P^{\vee} \times P \to \ZZ$ be the canonical pairing defined by the evaluation pairing.
In particular, $\inner{\alpha_i^{\vee}}{\alpha_j} = A_{ij}$.
In addition, let $\g_0$ be the classical subalgebra of $\g$ with index set $I_0 = I \setminus \{0\}$, fundamental weights
$\{ \clfw_i \mid i \in I_0 \}$, and weight lattice $\overline{P}=\bigoplus_{i\in I_0} \mathbb{Z} \clfw_i$.

{\unitlength=.8pt
\begin{figure}
\begin{tabular}[t]{rl}
$A_1^{(1)}$:&
\begin{picture}(26,20)(-5,-5)
\multiput( 0,0)(20,0){2}{\circle{6}}
\multiput(2.85,-1)(0,2){2}{\line(1,0){14.3}}
\put(0,-5){\makebox(0,0)[t]{$0$}}
\put(20,-5){\makebox(0,0)[t]{$1$}}
\put( 6, 0){\makebox(0,0){$<$}}
\put(14, 0){\makebox(0,0){$>$}}
\end{picture}
\\
&
\\
\begin{minipage}[b]{4em}
\begin{flushright}
$A_n^{(1)}$:\\$(n \ge 2)$
\end{flushright}
\end{minipage}&
\begin{picture}(106,40)(-5,-5)
\multiput( 0,0)(20,0){2}{\circle{6}}
\multiput(80,0)(20,0){2}{\circle{6}}
\put(50,20){\circle{6}}
\multiput( 3,0)(20,0){2}{\line(1,0){14}}
\multiput(63,0)(20,0){2}{\line(1,0){14}}
\multiput(39,0)(4,0){6}{\line(1,0){2}}
\put(2.78543,1.1142){\line(5,2){44.429}}
\put(52.78543,18.8858){\line(5,-2){44.429}}
\put(0,-5){\makebox(0,0)[t]{$1$}}
\put(20,-5){\makebox(0,0)[t]{$2$}}
\put(80,-5){\makebox(0,0)[t]{$n\!\! -\!\! 1$}}
\put(100,-5){\makebox(0,0)[t]{$n$}}
\put(55,20){\makebox(0,0)[lb]{$0$}}
\end{picture}
\\
&
\\
\begin{minipage}[b]{4em}
\begin{flushright}
$B_n^{(1)}$:\\$(n \ge 3)$
\end{flushright}
\end{minipage}&
\begin{picture}(126,40)(-5,-5)
\multiput(0,0)(20,0){3}{\circle{6}}
\multiput(100,0)(20,0){2}{\circle{6}}
\put(20,20){\circle{6}}
\multiput( 3,0)(20,0){3}{\line(1,0){14}}
\multiput(83,0)(20,0){1}{\line(1,0){14}}
\put(20,3){\line(0,1){14}}
\multiput(102.85,-1)(0,2){2}{\line(1,0){14.3}} 
\multiput(59,0)(4,0){6}{\line(1,0){2}} 
\put(110,0){\makebox(0,0){$>$}}
\put(0,-5){\makebox(0,0)[t]{$1$}}
\put(20,-5){\makebox(0,0)[t]{$2$}}
\put(40,-5){\makebox(0,0)[t]{$3$}}
\put(100,-5){\makebox(0,0)[t]{$n\!\! -\!\! 1$}}
\put(120,-5){\makebox(0,0)[t]{$n$}}
\put(25,20){\makebox(0,0)[l]{$0$}}
\put(120,13){\makebox(0,0)[t]{$2$}}
\end{picture}
\\
&
\\
\begin{minipage}[b]{4em}
\begin{flushright}
$C_n^{(1)}$:\\$(n \ge 2)$
\end{flushright}
\end{minipage}&
\begin{picture}(126,20)(-5,-5)
\multiput( 0,0)(20,0){3}{\circle{6}}
\multiput(100,0)(20,0){2}{\circle{6}}
\multiput(23,0)(20,0){2}{\line(1,0){14}}
\put(83,0){\line(1,0){14}}
\multiput( 2.85,-1)(0,2){2}{\line(1,0){14.3}} 
\multiput(102.85,-1)(0,2){2}{\line(1,0){14.3}} 
\multiput(59,0)(4,0){6}{\line(1,0){2}} 
\put(10,0){\makebox(0,0){$>$}}
\put(110,0){\makebox(0,0){$<$}}
\put(0,-5){\makebox(0,0)[t]{$0$}}
\put(20,-5){\makebox(0,0)[t]{$1$}}
\put(40,-5){\makebox(0,0)[t]{$2$}}
\put(100,-5){\makebox(0,0)[t]{$n\!\! -\!\! 1$}}
\put(120,-5){\makebox(0,0)[t]{$n$}}

\put(20,13){\makebox(0,0)[t]{$2$}}
\put(40,13){\makebox(0,0)[t]{$2$}}
\put(100,13){\makebox(0,0)[t]{$2$}}
\end{picture}
\\
&
\\
\begin{minipage}[b]{4em}
\begin{flushright}
$D_n^{(1)}$:\\$(n \ge 4)$
\end{flushright}
\end{minipage}&
\begin{picture}(106,40)(-5,-5)
\multiput( 0,0)(20,0){2}{\circle{6}}
\multiput(80,0)(20,0){2}{\circle{6}}
\multiput(20,20)(60,0){2}{\circle{6}}
\multiput( 3,0)(20,0){2}{\line(1,0){14}}
\multiput(63,0)(20,0){2}{\line(1,0){14}}
\multiput(39,0)(4,0){6}{\line(1,0){2}}
\multiput(20,3)(60,0){2}{\line(0,1){14}}
\put(0,-5){\makebox(0,0)[t]{$1$}}
\put(20,-5){\makebox(0,0)[t]{$2$}}
\put(80,-5){\makebox(0,0)[t]{$n\!\! - \!\! 2$}}
\put(103,-5){\makebox(0,0)[t]{$n\!\! -\!\! 1$}}
\put(25,20){\makebox(0,0)[l]{$0$}}
\put(85,20){\makebox(0,0)[l]{$n$}}
\end{picture}
\\
&
\\
$E_6^{(1)}$:&
\begin{picture}(86,60)(-5,-5)
\multiput(0,0)(20,0){5}{\circle{6}}
\multiput(40,20)(0,20){2}{\circle{6}}
\multiput(3,0)(20,0){4}{\line(1,0){14}}
\multiput(40, 3)(0,20){2}{\line(0,1){14}}
\put( 0,-5){\makebox(0,0)[t]{$1$}}
\put(20,-5){\makebox(0,0)[t]{$2$}}
\put(40,-5){\makebox(0,0)[t]{$3$}}
\put(60,-5){\makebox(0,0)[t]{$4$}}
\put(80,-5){\makebox(0,0)[t]{$5$}}
\put(45,20){\makebox(0,0)[l]{$6$}}
\put(45,40){\makebox(0,0)[l]{$0$}}
\end{picture}
\\
&
\\
$E_7^{(1)}$:&
\begin{picture}(126,40)(-5,-5)
\multiput(0,0)(20,0){7}{\circle{6}}
\put(60,20){\circle{6}}
\multiput(3,0)(20,0){6}{\line(1,0){14}}
\put(60, 3){\line(0,1){14}}
\put( 0,-5){\makebox(0,0)[t]{$0$}}
\put(20,-5){\makebox(0,0)[t]{$1$}}
\put(40,-5){\makebox(0,0)[t]{$2$}}
\put(60,-5){\makebox(0,0)[t]{$3$}}
\put(80,-5){\makebox(0,0)[t]{$4$}}
\put(100,-5){\makebox(0,0)[t]{$5$}}
\put(120,-5){\makebox(0,0)[t]{$6$}}
\put(65,20){\makebox(0,0)[l]{$7$}}
\end{picture}
\\
&
\\
$E_8^{(1)}$:&
\begin{picture}(146,40)(-5,-5)
\multiput(0,0)(20,0){8}{\circle{6}}
\put(100,20){\circle{6}}
\multiput(3,0)(20,0){7}{\line(1,0){14}}
\put(100, 3){\line(0,1){14}}
\put( 0,-5){\makebox(0,0)[t]{$0$}}
\put(20,-5){\makebox(0,0)[t]{$1$}}
\put(40,-5){\makebox(0,0)[t]{$2$}}
\put(60,-5){\makebox(0,0)[t]{$3$}}
\put(80,-5){\makebox(0,0)[t]{$4$}}
\put(100,-5){\makebox(0,0)[t]{$5$}}
\put(120,-5){\makebox(0,0)[t]{$6$}}
\put(140,-5){\makebox(0,0)[t]{$7$}}
\put(105,20){\makebox(0,0)[l]{$8$}}
\end{picture}
\\
&
\\
\end{tabular}
\begin{tabular}[t]{rl}
$F_4^{(1)}$:&
\begin{picture}(86,20)(-5,-5)
\multiput( 0,0)(20,0){5}{\circle{6}}
\multiput( 3,0)(20,0){2}{\line(1,0){14}}
\multiput(42.85,-1)(0,2){2}{\line(1,0){14.3}} 
\put(63,0){\line(1,0){14}}
\put(50,0){\makebox(0,0){$>$}}
\put( 0,-5){\makebox(0,0)[t]{$0$}}
\put(20,-5){\makebox(0,0)[t]{$1$}}
\put(40,-5){\makebox(0,0)[t]{$2$}}
\put(60,-5){\makebox(0,0)[t]{$3$}}
\put(80,-5){\makebox(0,0)[t]{$4$}}
\put(60,13){\makebox(0,0)[t]{$2$}}
\put(80,13){\makebox(0,0)[t]{$2$}}
\end{picture}
\\
&
\\
$G_2^{(1)}$:&
\begin{picture}(46,20)(-5,-5)
\multiput( 0,0)(20,0){3}{\circle{6}}
\multiput( 3,0)(20,0){2}{\line(1,0){14}}
\multiput(22.68,-1.5)(0,3){2}{\line(1,0){14.68}}
\put( 0,-5){\makebox(0,0)[t]{$0$}}
\put(20,-5){\makebox(0,0)[t]{$1$}}
\put(40,-5){\makebox(0,0)[t]{$2$}}
\put(30,0){\makebox(0,0){$>$}}
\put(40,13){\makebox(0,0)[t]{$3$}}
\end{picture}
\\
&
\\
$A^{(2)}_2$:&
\begin{picture}(26,20)(-5,-5)
\multiput( 0,0)(20,0){2}{\circle{6}}
\multiput(2.958,-0.5)(0,1){2}{\line(1,0){14.084}}
\multiput(2.598,-1.5)(0,3){2}{\line(1,0){14.804}}
\put(0,-5){\makebox(0,0)[t]{$0$}}
\put(20,-5){\makebox(0,0)[t]{$1$}}
\put(10,0){\makebox(0,0){$<$}}
\put(0,13){\makebox(0,0)[t]{$2$}}
\put(20,13){\makebox(0,0)[t]{$2$}}
\end{picture}
\\
&
\\
\begin{minipage}[b]{4em}
\begin{flushright}
$A_{2n}^{(2)}$:\\$(n \ge 2)$
\end{flushright}
\end{minipage}&
\begin{picture}(126,20)(-5,-5)
\multiput( 0,0)(20,0){3}{\circle{6}}
\multiput(100,0)(20,0){2}{\circle{6}}
\multiput(23,0)(20,0){2}{\line(1,0){14}}
\put(83,0){\line(1,0){14}}
\multiput( 2.85,-1)(0,2){2}{\line(1,0){14.3}} 
\multiput(102.85,-1)(0,2){2}{\line(1,0){14.3}} 
\multiput(59,0)(4,0){6}{\line(1,0){2}} 
\put(10,0){\makebox(0,0){$<$}}
\put(110,0){\makebox(0,0){$<$}}
\put(0,-5){\makebox(0,0)[t]{$0$}}
\put(20,-5){\makebox(0,0)[t]{$1$}}
\put(40,-5){\makebox(0,0)[t]{$2$}}
\put(100,-5){\makebox(0,0)[t]{$n\!\! -\!\! 1$}}
\put(120,-5){\makebox(0,0)[t]{$n$}}
\put(0,13){\makebox(0,0)[t]{$2$}}
\put(20,13){\makebox(0,0)[t]{$2$}}
\put(40,13){\makebox(0,0)[t]{$2$}}
\put(100,13){\makebox(0,0)[t]{$2$}}
\put(120,13){\makebox(0,0)[t]{$2$}}
\put(120,13){\makebox(0,0)[t]{$2$}}
\end{picture}
\\
&
\\
$A^{(2)\dagger}_2$:&
\begin{picture}(26,20)(-5,-5)
\multiput( 0,0)(20,0){2}{\circle{6}}
\multiput(2.958,-0.5)(0,1){2}{\line(1,0){14.084}}
\multiput(2.598,-1.5)(0,3){2}{\line(1,0){14.804}}
\put(0,-5){\makebox(0,0)[t]{$0$}}
\put(20,-5){\makebox(0,0)[t]{$1$}} \put(10,0){\makebox(0,0){$>$}}
\put(0,13){\makebox(0,0)[t]{$$}} \put(20,13){\makebox(0,0)[t]{$$}}
\end{picture}
\\
&
\\
\begin{minipage}[b]{4em}
\begin{flushright}
$A_{2n}^{(2)\dagger}$:\\$(n \ge 2)$
\end{flushright}
\end{minipage}&
\begin{picture}(126,20)(-5,-5)
\multiput( 0,0)(20,0){3}{\circle{6}}
\multiput(100,0)(20,0){2}{\circle{6}}
\multiput(23,0)(20,0){2}{\line(1,0){14}}
\put(83,0){\line(1,0){14}}
\multiput( 2.85,-1)(0,2){2}{\line(1,0){14.3}} 
\multiput(102.85,-1)(0,2){2}{\line(1,0){14.3}} 
\multiput(59,0)(4,0){6}{\line(1,0){2}} 
\put(10,0){\makebox(0,0){$>$}} \put(110,0){\makebox(0,0){$>$}}
\put(0,-5){\makebox(0,0)[t]{$0$}}
\put(20,-5){\makebox(0,0)[t]{$1$}}
\put(40,-5){\makebox(0,0)[t]{$2$}}
\put(100,-5){\makebox(0,0)[t]{$n\!\! -\!\! 1$}}
\put(120,-5){\makebox(0,0)[t]{$n$}}
\put(0,13){\makebox(0,0)[t]{$$}}
\put(20,13){\makebox(0,0)[t]{$$}}
\put(40,13){\makebox(0,0)[t]{$$}}
\put(100,13){\makebox(0,0)[t]{$$}}
\put(120,13){\makebox(0,0)[t]{$$}}
\put(120,13){\makebox(0,0)[t]{$$}}
\end{picture}
\\
&
\\
\begin{minipage}[b]{4em}
\begin{flushright}
$A_{2n-1}^{(2)}$:\\$(n \ge 3)$
\end{flushright}
\end{minipage}&
\begin{picture}(126,40)(-5,-5)
\multiput( 0,0)(20,0){3}{\circle{6}}
\multiput(100,0)(20,0){2}{\circle{6}}
\put(20,20){\circle{6}}
\multiput( 3,0)(20,0){3}{\line(1,0){14}}
\multiput(83,0)(20,0){1}{\line(1,0){14}}
\put(20,3){\line(0,1){14}}
\multiput(102.85,-1)(0,2){2}{\line(1,0){14.3}} 
\multiput(59,0)(4,0){6}{\line(1,0){2}} 
\put(110,0){\makebox(0,0){$<$}}
\put(0,-5){\makebox(0,0)[t]{$1$}}
\put(20,-5){\makebox(0,0)[t]{$2$}}
\put(40,-5){\makebox(0,0)[t]{$3$}}
\put(100,-5){\makebox(0,0)[t]{$n\!\! -\!\! 1$}}
\put(120,-5){\makebox(0,0)[t]{$n$}}
\put(25,20){\makebox(0,0)[l]{$0$}}

\put(120,13){\makebox(0,0)[t]{$2$}}
\end{picture}
\\
&
\\
\begin{minipage}[b]{4em}
\begin{flushright}
$D_{n+1}^{(2)}$:\\$(n \ge 2)$
\end{flushright}
\end{minipage}&
\begin{picture}(126,20)(-5,-5)
\multiput( 0,0)(20,0){3}{\circle{6}}
\multiput(100,0)(20,0){2}{\circle{6}}
\multiput(23,0)(20,0){2}{\line(1,0){14}}
\put(83,0){\line(1,0){14}}
\multiput( 2.85,-1)(0,2){2}{\line(1,0){14.3}} 
\multiput(102.85,-1)(0,2){2}{\line(1,0){14.3}} 
\multiput(59,0)(4,0){6}{\line(1,0){2}} 
\put(10,0){\makebox(0,0){$<$}}
\put(110,0){\makebox(0,0){$>$}}
\put(0,-5){\makebox(0,0)[t]{$0$}}
\put(20,-5){\makebox(0,0)[t]{$1$}}
\put(40,-5){\makebox(0,0)[t]{$2$}}
\put(100,-5){\makebox(0,0)[t]{$n\!\! -\!\! 1$}}
\put(120,-5){\makebox(0,0)[t]{$n$}}

\put(20,13){\makebox(0,0)[t]{$2$}}
\put(40,13){\makebox(0,0)[t]{$2$}}
\put(100,13){\makebox(0,0)[t]{$2$}}
\end{picture}
\\
&
\\
$E_6^{(2)}$:&
\begin{picture}(86,20)(-5,-5)
\multiput( 0,0)(20,0){5}{\circle{6}}
\multiput( 3,0)(20,0){2}{\line(1,0){14}}
\multiput(42.85,-1)(0,2){2}{\line(1,0){14.3}} 
\put(63,0){\line(1,0){14}}
\put(50,0){\makebox(0,0){$<$}}
\put( 0,-5){\makebox(0,0)[t]{$0$}}
\put(20,-5){\makebox(0,0)[t]{$1$}}
\put(40,-5){\makebox(0,0)[t]{$2$}}
\put(60,-5){\makebox(0,0)[t]{$3$}}
\put(80,-5){\makebox(0,0)[t]{$4$}}

\put(60,13){\makebox(0,0)[t]{$2$}}
\put(80,13){\makebox(0,0)[t]{$2$}}
\end{picture}
\\
&
\\
$D_4^{(3)}$:&
\begin{picture}(46,20)(-5,-5)
\multiput( 0,0)(20,0){3}{\circle{6}}
\multiput( 3,0)(20,0){2}{\line(1,0){14}}
\multiput(22.68,-1.5)(0,3){2}{\line(1,0){14.68}}
\put( 0,-5){\makebox(0,0)[t]{$0$}}
\put(20,-5){\makebox(0,0)[t]{$1$}}
\put(40,-5){\makebox(0,0)[t]{$2$}}
\put(30,0){\makebox(0,0){$<$}}

\put(40,13){\makebox(0,0)[t]{$3$}}
\end{picture}
\\
&
\\
\end{tabular}
\caption{Dynkin diagrams for $X^{(r)}_N$.
The enumeration of the nodes with
$I = \{0,1,\ldots, n\}$ is specified under or the right side of the nodes.
In addition, the numbers $t_i$ (resp. $t^\vee_i$) defined in
\eqref{equation.t} are attached \emph{above} the nodes for $r=1$ (resp. $r>1$)
if and only if $t_i \neq 1$ (resp. $t^\vee_i \neq 1$).\label{fig:Dynkin}}
\end{figure}
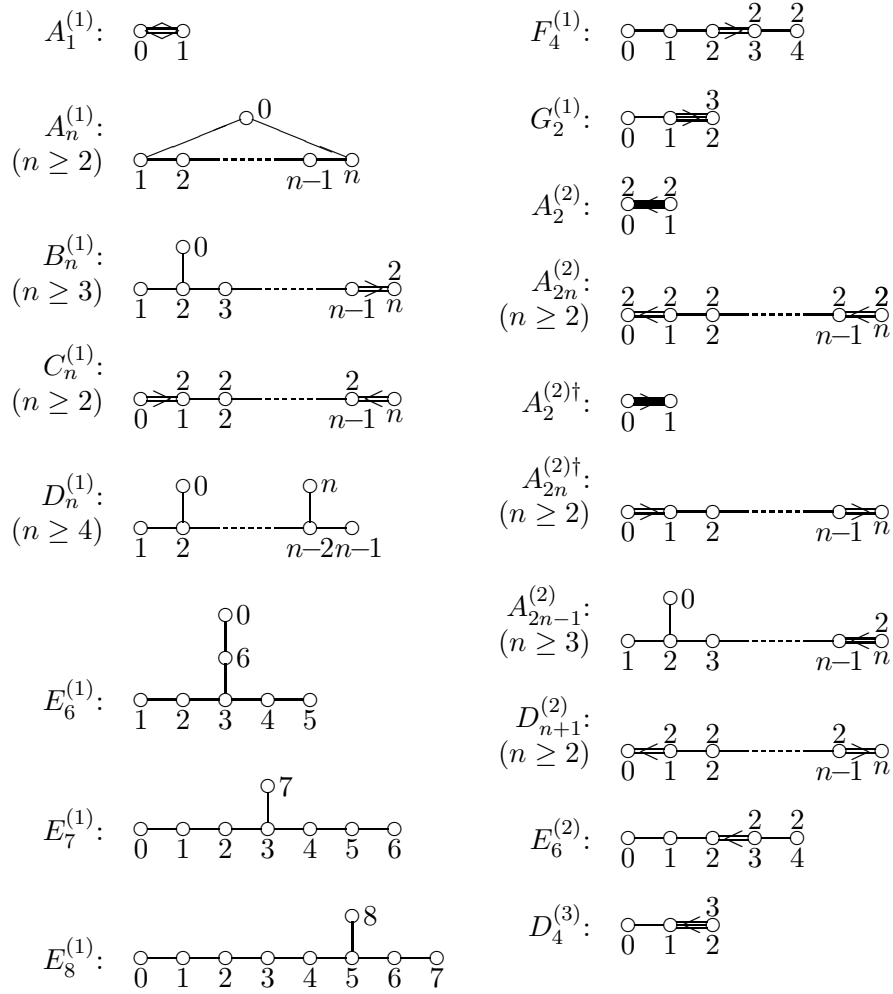}

\subsection{Crystals}
We begin by giving the axiomatic definition of a crystal.

\begin{dfn}
\label{definition.abstract crystal}
An \emph{abstract $U_q(\g)$-crystal} is a non-empty set $B$ together with maps
\[
	\wt \colon B \to P, \qquad
	\varepsilon_i, \varphi_i \colon B \to \ZZ \sqcup \{-\infty\}, \qquad
	e_i, f_i \colon B \to B \sqcup \{0\},
\]
subject to the following conditions:
\begin{enumerate}[(1)]
\item $\varphi_i(b) = \varepsilon_i(b) + \langle \alpha^\vee_i, \wt(b)\rangle$ for all $i \in I$,
\item if $b\in B$ satisfies $e_ib \neq 0$, then
  \begin{enumerate}[(a)]
  \item $\varepsilon_i(e_ib) = \varepsilon_i(b) - 1$,
  \item $\varphi_i(e_ib) = \varphi_i(b) + 1$,
  \item $\wt(e_ib) = \wt(b) + \alpha_i$,
  \end{enumerate}
\item if $b\in B$ satisfies $f_ib \neq 0$, then
  \begin{enumerate}[(a)]
  \item $\varepsilon_i(f_ib) = \varepsilon_i(b) + 1$,
  \item $\varphi_i(f_ib) = \varphi_i(b) - 1$,
  \item $\wt(f_ib) = \wt(b) - \alpha_i$,
  \end{enumerate}
\item $f_ib = b^{\prime}$ if and only if $b = e_i b^{\prime}$ for $b,b^{\prime} \in B$ and $i \in I$,
\item if $\varphi_i(b) = -\infty$ for $b\in B$, then $e_i b = f_i b = 0$.
\end{enumerate}
The maps $e_i$ and $f_i$ ($i \in I$) are called the \emph{crystal} or \emph{Kashiwara operators}.
\end{dfn}

We call a crystal \emph{regular} if
\begin{equation}
\label{equation.eps phi}
	\varepsilon_i(b) = \max\{ k \in \ZZ \mid e_i^k(b) \neq 0 \} \quad \text{and} \quad
	\varphi_i(b) = \max\{ k \in \ZZ \mid f_i^k(b) \neq 0 \}
\end{equation}
for all $i \in I$ and $b \in B$. In this case, we depict the entire $i$-string diagrammatically as
\[ \xymatrix{
e_i^{\varepsilon_i(b)} b \ar[r]^i & \cdots \ar[r]^i & e_i b \ar[r]^i & b \ar[r]^i & f_i b \ar[r]^i & \cdots \ar[r]^i & f_i^{\varphi_i(b)} b.
}\]

Let $B_1$ and $B_2$ be abstract $U_q(\g)$-crystals.  The tensor product $B_2 \otimes B_1$ is defined to be the Cartesian product 
$B_2\times B_1$ equipped with crystal operators given by
\begin{align*}
e_i(b_2 \otimes b_1) &= \begin{cases}
e_i b_2 \otimes b_1 & \text{if } \varepsilon_i(b_2) > \varphi_i(b_1), \\
b_2 \otimes e_i b_1 & \text{if } \varepsilon_i(b_2) \le \varphi_i(b_1),
\end{cases} \\
f_i(b_2 \otimes b_1) &= \begin{cases}
f_i b_2 \otimes b_1 & \text{if } \varepsilon_i(b_2) \ge \varphi_i(b_1), \\
b_2 \otimes f_i b_1 & \text{if } \varepsilon_i(b_2) < \varphi_i(b_1),
\end{cases} \\ 
\varepsilon_i(b_2 \otimes b_1) &= \max\big( \varepsilon_i(b_1), \varepsilon_i(b_1) + \varepsilon_i(b_2) - \varphi_i(b_1) \bigr), \\
\varphi_i(b_2 \otimes b_1) &= \max\big( \varphi_i(b_2), \varphi_i(b_2) + \varphi_i(b_1) - \varepsilon_i(b_2) \bigr), \\
\wt(b_2 \otimes b_1) &= \wt(b_2) + \wt(b_1).
\end{align*}

\begin{remark}
Our convention for tensor products is opposite to the convention given by Kashiwara in~\cite{K91}. The above tensor product
rules can also be described by the signature rule, see for example~\cite{HK02}.
\end{remark}

Again let $B_1$ and $B_2$ be abstract $U_q(\g)$-crystals. A \emph{crystal morphism} $\psi\colon B_1 \longrightarrow B_2$ is a map 
$B_1 \sqcup \{0\} \longrightarrow B_2 \sqcup \{0\}$ such that
\begin{enumerate}[(i)]
\item \label{mor:1} $\psi(0) = 0$;
\item \label{mor:2} if $b \in B_1$ and $\psi(b) \in B_2$, then $\wt(\psi(b)) = \wt(b)$, $\varepsilon_i(\psi(b)) = \varepsilon_i(b)$, 
and $\varphi_i(\psi(b)) = \varphi_i(b)$;
\item \label{mor:3} for $b, b' \in B_1$, $\psi(b),\psi(b')\in B_2$ and $f_ib=b'$, we have $\psi(f_ib) = f_i\psi(b)$ and $\psi(e_ib') = e_i \psi(b')$ for all $i \in I$.
\end{enumerate}
A morphism $\psi$ is called \emph{strict} if $\psi$ commutes with $e_i$ and $f_i$ for all $i\in I$.  Moreover, a morphism 
$\psi\colon B_1 \longrightarrow B_2$ is called an \emph{embedding} if the induced map $B_1 \sqcup\{0\} \longrightarrow B_2 \sqcup \{0\}$ is injective.

For a dominant integral weight $\lambda$, let $B(\lambda)$ denote the highest weight crystal with highest weight $\lambda$. Let $u_{\lambda}$ 
denote the unique highest weight vector in $B(\lambda)$. Recall that in general we can consider the classical dominant weight 
$\lambda = \sum_{i \in I_0} k_i \clfw_i$ as a partition with $k_i$ columns of height $i$ and width 1 (resp. width 1/2 for spin nodes $i$). 
We draw our diagrams (and hence our tableaux) using French convention. In the Kashiwara--Nakashima (KN) model~\cite{KN94}, the 
elements of $B(\lambda)$ are given by certain tableaux of shape $\lambda$. The crystal structure is determined by the embedding 
$B(\lambda) \hookrightarrow B(\clfw_1)^{\otimes |\lambda|}$, where the inclusion is the reading word given by reading down the columns 
from left to right. For more on this model, see for instance~\cite[Chapter 8]{HK02}.

\subsection{Simple subalgebras}
For later use, specific realizations are given for the simple roots and fundamental weights of the simple
Lie algebras of types $B_n$, $C_n$, and $D_n$. In each case, the sublattice of
$\overline{P}$ given by the weights appearing in tensor products of the
vector representation is identified with $\ZZ^n$. Let
$\{\epsilon_i\mid 1\le i\le n\}$ be the standard basis of $\ZZ^n$.

\subsubsection*{The simple Lie algebra $B_n$}

\begin{equation}
\label{eq:Broots}
\begin{aligned}
\alpha_a & =\epsilon_a-\epsilon_{a+1} & \hspace{90pt} & \text{for $1\le a<n$}\\
\alpha_n & =\epsilon_n & & \\
\clfw_a & =\epsilon_1+\cdots+\epsilon_a & & \text{for $1\le a<n$}\\
\clfw_n & = \frac{1}{2}(\epsilon_1+\cdots+\epsilon_n). & &
\end{aligned}
\end{equation}
$\lambda\in\ZZ^n$ is $B_n$-dominant if and only if
\begin{equation}
\label{eq:Bdom}
\begin{aligned}
\lambda_a-\lambda_{a+1} & \ge 0 & \hspace{120pt} & \text{for $1\le a < n$} \\
\lambda_n&\ge0.&
\end{aligned}
\end{equation}

\subsubsection*{The simple Lie algebra $C_n$}

\begin{equation}
\label{eq:Croots}
\begin{aligned}
\alpha_a&=\epsilon_a-\epsilon_{a+1} & \hspace{110pt} & \text{for $1\le a<n$}\\
\alpha_n&=2\epsilon_n & &\\
\clfw_a&=\epsilon_1+\cdots+\epsilon_a && \text{for $1\le
a\le n$.}
\end{aligned}
\end{equation}
$\lambda\in\ZZ^n$ is $C_n$-dominant if and only if it is $B_n$-dominant
\eqref{eq:Bdom}.

\subsubsection*{The simple Lie algebra $D_n$}

\begin{equation}
\label{eq:Droots}
\begin{aligned}
\alpha_a&=\epsilon_a-\epsilon_{a+1} & \hspace{70pt} &\text{for $1\le a<n$}\\
\alpha_n&=\epsilon_{n-1}+\epsilon_n &&\\
\clfw_a&=\epsilon_1+\cdots+\epsilon_a && \text{for $1\le a\le n-2$}\\
\clfw_{n-1}&=\frac{1}{2}(\epsilon_1+\cdots+\epsilon_{n-1}-\epsilon_n)&&\\
\clfw_n&=\frac{1}{2}(\epsilon_1+\cdots+\epsilon_{n-1}+\epsilon_n)&&
\end{aligned}
\end{equation}
$\lambda\in\ZZ^n$ is $D_n$-dominant if and only if
\begin{equation}
\label{eq:dominant D}
\begin{aligned}
\lambda_a-\lambda_{a+1}&\ge 0 & \hspace{120pt} &\text{for $1\le a<n$}\\
\lambda_{n-1}+\lambda_n & \ge 0. & &
\end{aligned}
\end{equation}

\subsection{Rigged configurations}

Set $\HH_0 = I_0 \times \ZZ_{>0}$. Let $c_i$ and $c_i^{\vee}$ be the Kac and dual Kac labels~\cite[Table Aff1-3]{kac90}, respectively. 
Let $( \cdot \mid \cdot )$ be the invariant bilinear form on $P$, normalized such that
\[
	(\alpha_i \mid \alpha_j) = \frac{c_i^{\vee}}{c_i} A_{ij}\;.
\]
We also define
\begin{equation}
\label{equation.t}
	t_i = \max\left( \frac{c_i}{c_i^{\vee}}, c_0^{\vee} \right), \quad t_i^{\vee} = \max\left( \frac{c_i^{\vee}}{c_i}, c_0 \right).
\end{equation}
Moreover let $(\widetilde{\alpha}_a)_{a \in I_0}$ denote the simple roots of the classical type $\g_0$ except for type $A_{2n}^{(2)}$, where 
it will be of type $B_n$ (as opposed to type $C_n$ and is the subalgebra fixed by the automorphism $\sigma$ of~\cite[Sec.~8.3]{kac90}).

We now define rigged configurations by mostly following~\cite{OSS03}. Consider the multiplicity array
\[
	L = \big(L_i^{(a)} \in \ZZ_{\geq 0} \mid (a,i) \in \HH_0\big)
\]
with only finitely many nonzero entries and a dominant integral weight $\lambda$ of $\g_0$. We call a sequence of partitions 
$\nu = \{ \nu^{(a)} \mid a \in I_0 \}$ an $(L, \lambda)$-\emph{configuration} if
\begin{equation}
\label{eq:LL-config}
	\sum_{(a,i)\in\HH_0} im_i^{(a)} \widetilde{\alpha}_a = \eta \left(\sum_{(a,i)\in\HH_0} i L_i^{(a)} \clfw_a - \lambda \right)\;,
\end{equation}
where $m_i^{(a)}$ is the number of parts of length $i$ in the partition $\nu^{(a)}$. Here $\eta$ is the identity map except 
in type $A_{2n}^{(2)}$, in which case $\eta$ is the $\ZZ$-linear map from the weight lattice of type $C_n$ to the weight lattice of type $B_n$ 
such that
\[
	\eta(\clfw_a^C) = \begin{cases} \clfw_a^B & 1 \leq a < n, \\ 2\clfw_n^B & a = n. \end{cases}
\]
The set of all such $(L,\lambda)$-configurations is denoted by $C(L,\lambda)$. For $\nu \in C(L,\lambda)$, define the \emph{vacancy numbers} 
of $\nu$ as
\begin{equation}\label{eq:vacancy}
p_i^{(a)}(\nu) = p_i^{(a)} = \sum_{j \geq 1} \min(i,j) L_j^{(a)} - \frac{1}{t_a^{\vee}} \sum_{(b,j) \in \HH_0} (\widetilde{\alpha}_a | \widetilde{\alpha}_b) 
\min( t_b \upsilon_a i, t_a \upsilon_b j) m_j^{(b)},
\end{equation}
where
\[
\upsilon_a = \begin{cases}
2 & a = n \text{ and } \g = C_n^{(1)}, \\
\frac{1}{2} & a = n \text{ and } \g = B_n^{(1)}, \\
1 & \text{otherwise.}
\end{cases}
\]
Define $C^*(L,\lambda) = \{ \nu \in C(L,\lambda) \mid p_i^{(a)}(\nu) \geq 0 \text{ for all $(a,i)\in \HH_0$}\}$.

Recall that we can consider a partition as a multiset of positive integers (typically sorted in decreasing order).
A \emph{rigged partition} is a multiset of pairs of integers $(i, x)$ such that $i > 0$ (typically sorted under decreasing
lexicographic order). Each $(i,x)$ is called a \emph{string}, and we call $i$ the \emph{length} (or \emph{size}) of the string and $x$
is the \emph{label} (or \emph{quantum number}) of the string. Finally, a \emph{rigged configuration} is a pair $(\nu, J)$,
where $\nu \in C(L,\lambda)$ and $J = \big( J_i^{(a)} \big)_{(a, i) \in \HH_0}$ with each $J_i^{(a)}$ a multiset of labels of strings of length $i$ in 
$\nu^{(a)}$ and $\max J_i^{(a)} \le p_i^{(a)}$ for all $(a, i) \in \HH_0$ for which $m_i^{(a)}>0$. In particular, the multiset $J_i^{(a)}$ has
$m_i^{(a)}$ elements.
We define the \emph{colabel} (or \emph{coquantum number}) of a string $(i,x)$ to be $p_i^{(a)} - x$. We call a string $(i,x)$ \emph{singular}
if $x = p_i^{(a)}$ (or equivalently, its colabel is 0). For brevity, we will denote the $a$-th part (rigged partition) of $(\nu,J)$
by $(\nu,J)^{(a)}$ (as opposed to $(\nu^{(a)},J^{(a)})$). In type $A_{2n}^{(2)\dagger}$, we must have $x \in \ZZ + \frac{1}{2}$ for all $x \in J_{2j-1}^{(n)}$

\begin{remark}
\label{rem:rc_convention}
We use a slightly different definition of rigged configurations than the one given in~\cite{OSS03}. In particular, $\nu^{(n)}$ in 
type $B_n^{(1)}$ and $C_n^{(1)}$ in our definition is a usual partition as compared to $\frac{1}{2} \nu^{(n)}$ and $2 \nu^{(n)}$ 
of half-width or double-width as in~\cite{OSS03}, respectively. An example of this convention choice can be seen with 
$\upsilon_a$ used in Equation~\eqref{eq:vacancy}, which are the values $P_i^{(a)}$ in~\cite{OSS03}.
We use this convention since it makes the definition of the crystal structure given in Definition~\ref{def:crystal_ops_nsl}
more uniform.
\end{remark}

Define the set of $L$-\emph{highest weight rigged configurations} of dominant weight $\lambda$ as
\begin{equation*}
\begin{split}
	&\hwRC(L;\lambda) := \{ (\nu,J) \mid \nu \in C^*(L,\lambda) \text{ and } 0 \leq x \leq p_i^{(a)}
	\text{ for all $x\in J_i^{(a)}$ and all } (a,i) \in \HH_0\},\\
	&\hwRC(L) := \bigsqcup_\lambda \hwRC(L;\lambda).
\end{split}
\end{equation*}

\begin{ex}
\label{ex:highest_weight_RC}
Consider type $D_5^{(1)}$ with $L_1^{(2)} = L_2^{(1)} = 1$ and all other $L_s^{(r)} = 0$. An example of a rigged configuration 
$(\nu^*, J^*) \in \hwRC(L;2\clfw_1)$ is
\[
{
\begin{array}[t]{r|c|l}
\cline{2-2} 1 &\phantom{|}& 0 \\
 \cline{2-2} 
\end{array}
} 
\quad
\raisebox{13pt}{
\begin{array}[t]{r|c|l}
\cline{2-2} 0 &\phantom{|}& 0 \\
 \cline{2-2}  &\phantom{|}& 0 \\
 \cline{2-2} 
\end{array}
} 
\quad
\raisebox{13pt}{
\begin{array}[t]{r|c|l}
\cline{2-2} 0 &\phantom{|}& 0 \\
 \cline{2-2}  &\phantom{|}& 0 \\
 \cline{2-2} 
\end{array}
} 
\quad
 {
\begin{array}[t]{r|c|l}
\cline{2-2} 0 &\phantom{|}& 0 \\
 \cline{2-2} 
\end{array}
} 
\quad
 {
\begin{array}[t]{r|c|l}
\cline{2-2} 0 &\phantom{|}& 0 \\
 \cline{2-2} 
\end{array}
}
\raisebox{-15pt}{,}
\]
where the vacancy numbers are displayed to the left of a part and the riggings to the right.
\end{ex}

\begin{dfn}[\cite{S06}]
\label{def:rc_crystal_ops}
Let $\g_0$ be a Lie algebra of type $A_n$, $D_n$, or $E_{6,7,8}$ and $L$ a multiplicity array.  
We define the set $\RC(L)$ as the set generated from the highest weight
rigged configurations $\hwRC(L)$ by the application of operators $e_a, f_a$ for $a\in I_0$ as follows.
Fix $a\in I_0$, and let $x$ be the smallest label of $(\nu,J)^{(a)}$.  
\begin{enumerate}
\item Definition of $e_a$: If $x \geq 0$, then set $e_a(\nu,J) = 0$. Otherwise, let $\ell$ be the minimal length of all strings in $(\nu,J)^{(a)}$ which 
have label $x$.  
The rigged configuration $e_a(\nu,J)$ is obtained by replacing the string $(\ell, x)$ with the string $(\ell-1, x+1)$ and by changing all other labels 
so that all colabels remain fixed.
\item Definition of $f_a$: If $x > 0$, then add the string $(1,-1)$ to $(\nu,J)^{(a)}$. Otherwise, let $\ell$ be the maximal length of all strings in 
$(\nu,J)^{(a)}$ which have label $x$. Replace the string $(\ell, x)$ by the string $(\ell+1, x-1)$ and change all other labels so that all colabels 
remain fixed.  If the result is a rigged configuration, then it is $f_a(\nu,J)$. Otherwise $f_a(\nu,J) = 0$.
\end{enumerate}
\end{dfn}

We define the classical weight $\overline{\wt} \colon \RC(L) \to \overline{P}$ by solving Equation~\eqref{eq:LL-config} for $\lambda$. Thus we have
\begin{equation}
\label{eq:weight_function}
\begin{split}
	\overline{\wt}(\nu, J) & = \sum_{(a,i) \in \HH_0} i \bigl( L_i^{(a)} \clfw_a - m_i^{(a)} \eta^{-1}(\widetilde{\alpha}_a) \bigr),
	\\ & = \sum_{(a,i) \in \HH_0} i L_i^{(a)} \clfw_a - \sum_{a \in I_0} \bigl\lvert \nu^{(a)} \bigr\rvert \eta^{-1}(\widetilde{\alpha}_a).
\end{split}
\end{equation}
Note that
\begin{equation}
\label{eq:weight_p}
\langle \alpha^\vee_a, \overline{\wt}(\nu,J) \rangle = k_a p_\infty^{(a)},
\end{equation}
where $k_a=1$ except $k_n=2$ for type $A_{2n}^{(2)\dagger}$. We can extend this to an affine weight $\wt \colon \RC(L) \to P$ by
\[
\wt(\nu, J) = c_0 \Lambda_0 + \overline{\wt}(\nu, J),
\]
where we lift $\clfw_a \to \Lambda_a$ for all $a \in I_0$ and $c_0$ is such that $\inner{c}{\wt(\nu,J)} = 0$ with $c = \sum_{i \in I} c_i^{\vee} \alpha_i^{\vee}$ 
(the canonical central element). That is to say, we make the resulting affine weight level 0.

From Definition~\ref{def:rc_crystal_ops}, we can see that applying $e_a$ for $a \in I_0$ to a highest weight rigged configuration returns $0$. 
So this agrees with the usual notation of a highest weight element of a crystal. It is known that Definition~\ref{def:rc_crystal_ops} and 
Equation~\eqref{eq:weight_function} gives a classical crystal structure on $\RC(L)$ for simply-laced types.

\begin{thm}[{\cite[Thm.\ 3.7]{S06}}]\label{S06-thm}
\label{thm:crystal_structure}
Let $\g_0$ be a Lie algebra of type $A_n$, $D_n$, or $E_{6,7,8}$ and $\lambda$ a dominant weight in $\overline{P}$.  
For $(\nu, J) \in \hwRC(L;\lambda)$, let $X_{(\nu, J)}$ be the graph generated by $(\nu, J)$ and $e_a, f_a$ for $a \in I_0$.  Then $X_{(\nu, J)}$ 
is isomorphic to the crystal graph $B(\lambda)$ as $U_q(\g_0)$-crystals.
\end{thm}

\begin{remark}
In~\cite{S06}, elements of $X_{(\nu,J)}$ were called unrestricted rigged configurations.
\end{remark}

There is a natural statistic on rigged configurations called \emph{cocharge}. We first define cocharge on $C(L, \lambda)$-configurations by
\begin{equation}
\label{eq:cocharge_nu}
	\cc(\nu) := \frac{1}{2} \sum_{\substack{(a,i) \in \HH_0 \\ (b,j) \in \HH_0}} ( \widetilde{\alpha}_a | \widetilde{\alpha}_b ) 
	\min(t_b \upsilon_a i, t_a \upsilon_b j) m_i^{(a)} m_j^{(b)}
\end{equation}
and then on rigged configurations by
\begin{equation}
\label{eq:cocharge}
\cc(\nu, J) := \cc(\nu) + \sum_{(a,i) \in \HH_0} t_a^{\vee} \lvert J_i^{(a)} \rvert,
\end{equation}
where $\lvert J_i^{(a)} \rvert$ is the sum of the entries in $J_i^{(a)}$.

\begin{ex}
Consider $(\nu^*, J^*)$ from Example~\ref{ex:highest_weight_RC}. Let
\[
f_5 f_2 f_3 f_1 f_2 f_1(\nu^*, J^*) = (\nu, J) \in X_{(\nu^*, J^*)},
\]
so
\[
 \raisebox{-6pt}{$(\nu, J) =$} {
\begin{array}[t]{r|c|c|c|l}
\cline{2-4} 0 &\phantom{|}&\phantom{|}&\phantom{|}& 0 \\
 \cline{2-4} 
\end{array}
} 
\quad
\raisebox{13pt}{
\begin{array}[t]{r|c|c|c|l}
 \cline{2-2} 0 &\phantom{|}& \multicolumn{3 }{l}{ 0 } \\
 \cline{2-4} -1 &\phantom{|}&\phantom{|}&\phantom{|}& -1 \\
 \cline{2-4}
\end{array}
} 
\quad
\raisebox{13pt}{
\begin{array}[t]{r|c|c|l}
 \cline{2-2} 0 &\phantom{|}& \multicolumn{2 }{l}{ 0 } \\
 \cline{2-3} 0 &\phantom{|}&\phantom{|}& 0 \\
 \cline{2-3} 
\end{array}
} 
\quad
 {
\begin{array}[t]{r|c|l}
\cline{2-2} 0 &\phantom{|}& 0 \\
 \cline{2-2}
\end{array}
} 
\quad
 {
\begin{array}[t]{r|c|c|l}
\cline{2-3} -1 &\phantom{|}&\phantom{|}& -1 \\
 \cline{2-3} 
\end{array}
}\; \raisebox{-12pt}{.}
\]
Then we have
\begingroup
\addtolength{\jot}{1em}
\begin{align*}
\raisebox{-6pt}{$e_2(\nu, J) = $} \hspace{10pt} &
{
\begin{array}[t]{r|c|c|c|l}
\cline{2-4} -1 &\phantom{|}&\phantom{|}&\phantom{|}& -1 \\
 \cline{2-4} 
\end{array}
} 
\quad
\raisebox{13pt}{
\begin{array}[t]{r|c|c|l}
 \cline{2-2} 0 &\phantom{|}& \multicolumn{2}{l}{ 0 } \\
 \cline{2-3} 0 &\phantom{|}&\phantom{|}& 0 \\
 \cline{2-3}
\end{array}
} 
\quad
\raisebox{13pt}{
\begin{array}[t]{r|c|c|l}
 \cline{2-2} 0 &\phantom{|}& \multicolumn{2 }{l}{ 0 } \\
 \cline{2-3} 0 &\phantom{|}&\phantom{|}& 0 \\
 \cline{2-3} 
\end{array}
} 
\quad
 {
\begin{array}[t]{r|c|l}
\cline{2-2} 0 &\phantom{|}& 0 \\
 \cline{2-2} 
\end{array}
} 
\quad
 {
\begin{array}[t]{r|c|c|l}
\cline{2-3} -1 &\phantom{|}&\phantom{|}& -1 \\
 \cline{2-3} 
\end{array}
}
\raisebox{-12pt}{,}
\\ \raisebox{-6pt}{$f_3(\nu, J) = $} \hspace{10pt} &
{
\begin{array}[t]{r|c|c|c|l}
\cline{2-4} 0 &\phantom{|}&\phantom{|}&\phantom{|}& 0 \\
 \cline{2-4} 
\end{array}
} 
\quad
\raisebox{13pt}{
\begin{array}[t]{r|c|c|c|l}
 \cline{2-2} 0 &\phantom{|}& \multicolumn{3 }{l}{ 0 } \\
 \cline{2-4} 0 &\phantom{|}&\phantom{|}&\phantom{|}& 0 \\
 \cline{2-4}
\end{array}
} 
\quad
\raisebox{13pt}{
\begin{array}[t]{r|c|c|c|l}
 \cline{2-2} 0 &\phantom{|}& \multicolumn{3 }{l}{ 0 } \\
 \cline{2-4} -1 &\phantom{|}&\phantom{|}&\phantom{|}& -1 \\
 \cline{2-4} 
\end{array}
} 
\quad
 {
\begin{array}[t]{r|c|l}
\cline{2-2} 0 &\phantom{|}& 0 \\
 \cline{2-2} 
\end{array}
} 
\quad
 {
\begin{array}[t]{r|c|c|l}
\cline{2-3} -1 &\phantom{|}&\phantom{|}& -1 \\
 \cline{2-3} 
\end{array}
}\;\raisebox{-12pt}{.}
\end{align*}
\endgroup
Next we look at $f_2(\nu, J)$, and after adding a box, we obtain
\[
{
\begin{array}[t]{r|c|c|c|l}
\cline{2-4} 0 &\phantom{|}&\phantom{|}&\phantom{|}& 0 \\
 \cline{2-4} 
\end{array}
} 
\quad
\raisebox{13pt}{
\begin{array}[t]{r|c|c|c|c|l}
 \cline{2-2} 0 &\phantom{|}& \multicolumn{4 }{l}{ 0 } \\
 \cline{2-5} -3 &\phantom{|}&\phantom{|}&\phantom{|}&\phantom{|}& -2 \\
 \cline{2-5}
\end{array}
} 
\quad
\raisebox{13pt}{
\begin{array}[t]{r|c|c|l}
 \cline{2-2} 0 &\phantom{|}& \multicolumn{2 }{l}{ 0 } \\
 \cline{2-3} 0 &\phantom{|}&\phantom{|}& 0 \\
 \cline{2-3} 
\end{array}
} 
\quad
 {
\begin{array}[t]{r|c|l}
\cline{2-2} 0 &\phantom{|}& 0 \\
 \cline{2-2} 
\end{array}
} 
\quad
 {
\begin{array}[t]{r|c|c|l}
\cline{2-3} -1 &\phantom{|}&\phantom{|}& -1 \\
 \cline{2-3} 
\end{array}
}\; \raisebox{-12pt}{,}
\]
and since $p_4^{(2)} = -3 < -2 = \max J_4^{(2)}$, we have $f_2(\nu, J) = 0$. Additionally we have
\begin{gather*}
\arraycolsep=1.4pt\def\arraystretch{1.15}
\begin{array}{rlrl}
\overline{\wt}(\nu, J) & = -\clfw_2 + \clfw_3 + \clfw_4 - \clfw_5, & \wt(\nu, J) & = -\Lambda_2 + \Lambda_3 + \Lambda_4 - \Lambda_5
\\ \overline{\wt}\bigl( e_2(\nu, J) \bigr) & = -\clfw_1 + \clfw_2 + \clfw_4 - \clfw_5, & \hspace{20pt} \wt\bigl( e_2(\nu, J) \bigr) & = - \Lambda_0 - \Lambda_1 + \Lambda_2 + \Lambda_4 - \Lambda_5
\\ \overline{\wt}\bigl( f_3(\nu, J) \bigr) & = - \clfw_3 + 2\clfw_4, & \wt\bigl( f_3(\nu, J) \bigr) & = -\Lambda_3 + 2 \Lambda_4
\end{array}
\\ \cc(\nu, J) = \cc(\nu^*, J^*) = 1.
\end{gather*}
\end{ex}

\subsection{Kirillov-Reshetikhin crystals}

Let $\g$ be an affine Kac--Moody algebra, $\g^{\prime} = [\g, \g]$ the derived subalgebra of $\g$ and $U_q^{\prime}(\g) := U_q(\g^{\prime})$ 
the associated quantum group. We consider a particular class of finite-dimensional irreducible representations called \emph{Kirillov-Reshetikhin 
(KR) modules} which are indexed by $(r, s) \in \HH_0$ and denoted by $W^{(r)}_s$. It was shown in~\cite{OS08} that in all non-exceptional types 
KR modules have crystal bases, which were described combinatorially in~\cite{FOS09}. We call these crystals \emph{Kirillov-Reshetikhin (KR) 
crystals} and denote them by $B^{r,s}$. As classical crystals, they decompose as
\[
B^{r,s} \iso B(s \clfw_r) \oplus \bigoplus_{\lambda} B(\lambda).
\]
Explicitly we have the following classical decompositions:
\begin{itemize}
\item In type $A_n^{(1)}$, we have $B^{r,s} \iso B(s \clfw_r)$ for all $r\in I_0$.
\item In type $B_n^{(1)}$, we obtain $\lambda$ by removing vertical dominoes from an $r \times s$ rectangle for $r<n$ or an $n \times (s/2)$ 
rectangle for $r=n$.
\item In type $C_n^{(1)}$, we obtain $\lambda$ by removing horizontal dominoes from an $r \times s$ rectangle for $r < n$ and we have 
$B^{r,s} \iso B(s \clfw_r)$ for $r = n$.
\item In type $D_n^{(1)}$, we obtain $\lambda$ by removing vertical dominoes from an $r \times s$ rectangle for $r < n-1$ and we have 
$B^{r,s} \iso B(s \clfw_r)$ for $r = n-1,n$.
\item In type $A_{2n-1}^{(2)}$, we obtain $\lambda$ by removing vertical dominoes from an $r \times s$ rectangle for all $r \le n$.
\item In type $A_{2n}^{(2)}$, we obtain $\lambda$ by removing boxes from an $r \times s$ rectangle for all $r\le n$.
\item In type $D_{n+1}^{(2)}$, we obtain $\lambda$ by removing boxes from an $r \times s$ rectangle for $r < n$ and we have 
$B^{r,s} \iso B(s \clfw_r)$ for $r = n$.
\item In type $A_{2n}^{(2)\dagger}$, we obtain $\lambda$ by removing horizontal dominoes from an $r \times s$ rectangle for all $r \leq n$.
\end{itemize}
We note that the decomposition for general $r$ depends only on how the affine node attaches to the classical diagram. We let $\diamond$ 
denote the type of boxes removed in each decomposition (that is a single box, a vertical domino, or a horizontal domino).

\begin{dfn}
\label{def:energy}
Consider a KR crystal $B^{r,s}$. There exists a statistic called \emph{energy} $D \colon B^{r,s} \to \ZZ$, which on the classical 
component $B(\lambda)$ is equal to the number of $\diamond$ which have been removed from the $r \times s$ rectangle in order to obtain 
$\lambda$~\cite{HKOTY99}. Thus the energy is constant on all classical components. If $B^{r,s} \iso B(s \Lambda_r)$ (as classical crystals), 
then $D(v) = 0$ for all $v \in B^{r,s}$.
\end{dfn}
Definition~\ref{def:energy} can be extended to arbitrary tensor factors; see for example~\cite[Section 4]{ST12}.

Next we consider a tensor product of KR crystals $B = \bigotimes_{i=1}^N B^{r_i,s_i}$. We define a multiplicity array $L$ from $B$ by $L^{(r)}_s$ 
as the number of factors $B^{r,s}$ occurring in $B$. Alternative to our notation $\RC(L)$, we also use the notation $\RC(B)$, where $L$ is the 
multiplicity array associated to $B$, in order to signify the ordering of the factors. In~\cite{OSS03}, it was shown that for 
$B = \bigotimes_{i=1}^N B^{1,1}$ there exists a bijection $\Phi \colon \RC(B) \to B$ for all non-exceptional affine types. The bijection $\Phi$ is 
formed by repeatedly applying a map
\[
	\delta \colon \RC(B^{1,1} \otimes B^*) \to \RC(B^*) \times B^{1,1},
\]
where $B^*$ is some tensor product of KR crystals (in~\cite{OSS03}, $B^* = \bigotimes_{i=1}^{N-1} B^{1,1}$).
The map $\delta$ generally is given by traversing the crystal $B^{1,1}$ from classically highest weight to classically lowest weight, and for 
every crystal edge labelled by $a\in I_0$, the smallest singular string from $(\nu, J)^{(a)}$ of length bigger or equal to the previously
selected singular string is removed, if possible.
If it is not possible, the process stops and determines the element in $B^{1,1}$. We say $\delta$ returns the element of $B^{1,1}$.
For brevity, we refer to~\cite{OSS03} for an explicit description of $\delta$.

\begin{remark}
\label{rem:rc_convention_delta}
Since we are using a different convention for rigged configurations than~\cite{OSS03} (see Remark~\ref{rem:rc_convention}), we must make appropriate modifications here to the map $\delta$.
\end{remark}

In addition to the map $\delta$, we require the following maps for defining $\Phi$ on arbitrary tensor factors
\begin{align*}
\ls \colon & \RC(B^{r,s} \otimes B^*) \hookrightarrow \RC(B^{r,1} \otimes B^{r,s-1} \otimes B^*) \qquad \text{if $s>1$,}
\\ \lt \colon & \RC(B^{r,1} \otimes B^*) \hookrightarrow \RC(B^{1,1} \otimes B^{r-1,1} \otimes B^*) \qquad \text{if $r>1$.}
\end{align*}
The map $\ls$ on a rigged configuration is the identity (perhaps with larger vacancy numbers; so it is well-defined). The map $\lt$ adds a length 
1 singular string to $(\nu, J)^{(a)}$ for $1\le a < r$. A straightforward computation shows that this preserves the vacancy numbers and hence is well-defined.

When the left factor is a spinor column $B^{r,1}$, which happens in type $B_n^{(1)}$ when $r=n$ and $D_n^{(1)}$ when $r=n-1$ or $n$, 
the application of $\lt$ needs to be modified. We must perform a ``doubling map'' before applying $\lt$, and then a ``halving map'' once we
have completed the column. The doubling map is generally given by
\begin{equation}
\label{eq:doubling_map}
\begin{aligned}
\widetilde{m}_{2i}^{(a)} & = m_i^{(a)}, \\
\widetilde{J}_{2i}^{(a)} & = 2 J_i^{(a)},
\end{aligned}
\end{equation}
and the halving map is the inverse.

For $B^{n,s}$ in type $B_n^{(1)}$, following~\cite[Lemma~4.2]{FOS09} our doubling map also consists of embedding this into type $A_{2n-1}^{(2)}$ 
with $\widetilde{\nu}^{(r)} = 2 \nu^{(r)}$ and $\widetilde{L}_{2s}^{(r)} = L_s^{(r)}$ for $r < n$, with $\widetilde{\nu}^{(n)} = \nu^{(n)}$ and 
$\widetilde{L}_s^{(n)} = L_s^{(n)}$, and does not change the labels (our convention choice for rigged configurations can be seen here). 
We then perform the usual $A_{2n-1}^{(2)}$ bijection algorithm on the leftmost factor, followed by the halving map.

For $B^{r,s}$ with $r = n-1,n$ in type $D_n^{(1)}$, the doubling map on the rigged configuration is given by Equation~\eqref{eq:doubling_map} with 
$\widetilde{L}_{2s}^{(r)} = L_{s}^{(r)}$. We perform the doubling map after performing $\ls$. Next we must apply
\[
\delta^{(r)} \colon \RC(B^{r,2} \otimes B^*) \hookrightarrow \RC(B^{n,1} \otimes B^{n-1,1} \otimes B^*),
\]
which is given by the usual algorithm for $\delta$ but starting with $\nu^{(r)}$. We then apply
\[
\widetilde{\delta}^{(r)} \colon \RC(B^{n,1} \otimes B^{n-1,1} \otimes B^*) \hookrightarrow \RC(B^{n-2,1} \otimes B^*),
\]
which is given by the usual algorithm for $\delta$ but starting with $\nu^{(n-2)}$. We then proceed with $\lt$ and $\delta$ as we 
normally would until we finish the column. After this, we perform the halving map. For an alternative description of the map $\Phi$ 
for type $D_n^{(1)}$ spinors, see~\cite{S05}.

For simplicity, we consider $\delta^{\prime} := \delta \circ \lt$, and it is straightforward to show that this is equivalent to 
beginning at $\nu^{(r)}$ (instead of $\nu^{(1)}$) and following the usual procedure of $\delta$. From now on, if there is no 
cause for confusion, we will write $\delta$ for $\delta^{\prime}$ in the remainder of the paper. One of the main results in this paper 
will be the definition of the analogues of $\lt$ and $\ls$ on $B$ itself. Then, defining a 
map $\Phi \colon \RC(B) \to B$ that commutes with $\ls$ and $\delta'$, we have the following conjecture.
\begin{conj}
\label{conj:bijection}
Let $\g$ be an affine Kac--Moody algebra and $B = \bigotimes_{i=1}^N B^{r_i,s_i}$ a tensor product
of KR crystals of type $\g$. The map $\Phi \colon \RC(B) \to B$ is a bijection. In addition, $\Phi \circ \theta$ sends cocharge to energy,
where $\theta$ maps each rigging $x$ to its colabel.
\end{conj}

Note that restricting $\Phi$ to classically highest weight elements implies the $X = M$ 
conjecture of~\cite{HKOTY99, HKOTT02}.

Conjecture~\ref{conj:bijection} is known on highest weight elements in the following cases.
\begin{itemize}
\item For type $A_n^{(1)}$~\cite{KSS02}.
\item For $B = \bigotimes_{i=1}^N B^{1,s_i}$ in all non-exceptional types~\cite{OSS03,SS2006}.
\item For $B = \bigotimes_{i=1}^N B^{r_i,1}$ in types $D_{n+1}^{(2)}$, $A_{2n}^{(2)}$ and $C_n^{(1)}$~\cite{OSS03III}
and type $D_n^{(1)}$~\cite{S05}.
\item For $\bigotimes_{i=1}^N B^{1,1}$ in type $E_6^{(1)}$~\cite{OS12}.
\end{itemize}

Conjecture~\ref{conj:bijection} was also verified by computer for tensor products for non-exceptional types up to rank 4, up to 2 factors of 
the same level, and $s \leq 2$. 
See~\cite{ScrimshawThesis} for some sample code.

We also have the following refinement of Conjecture~\ref{conj:bijection}.

\begin{conj}
\label{conj:crystal_isomorphism}
Let $\g$ be an affine Kac--Moody algebra and $B = \bigotimes_{i=1}^N B^{r_i, s_i}$ a tensor product of KR crystals of type $\g$. The map 
$\Phi \colon \RC(B) \to B$ is an affine crystal isomorphism.
\end{conj}

Given that $\Phi$ is a classical crystal isomorphism, we can in principle extend $\Phi$ to an affine crystal isomorphism. Since $\Phi$ preserves 
the weights, it suffices to show that $\Phi$ is a bijection that commutes with the classical crystal operators for Conjecture~\ref{conj:crystal_isomorphism}.
For type $A_n^{(1)}$, it was shown in~\cite{KSS02} that $\Phi$ is a bijection and in~\cite{DS06} that it intertwines with the crystal operators.
For type $D_n^{(1)}$ commutativity with the crystal operators was shown in~\cite{Sakamoto13}, however currently $\Phi$ is only known to
be a bijection for single columns and single rows~\cite{SS2006,S05}.
Conjecture~\ref{conj:crystal_isomorphism} has been verified by computer for non-exceptional affine types up to rank 4, up to 2 factors, and $s \leq 2$.

The combinatorial $R$-matrix is the affine crystal isomorphism $R \colon B^{r,s} \otimes B^{r',s'} \to B^{r',s'} \otimes B^{r,s}$
mapping $u_{r,s} \otimes u_{r',s'}$ to $u_{r',s'} \otimes u_{r,s}$, where $u_{r,s}$ is the unique element in $B^{r,s}$ of classical weight
$s\overline{\Lambda}_r$.
In general, it is hard to give an explicit combinatorial description of this map. Since $R$ is conjectured to be the identity on rigged
configurations (proven in certain cases), the bijection $\Phi$ would give an explicit way to obtain the combinatorial $R$-matrix.

\subsection{Virtual crystals}
\label{subsection.virtual crystals}

We now recall the notation of virtual crystals~\cite{OSS03III,OSS03II}. Let $\g$ be any non-simply-laced affine Kac-Moody algebra, and
consider the well-known natural embeddings of algebras~\cite{JM85}:
\begin{equation}
\label{eqn:affine_embeddings}
\begin{aligned}
C_n^{(1)}, A_{2n}^{(2)}, A_{2n}^{(2)\dagger}, D_{n+1}^{(2)} & \lhook\joinrel\longrightarrow A_{2n-1}^{(1)}
\\ B_n^{(1)}, A_{2n-1}^{(2)} & \lhook\joinrel\longrightarrow D_{n+1}^{(1)}
\\ E_6^{(2)}, F_4^{(1)} & \lhook\joinrel\longrightarrow E_6^{(1)}
\\ G_2^{(1)}, D_4^{(3)} & \lhook\joinrel\longrightarrow D_4^{(1)}.
\end{aligned}
\end{equation}
Let $\virtual{\g}$ denote the simply-laced type under the image with index set $\virtual{I}$. Let $\Gamma$ and $\virtual{\Gamma}$ be the 
Dynkin diagrams of $\g$ and $\widehat{\g}$, respectively. These embeddings arise from the diagram foldings 
$\phi \colon \virtual{\Gamma} \searrow \Gamma$ which fix the affine node. By abuse of notation, we will also denote the corresponding 
map on the index sets by $\phi$. In addition, we require \emph{scaling factors} $\gamma = (\gamma_a)_{a \in I}$ defined in the following way:
\begin{enumerate}
\item Suppose $\Gamma$ has a unique arrow.
\begin{enumerate}
  \item Suppose the arrow points towards the component of the special node $0$. Then $\gamma_a = 1$ for all $a \in I$.
  \item Otherwise, $\gamma_a$ is the order of $\phi$ for all $a$ in the component of $0$ after removing the arrow and $\gamma_a = 1$
  in all other components.
\end{enumerate}
\item Otherwise $\Gamma$ has 2 arrows and embeds in $A_{2n-1}^{(1)}$. Then $\gamma_a = 1$ for all $1 \leq a \leq n-1$, and for 
$a \in \{0, n\}$, we have $\gamma_a = 2$ if the arrow points away from $a$ and $\gamma_a = 1$ otherwise.
\end{enumerate}

\begin{remark}
\label{rem:phi_gamma}
Note that for the first set of embeddings in~\eqref{eqn:affine_embeddings}, we have $\lvert \phi^{-1}(a) \rvert=2$, $\gamma_a=1$
for $a\neq 0,n$ and $\lvert \phi^{-1}(a) \rvert=1$, $\gamma_a \in \{1,2\}$ for $a=0,n$. For the second set of embeddings in~\eqref{eqn:affine_embeddings}, 
we have $\lvert \phi^{-1}(a) \rvert=1$ for $a\neq n$, $\lvert \phi^{-1}(n) \rvert=2$, and $\gamma_a=1$ except for $\gamma_a=2$ for $0\le a<n$ and type $B_n^{(1)}$.

Overall (including the exceptional cases) note that all orbits under $\phi$ have either 1 or the order of $\phi$ elements. Also for any fixed $a \in I$, we cannot have
$\gamma_a \neq 1$ and $\lvert \phi^{-1}(a) \rvert \neq 1$ simultaneously. In addition, if $\gamma_a \neq 1$, then $\gamma_a$ 
equals the order of $\phi$.
\end{remark}

The above embeddings of algebras yield natural embeddings $\Psi \colon P \longrightarrow \virtual{P}$ of weight lattices as 
\begin{align*}
\Lambda_a & \mapsto \gamma_a \sum_{b \in \phi^{-1}(a)} \virtual{\Lambda}_b,
\\ \alpha_a & \mapsto \gamma_a \sum_{b \in \phi^{-1}(a)} \virtual{\alpha}_b.
\end{align*}
This implies that $\Psi(\delta) = c_0 \gamma_0 \virtual{\delta}$, where $c_0$ is from Table Aff~2 in~\cite{kac90} (denoted by $a_0$) and $\delta$ 
(resp.\ $\virtual{\delta}$) is the minimal positive imaginary root in $P$ (resp.\ $\virtual{P}$).

\begin{dfn}
\label{dfn:virtual_crystal}
Let $\virtual{B}$ be a $U_q^{\prime}(\virtual{\g})$-crystal and $V \subseteq \virtual{B}$. Let $\phi$ and 
$(\gamma_a)_{a \in I}$ be the folding and the scaling factors given above. The \emph{virtual crystal operators} (of type $\g$) are defined as
\begin{align*}
	e^v_a & := \prod_{b \in \phi^{-1}(a)} \virtual{e}_b^{\;\gamma_a},\\ 
	f^v_a & := \prod_{b \in \phi^{-1}(a)} \virtual{f}_b^{\;\gamma_a}.
\end{align*}
A \emph{virtual crystal} is the pair $(V, \virtual{B})$ such that $V$ has a $U_q^{\prime}(\g)$-crystal structure defined by
\begin{equation}
\label{equation.crystal virtual}
\def\arraystretch{1.3}
\begin{array}{c}
e_a := e^v_a \hspace{40pt} f_a := f^v_a \\
\varepsilon_a := \gamma_a^{-1} \virtual{\varepsilon}_b \hspace{40pt} \varphi_a := \gamma_a^{-1} \virtual{\varphi}_b \\
\wt := \Psi^{-1} \circ \virtual{\wt}
\end{array}
\end{equation}
for any $b \in \phi^{-1}(a)$.
\end{dfn}

\begin{remark}
The order in which the operators in $e^v_a$ and $f^v_a$ are applied does not matter since $b, b^{\prime}$ are not connected 
in $\virtual{\Gamma}$ for all $b \neq b^{\prime} \in \phi^{-1}(a)$. Also the fact that $\varepsilon_a$ and $\varphi_a$ as defined
in~\eqref{equation.crystal virtual} are indeed the string lengths of the $U'_q(\g)$-crystal is a property called ``aligned'' in~\cite{OSS03III,OSS03II}.
\end{remark}

We say $B$ \emph{virtualizes} in $\virtual{B}$ if there exists a $U_q^{\prime}(\g)$-crystal isomorphism $v \colon B \to V$ for some 
virtual crystal $(V, \virtual{B})$. The resulting isomorphism is called the \emph{virtualization map}.

In subsequent sections, we will denote an object $S$ associated with $\g$ by $\virtual{S}$ for the corresponding object in $\virtual{\g}$.

We modify Definition~\ref{dfn:virtual_crystal} for classical types by using $U_q(\g_0)$ in place of $U_q^{\prime}(\g)$ and 
restricting Equation~\eqref{eqn:affine_embeddings} to the corresponding classical types:
\begin{equation}
\label{eq:classical_embeddings}
\begin{split}
C_n, C_n, B_n, B_n & \lhook\joinrel\longrightarrow A_{2n-1}
\\ B_n, C_n & \lhook\joinrel\longrightarrow D_{n+1}
\\ F_4, F_4 & \lhook\joinrel\longrightarrow E_6
\\ G_2, G_2 & \lhook\joinrel\longrightarrow D_4.
\end{split}
\end{equation}
Note that for the exceptional types, there are different scaling factors $(\gamma_a)_{a \in I_0}$ in each of the embeddings above.
For the non-exceptional types, the classical embedding are the same.

The following result is due to Baker~\cite{baker2000} when $\virtual{\g}_0$ is of type $A_{2n-1}$ and we show the other cases 
in Appendix~\ref{appendix:extension}.
\begin{thm}
\label{thm:virtual_highest_weight}
Let $\g_0$ be of classical type. The highest weight crystal $B(\lambda)$ virtualizes in $B(\Psi(\lambda))$ with the virtualization 
map $v$ given by $v(u_{\lambda}) \mapsto u_{\Psi(\lambda)}$ and extended by $f_a \mapsto f^v_a$ (recall $u_{\lambda}$ 
is the unique highest weight element in $B(\lambda)$).
\end{thm}
We note that Baker's result is for the restriction of $C_n^{(1)}, D_{n+1}^{(2)} \lhook\joinrel\longrightarrow A_{2n-1}^{(1)}$, but the other cases of 
$A_{2n}^{(2)}, A_{2n}^{(2)\dagger} \lhook\joinrel\longrightarrow A_{2n-1}^{(1)}$ considered in~\cite{OSS03III} gives the same classical virtualization.

It is clear that if $(V_1, \virtual{B}_1)$ and $(V_2, \virtual{B}_2)$ are virtual crystals, then $(V_1 \oplus V_2, \virtual{B}_1 \oplus \virtual{B}_2)$ 
is a virtual crystal. Moreover, virtual crystals are closed under taking tensor products.
\begin{prop}[{\cite[Prop.~6.4]{OSS03III}}]
\label{prop:virtual_tensor}
Virtual crystals form a tensor category.
\end{prop}
We note that the proof also holds for classical types since it is a statement about the tensor product rule.

Next we restate a conjecture given in~\cite[Conj. 3.7]{OSS03II}.
\begin{conj}
\label{conj:KR_virtualization}
The KR crystal $B^{a,s}$ virtualizes into
\[
\virtual{B}^{a,s} = \begin{cases}
B^{n,s} \otimes B^{n,s} & \text{if }\g = A_{2n}^{(2)}, A_{2n}^{(2)\dagger} \text{ and } a = n, \\
\bigotimes_{b \in \phi^{-1}(a)} B^{b,\gamma_a s} & \text{otherwise.}
\end{cases}
\]
\end{conj}
This conjecture is known for $B^{r,1}$ in types $D_{n+1}^{(2)}$, $A_{2n}^{(2)}$ and $C_n^{(1)}$~\cite{OSS03III} and $B^{1,s}$ for all 
non-exceptional types~\cite{OSS03II}.

\subsection{The (virtual) Kleber algorithm}

Next we recall the \emph{Kleber algorithm}~\cite{Kleber98}. Let $\g$ be an affine type whose canonical classical subalgebra is of simply-laced type.
\begin{dfn}[Kleber algorithm]
\label{def:kleber_algorithm}
Let $B$ be a tensor product of KR crystals of type $\g$ with multiplicity array $L$.
We construct the \emph{Kleber tree} $T(B)$ whose nodes will be labelled by weights in $\overline{P}^+$ and edges are labelled by 
$d_{xy} = x - y \in \overline{Q}^+ \setminus \{0\}$ recursively starting with $T_0$ consisting of a single node of weight $0$.
\begin{itemize}
\item[(K1)] Let $T^{\prime}_{\ell}$ be obtained from $T_{\ell-1}$ by adding $\sum_{a=1}^n \clfw_a \sum_{i \geq \ell} L_i^{(a)}$ to the weight of each node.
\item[(K2)] Construct $T_{\ell}$ from $T^{\prime}_{\ell}$ as follows. Let $x$ be a node at depth $\ell - 1$. Suppose there is a weight 
$y \in \overline{P}^+$ such that $x - y \in \overline{Q}^+ \setminus \{0\}$. If $x$ is not the root, then let $w$ be the parent of $x$. Then 
$(w - x)$ is larger than $(x - y)$ component-wise expressed as a sum of the simple roots $\alpha_i$ (equivalently we have 
$d_{wx} - d_{xy} \in \overline{Q}^+ \setminus \{0\}$). For all such $y$, attach $y$ as a child of $x$.
\item[(K3)] If $T_{\ell} \neq T_{\ell-1}$, then repeat from~(K1); otherwise terminate and return $T(B) = T_{\ell}$.
\end{itemize}
Now we convert the tree to highest weight rigged configurations as follows. Let $x$ be a node at depth $p$ in $T(B)$, and 
$x^{(0)}, x^{(1)}, \dotsc, x^{(p)} = x$ be the weights of nodes on the path from the root of $T(B)$ to $x$. The resulting configuration $\nu$ is given by
\[
m_i^{(a)} = (x^{(i-1)} - 2 x^{(i)} + x^{(i+1)} \mid \clfw_a)
\]
where we make the convention that $x = x^{(k)}$ for all $k > p$. In other words, there are $j$ rows of length $i$ in $\nu^{(a)}$ where $j$ is the 
coefficient of $\alpha_a$ in the difference of the corresponding edge labels. We then take the riggings over all possible values between $0$ 
and $p_i^{(a)}$.
\end{dfn}

For non-simply-laced types, we modify the algorithm by using virtual rigged configurations. The resulting algorithm is known as the 
\emph{virtual Kleber algorithm}~\cite{OSS03II}. 

\begin{dfn}[Virtual Kleber algorithm]
\label{def:virtual_kleber}
The virtual Kleber tree is defined from the Kleber tree of $\virtual{B}$ in the ambient type, but we only add a child in step~(K2) if the 
following conditions are satisfied:
\begin{itemize}
\item[(V1)] $(y \mid \virtual{\alpha}_b) = (y \mid \virtual{\alpha}_{b'})$ for all $b, b' \in \phi^{-1}(a)$.
\item[(V2)] If $\ell - 1 \notin \gamma_a \ZZ$, then for $w$ the parent of $x$, the $a$-th component of $d_{wx}$ and $d_{xy}$ must be equal.
\end{itemize}
Let $\virtual{T}(B)$ be the resulting tree, which we will call the \emph{ambient tree}, and let $\gamma = \max_{a \in I} \gamma_a$. We now 
select nodes which satisfy either:
\begin{itemize}
\item[(A1)] $y$ is at depth $\ell \in \gamma \ZZ$, or
\item[(A2)] $(d_{xy} \mid \virtual{\clfw}_a) = 0$ for every $a$ such that $1 < \gamma = \gamma_a$, where $x$ is the parent of $y$.
\end{itemize}
To construct the rigged configurations from the selected nodes, we take the devirtualization of the resulting virtual rigged configurations (with the 
appropriate riggings) obtained from the usual Kleber algorithm.
\end{dfn}

\section{Crystal operators on rigged configurations in non-simply-laced types}
\label{sec:virtual_rc}

In this section, we show using Theorem~\ref{thm:virtual_highest_weight} that there exists a classical crystal structure on rigged configurations 
given by Definition~\ref{def:rc_crystal_ops} for non-simply-laced finite types.

\subsection{Virtualization map}

We define the virtualization map on rigged configurations as in~\cite{OSS03III,OSS03II,S06II,SS2006}.
In order to do so, we must make a modification to the scaling factors $(\gamma_a)_{a \in I}$ by
\[
\widetilde{\gamma}_a = \begin{cases}
\gamma_a & \text{if $a \neq n$ or } \g \neq A_{2n}^{(2)}, A_{2n}^{(2)\dagger}, \\
1 & \text{if $a = n$ and } \g = A_{2n}^{(2)}, \\
2 & \text{if $a = n$ and } \g = A_{2n}^{(2)\dagger}.
\end{cases}
\]
The reason for this modification is due to the fact that we use the virtual embedding $B^{n,s} \lhook\joinrel\longrightarrow B^{n,s} \otimes B^{n,s}$ 
of the type $A_{2n}^{(2)}$ or $A_{2n}^{(2)\dagger}$ KR crystal into type $A_{2n-1}^{(1)}$, rather than $B^{n,\gamma_n s}$ of type $A_{2n-1}^{(1)}$.
In light of Conjecture~\ref{conj:KR_virtualization}, we define
\begin{equation*}
\begin{aligned}
\virtual{L}_{\gamma_a i}^{(b)} & = L_i^{(a)} && \text{for $b\in \phi^{-1}(a)$,}\\
\virtual{L}_{j}^{(b)} & = 0 && \text{if $j\not \in \gamma_a \mathbb{Z}$,}
\end{aligned}
\end{equation*}
except for $ \g = A_{2n}^{(2)}, A_{2n}^{(2)\dagger}$ and $a=n$ in which case $\virtual{L}_i^{(n)} = 2L_i^{(n)}$.

The virtualization map $v$ from rigged configurations of type $\g \neq A_{2n}^{(2)\dagger}$ to rigged configurations of type $\virtual{\g}$ is given 
by~\cite[Thm.~4.2]{OSS03II}
\begin{equation}
\label{eq:rc_virtualization_map}
\begin{aligned}
\virtual{m}_{\widetilde{\gamma}_a i}^{(b)} & = m_i^{(a)},
\\ \virtual{J}_{\widetilde{\gamma}_a i}^{(b)} & = \gamma_a J_i^{(a)},
\end{aligned}
\end{equation}
for all $b \in \phi^{-1}(a)$. For $\g = A_{2n}^{(2)\dagger}$, we use the virtualization map~\cite[Def.~7.1/Thm.~7.2]{SS2006}
\begin{equation}
\label{eq:rc_virtualization_map_A2dual}
\begin{aligned}
\virtual{m}_i^{(b)} & = m_i^{(a)},
\\ \virtual{J}_i^{(b)} & = \widetilde{\gamma}_a J_i^{(a)}.
\end{aligned}
\end{equation}
We note that this is the same as Equation~\eqref{eq:rc_virtualization_map} except for $a = n$.
Let $(\virtual{\nu}, \virtual{J})$ denote the resulting virtual rigged configuration. Under the virtualization map, we have~\cite{OSS03II}
\begin{align}
\label{eq:vacancy_virtualization}
\virtual{p}_{\widetilde{\gamma}_a i}^{(b)} & = \gamma_a p_i^{(a)},
\\ \label{eq:cocharge_virtualization}
\cc(\virtual{\nu}, \virtual{J}) & = \gamma_0 \cc(\nu, J),
\end{align}
for all $b \in \phi^{-1}(a)$, except for $\g = A_{2n}^{(2)\dagger}$ and $a = n$, where we have
\[
\virtual{p}_i^{(n)} = \widetilde{\gamma}_n p_i^{(n)}.
\]

\subsection{Crystal operators}

We begin by giving an explicit description of $e_a$ and $f_a$ for $a \in I_0$ for non-simply-laced rigged configurations.

\begin{dfn}
\label{def:crystal_ops_nsl} 
Define the crystal operators $f_a$ and $e_a$ for all types except $A_{2n}^{(2)}$ and $A_{2n}^{(2)\dagger}$ with $a = n$ as in
Definition~\ref{def:rc_crystal_ops}.
For $A_{2n}^{(2)\dagger}$, the algorithm is modified for $e_n$ by adding $1/2$ 
to the new label and for $f_n$ by removing $1/2$ from the new label or the added length 1 string is given a label of $-1/2$.
For $A_{2n}^{(2)}$, the algorithm for $e_n$ and $f_n$ for type $A_{2n}^{(2)\dagger}$ is performed twice.

The maps $\varepsilon_a$ and $\varphi_a$ are defined by~\eqref{equation.eps phi} and the weight is given by~\eqref{eq:weight_function}. 
\end{dfn}

\begin{ex}
\label{ex:crystal_operator_A2even}
Consider $\g$ of type $A_6^{(2)}$ and $(\nu, J) \in \RC(B^{3,2})$ to be
\[
\raisebox{-6pt}{$\emptyset$}
\quad\quad
 {
\begin{array}[t]{r|c|l}
 \cline{2-2} 0 &\phantom{|}& 0 \\
 \cline{2-2} 
\end{array}
} 
\quad
\raisebox{13pt}{
\begin{array}[t]{r|c|l}
 \cline{2-2} 0 &\phantom{|}& 0 \\
 \cline{2-2}  &\phantom{|}& 0 \\
 \cline{2-2} 
\end{array}
}
\raisebox{-12pt}{.}
\]
The first application of the algorithm for $f_3$ in type $A_6^{(2)\dagger}$ results in
\[
\raisebox{-6pt}{$\emptyset$}
\quad\quad
 {
\begin{array}[t]{r|c|l}
 \cline{2-2} 0 &\phantom{|}& 0 \\
 \cline{2-2} 
\end{array}
} 
\quad
\raisebox{13pt}{
\begin{array}[t]{r|c|c|l}
 \cline{2-2} 0 &\phantom{|}& \multicolumn{2 }{l}{ 0 } \\
 \cline{2-3} 0 &\phantom{|}&\phantom{|}& -\frac{1}{2} \\
 \cline{2-3} 
\end{array}
}
\raisebox{-12pt}{,}
\]
and so we have
\[
\raisebox{-6pt}{$f_3(\nu, J) =$}
\hspace{10pt}
\raisebox{-6pt}{$\emptyset$}
\quad\quad
 {
\begin{array}[t]{r|c|l}
 \cline{2-2} 0 &\phantom{|}& 0 \\
 \cline{2-2} 
\end{array}
} 
\quad
\raisebox{13pt}{
\begin{array}[t]{r|c|c|c|l}
 \cline{2-2} 0 &\phantom{|}& \multicolumn{3 }{l}{ 0 } \\
 \cline{2-4} -1 &\phantom{|}&\phantom{|}&\phantom{|}& -1 \\
 \cline{2-4} 
\end{array}
}
\raisebox{-12pt}{.}
\]
\end{ex}

\begin{prop}
\label{prop:ep_phi}
Let $\g$ be of affine type. Fix some $a \in I_0$.
Let $x$ be the smallest label of $(\nu, J)^{(a)}$, $s = \min(0, x)$, and $k_a = 1$ except $k_n = 2$ for type $A_{2n}^{(2)\dagger}$. Then we have
\[
\varepsilon_a(\nu, J) = -k_a s \quad \quad \varphi_a(\nu, J) = k_a (p_{\infty}^{(a)} - s).
\]
\end{prop}

\begin{proof}
We note that unless $\g = A_{2n}^{(2)}, A_{2n}^{(2)\dagger}$ and $a = n$, we have $k_a = 1$. The proof that 
$\varphi_a(\nu, J) = p_{\infty}^{(a)} - s$ was originally given for simply-laced types in~\cite[Lemma 3.6]{S06}, whereas
$\varepsilon_a(\nu, J) = -s$ for types $A_n^{(1)}$ and $D_n^{(1)}$ was given in~\cite[Theorem 3.8]{Sakamoto13}. 
For non-simply-laced types, we separate the proof into cases depending on the value of $k_a$.

\case{Case $k_a = 1$}

\noindent
Let either $a \neq n$ or $\g$ not be of type $A_{2n}^{(2)}$. 
The proofs for  $\varphi_a$ and $\varepsilon_a$ hold verbatim because the vacancy numbers for $\nu^{(a)}$ change as in the simply-laced case and the 
proof only involves $\nu^{(a)}$.

Thus assume $\g$ is of type $A_{2n}^{(2)}$ and $a = n$. The proof is the same as in~\cite{S06,Sakamoto13} except everything is scaled by $1/2$. 
Since we apply this algorithm twice, the claim follows.

\case{Case $k_a = 2$}

\noindent
Therefore $\g = A_{2n}^{(2)\dagger}$ and $a = n$. 
The proof that $p_{\infty}^{(n)} - s$ changes by $1/2$ is the same as in~\cite{S06} except everything is scaled by $1/2$. 
Since we multiply this by $k_n=2$, the claim follows. Similar changes apply to $\varepsilon_a$ in comparison 
to~\cite{Sakamoto13}.
\end{proof}

\begin{remark}
Proposition~\ref{prop:ep_phi} immediately implies that $\langle \alpha_a^\vee, \wt(\nu,J)\rangle = \varphi_a(\nu,J) - \varepsilon_a(\nu,J)$ using
Equation~\eqref{eq:weight_p}. This shows that Definition~\ref{def:crystal_ops_nsl} defines an abstract crystal in the sense of 
Definition~\ref{definition.abstract crystal}.
\end{remark}

\subsection{Crystal operators and virtualization}

To show that the virtualization map $v$ defines a virtual crystal, we must first prove a lemma showing that applying $f_a^{\gamma_a}$ 
and $e_a^{\gamma_a}$ on the ambient rigged configurations gives us another element in our image under $v$.

\begin{lemma}
\label{lemma:act_same_string}
Fix $a \in I_0$ and $\gamma \in \ZZ_{>0}$. Consider a rigged configuration $(\nu, J)$ with $m_i^{(a)} = 0$ for all $i \notin \gamma \ZZ$ 
and $x \in \gamma \ZZ$ for all $x \in J_i^{(a)}$ and $i \in \ZZ_{>0}$. Let $1 \leq k \leq \gamma$, and suppose $e_a$ and $f_a$ act on 
the string $(\ell, x_{\ell})$ in $(\nu,J)^{(a)}$. Then $e_a^k(\nu, J)$ and $f_a^k(\nu, J)$ send $(\ell, x_{\ell})$ to the string $(\ell \mp k, x_{\ell} \pm k)$. 
Moreover, $f_a^{\gamma}(\nu, J)$ and $e_a^{\gamma}(\nu, J)$ both have $m_i^{(a)} = 0$ for all $i \notin \gamma \ZZ$ and 
$x \in \gamma \ZZ$ for all $x \in J_i^{(a)}$ and $i \in \ZZ_{>0}$.
\end{lemma}

\begin{proof}
We consider $f_a^k(\nu, J)$. Let $\ell$ be the maximal length of all strings of the smallest label $x_{\ell}$ in $(\nu,J)^{(a)}$. Since $\ell$ is the 
largest such string, all strings of length at least $\ell$ have labels $x \geq x_{\ell} + \gamma$. Thus when we apply $f_a$ to 
$(\nu, J)$, the new string in $(\nu,J)^{(a)}$ is $(\ell + 1, x_{\ell}^{\prime})$ with $x_{\ell+1}^{\prime} = x_{\ell} - 1$ and all strings $(i,x)\in (\nu,J)^{(a)}$ 
of length $i\ge \ell + 1$ have labels $x^{\prime} = x - 2$. Thus we have $x^{\prime} \geq x_{\ell}^{\prime} + \gamma - 1$. Therefore 
applying $f_a$ to $(\nu^{\prime}, J^{\prime})$, we act on $(\ell+1, x_{\ell}-1)$ as before. Also note that $p_i^{(a)}$ does not 
change for any $i < \ell$, therefore $J_i^{(a)}$ for $i < \ell$ does not change either. Iterating this we obtain a new string
$(\ell+k,x_\ell-k)$ and $x\ge (x_\ell-k) +\gamma -k$ for any string $(i,x) \in f_a^k(\nu,J)^{(a)}$ with $i\ge \ell+\gamma$. Hence
$f_a$ acts on the string $(\ell+k,x_\ell-k)$ again.
Taking $k = \gamma$ we get our second claim. The proof for $e_a$ is similar.
\end{proof}

In other words, given the conditions of Lemma~\ref{lemma:act_same_string}, $e_a^k$ and $f_a^k$ act on the same string for all $1 \leq k \leq \gamma$
(see also~\cite[Proposition 4.10]{Sakamoto13}).

\begin{remark}
Consider type $A_{2n}^{(2)}$ with $f_n$. We note that since all riggings in $(\nu,J)$ are integral and the first application of the 
type $A_{2n}^{(2)\dagger}$ algorithm changes the selected rigging by 1/2, the second application of the algorithm will act on 
the same string similar to Lemma~\ref{lemma:act_same_string}. Thus all riggings in $f_n(\nu, J)$ will be integral. A similar 
statement holds for $e_n$.
\end{remark}

We also need the following key fact.

\begin{prop}
\label{prop:crystal_ops_commute_virtual}
The crystal operators defined in Definition~\ref{def:crystal_ops_nsl} commute with the virtualization map 
$v \colon \RC(L) \to V \subseteq \RC(\virtual{L})$, where $V$ is defined by Equations~\eqref{eq:rc_virtualization_map} 
and~\eqref{eq:rc_virtualization_map_A2dual}.
\end{prop}

Before we give the proof, let us provide an example.

\begin{ex}
Using the rigged configuration $(\nu, J)$ from Example~\ref{ex:crystal_operator_A2even}, we have
\begin{align*}
\raisebox{-6pt}{$(\virtual{\nu}, \virtual{J}) = $}
& \hspace{10pt}
\raisebox{-6pt}{$\emptyset$}
\quad\quad
{
\begin{array}[t]{r|c|l}
\cline{2-2} 0 &\phantom{|}& 0 \\
\cline{2-2}
\end{array}
}
\hspace{19pt}
\raisebox{13pt}{
\begin{array}[t]{r|c|l}
\cline{2-2} 0 &\phantom{|}& 0 \\
\cline{2-2} &\phantom{|}& 0 \\
\cline{2-2}
\end{array}
}
\hspace{45pt}
{
\begin{array}[t]{r|c|l}
\cline{2-2} 0 &\phantom{|}& 0 \\
\cline{2-2}
\end{array}
}
\quad\quad
\raisebox{-6pt}{$\emptyset$}
\;\raisebox{-9pt}{,}
\\
\raisebox{-6pt}{$\virtual{f}_3(\virtual{\nu}, \virtual{J}) = $}
& \hspace{10pt}
\raisebox{-6pt}{$\emptyset$}
\quad\quad
{
\begin{array}[t]{r|c|l}
\cline{2-2} 0 &\phantom{|}& 0 \\
\cline{2-2}
\end{array}
}
\quad
\raisebox{13pt}{
\begin{array}[t]{r|c|c|l}
\cline{2-2} 0 &\phantom{|}& \multicolumn{2 }{l}{ 0 } \\
\cline{2-3} -1 &\phantom{|}&\phantom{|}& -1 \\
\cline{2-3}
\end{array}
}
\hspace{24pt}
{
\begin{array}[t]{r|c|l}
\cline{2-2} 0 &\phantom{|}& 0 \\
\cline{2-2}
\end{array}
}
\quad\quad
\raisebox{-6pt}{$\emptyset$}
\;\raisebox{-9pt}{,}
\\
\raisebox{-6pt}{$f_3^v(\virtual{\nu}, \virtual{J}) = $}
& \hspace{10pt}
\raisebox{-6pt}{$\emptyset$}
\quad\quad
{
\begin{array}[t]{r|c|l}
\cline{2-2} 0 &\phantom{|}& 0 \\
\cline{2-2}
\end{array}
}
\quad
\raisebox{13pt}{
\begin{array}[t]{r|c|c|c|l}
\cline{2-2} 0 &\phantom{|}& \multicolumn{3 }{l}{ 0 } \\
\cline{2-4} -2 &\phantom{|}&\phantom{|}&\phantom{|}& -2 \\
\cline{2-4}
\end{array}
}
\quad
{
\begin{array}[t]{r|c|l}
\cline{2-2} 0 &\phantom{|}& 0 \\
\cline{2-2}
\end{array}
}
\quad\quad
\raisebox{-6pt}{$\emptyset$}
\;\raisebox{-9pt}{.}
\end{align*}
The last line can be seen as the virtualization of $f_3(\nu,J)$ in Example~\ref{ex:crystal_operator_A2even}.
\end{ex}

\begin{proof}[Proof of Proposition~\ref{prop:crystal_ops_commute_virtual}]
We will handle types $A_{2n}^{(2)}$ and $A_{2n}^{(2)\dagger}$ separately, so we consider the case when 
$\widetilde{\gamma}_a = \gamma_a$ 
for all $a \in I$. If we write $p_i^{(a)} = \sum_{j \geq 1} \min(i, j) L_j^{(a)} - q_i^{(a)}$, from~\cite[Eq.~(3.2)]{S06II} we can express
\begin{equation}
\label{eq:alt_vac_nums}
q_i^{(a)} = \sum_{b \in I_0}  \frac{A_{ab}}{\gamma_b} \sum_{j \in \ZZ} \min(\gamma_a i, \gamma_b j) m_{j}^{(b)}.
\end{equation}

For explicit expressions for $q_i^{(a)}$ in all non-exceptional affine types, see~\cite{OSS03} (for example Equation~(4.2)). From the definition 
of the (virtual) crystal operators, Equation~\eqref{eq:vacancy_virtualization}, and Lemma~\ref{lemma:act_same_string}, we have $f_a(\nu, J) = 0$ 
if and only if $f^v_a \bigl( v(\nu, J) \bigr) = 0$. Note that in the virtualization map, we do not simultaneously have $\gamma_a \neq 1$ and 
$\absval{\phi^{-1}(a)} \neq 1$ for any fixed $a \in I_0$ by Remark~\ref{rem:phi_gamma}. Thus we have the following 3 disjoint cases. 
Let $(\virtual{\nu},\virtual{J}) = v(\nu,J)$.

\case{Case $\gamma_a = 1$ and $\absval{\phi^{-1}(a)} \neq 1$}

\noindent 
Since $\virtual{\nu}^{(b)} = \virtual{\nu}^{(b^{\prime})}$ for all $b,b^{\prime} \in \phi^{-1}(a)$, their 
images under $f^v_a$ agree, so we only need to check the vacancy numbers on neighboring $\nu^{(k)}$, that is $k$ such that $\{k, a\}$ is 
an edge in the Dynkin diagram of $\g$. If for any $c \in \phi^{-1}(k)$ such that $\virtual{\nu}^{(c)}$ is not a common neighbor of $\virtual{\nu}^{(b)}$ 
and $\virtual{\nu}^{(b^{\prime})}$ (if it holds for some $c \in \phi^{-1}(k)$, then it holds for all $c \in \phi^{-1}(k)$), we have $\gamma_k = 1$ and so 
$f^v_a(\virtual{\nu}, \virtual{J}) \in V$. Also for all $b \in \phi^{-1}(a)$ and $c \in \phi^{-1}(k)$ adjacent to $b$, we have $A_{ak} = A_{bc}$, and 
therefore $v \circ f_a = f^v_a \circ v$.

Now suppose $c = \phi^{-1}(k)$ is a common neighbor (note that $\absval{\phi^{-1}(k)} = 1$). This occurs for $a = n-1$ and $k = n$ with 
$\g = C_n^{(1)}, D_{n+1}^{(2)}, A_{2n}^{(2)}, A_{2n}^{(2)\dagger}$, for $a = n$ and $k = n-1$ with $\g = B_n^{(1)}, A_{2n-1}^{(2)}$, for $a = 2$ 
and $k = 3$ with $\g = F_4^{(1)}, E_6^{(2)}$, and $a = 1$ and $k = 2$ with $\g = G_2^{(1)}, D_4^{(3)}$. 
In this case $\virtual{p}_j^{(c)}$ is increased by $\absval{\phi^{-1}(a)}$ for all $j \geq \gamma_k^{-1} \gamma_a i$, where $i$ is the length of 
string $f_a$ acts on (equivalently $f_a^v$). Thus the riggings are also increased by $\absval{\phi^{-1}(a)}$. Recall from 
Remark~\ref{rem:phi_gamma} that if $\gamma_k \neq 1$, then $\gamma_k = \absval{\phi^{-1}(a)}$, so $f^v_a(\virtual{\nu}, \virtual{J}) \in V$. 
Looking at $f_a(\nu, J)$, from Equation~\eqref{eq:alt_vac_nums} we see the change to $p_j^{(k)}$ is $-A_{ka}$ for all 
$j \geq \gamma_k^{-1} \gamma_a i$. We note that $-A_{ka} = \gamma_k^{-1} \absval{\phi^{-1}(a)}$, which can be seen by direct computation. 
Therefore we have $v \circ f_a = f^v_a \circ v$.

\case{Case $\gamma_a \neq 1$ and $\absval{\phi^{-1}(a)} = 1$}

\noindent
By Equation~\eqref{eq:rc_virtualization_map} and Lemma~\ref{lemma:act_same_string}, applying $f_a^v = \virtual{f}_a^{\gamma_a}$ adds boxes 
to the same string and changes the rigging by $-\gamma_a$. Thus $f^v_a(\virtual{\nu}, \virtual{J}) \in V$, and from the definition of $f_a$ and a 
straightforward check of the vacancy numbers similar to the previous case, we have $v \circ f_a = f^v_a \circ v$.

\case{Case $\gamma_a = 1$ and $\absval{\phi^{-1}(a)} = 1$}

\noindent
Trivially we have $f^v_a(\virtual{\nu}, \virtual{J}) \in V$, and a straightforward check of the vacancy numbers similar to the first case implies 
$v \circ f_a = f^v_a \circ v$.

\vspace{12pt}
Now for type $A_{2n}^{(2)}$, we have $\widetilde{\gamma}_a = 1$ for all $a \in I_0$. Furthermore because $f^v_n = \virtual{f}_n^2$ and the algorithm
for $f_n$ does the usual algorithm twice but the changes in the riggings and $p_i^{(n)}$ are scaled by $1/2 = \gamma_n^{-1}$, we have 
$v \circ f_a = f^v_a \circ v$ for all $a \in I_0$. 
For type $A_{2n}^{(2)\dagger}$, it has the same virtualization map on the partitions as type $A_{2n}^{(2)}$, but with $f^v_n = \virtual{f}_n$. 
Thus from the definition of $f_a$, we have $v \circ f_a = f^v_a \circ v$. 

Similarly for all types/cases we have $v \circ e_a = e^v_a \circ v$ for all $a \in I_0$.
\end{proof}

\begin{lemma}
\label{lemma:rc_virtual_crystal}
Let $\g$ be of affine type. The crystal $\RC(L; \lambda)$ of type $\g_0$ virtualizes into $\RC(\virtual{L}; \Psi(\lambda))$ with 
virtualization map $v$ given by Equations~\eqref{eq:rc_virtualization_map} and~\eqref{eq:rc_virtualization_map_A2dual}.
\end{lemma}

\begin{proof}
By the definition of $v$, condition~(\ref{mor:1}) for a crystal morphism is satisfied. Condition~(\ref{mor:3}) is satisfied by 
Proposition~\ref{prop:crystal_ops_commute_virtual}. The condition that the weights agree is easy to see from our 
definition of the virtualization map. The remainder of condition~(\ref{mor:2}) holds because of the computation of 
$\varepsilon_a$ and $\varphi_a$ in Proposition~\ref{prop:ep_phi}, Equation~\eqref{eq:rc_virtualization_map}, and 
Equation~\eqref{eq:vacancy_virtualization}. By Lemma~\ref{lemma:act_same_string}, Proposition~\ref{prop:ep_phi}, and 
the fact that we have defined these as regular crystals, we have that $v$ is a bijection. Therefore $v$ is a crystal isomorphism.
\end{proof}

We note that we can characterize the image of $v$ by using Equations~\eqref{eq:rc_virtualization_map} 
and~\eqref{eq:rc_virtualization_map_A2dual}. Moreover, by weight considerations this virtualization map is the unique virtualization of $\RC(L)$ 
into $\RC(\virtual{L})$ up to permutation of the classical components.
Thus we have the extension of Theorem~\ref{thm:crystal_structure} to all finite types (with possibly different rigged 
configurations coming from the affine type).

\begin{thm}
\label{thm:hw_crystals_nsl}
Let $\g$ be an affine Lie algebra. For $(\nu, J) \in \hwRC(L; \lambda)$, let $X_{(\nu, J)}$ be the crystal generated by $(\nu, J)$ 
and $e_a, f_a$ for $a \in I_0$. Then $X_{(\nu, J)}$ is isomorphic to the crystal graph $B(\lambda)$ as $U_q(\g_0)$-crystals.
\end{thm}

\begin{proof}
This follows from Theorem~\ref{thm:crystal_structure}, Theorem~\ref{thm:virtual_highest_weight}, and Lemma~\ref{lemma:rc_virtual_crystal}.
\end{proof}

We also have that cocharge is invariant on each classical component.

\begin{prop}
\label{prop:constant_cocharge}
Consider a classical component $X_{(\nu, J)}$ as in Theorem~\ref{thm:hw_crystals_nsl}. The cocharge $\cc$ is constant on $X_{(\nu,J)}$.
\end{prop}

\begin{proof}
In~\cite[Thm.~3.9]{S06}, it was shown that cocharge is constant on classical components in simply-laced types. From 
Equation~\eqref{eq:cocharge_virtualization}, we have that cocharge is constant on classical components.
\end{proof}

We could also prove Proposition~\ref{prop:constant_cocharge} directly by a similar argument to~\cite[Theorem~3.9]{S06}.

\begin{cor}
\label{cor:ls_lt_morphisms}
The maps $\ls$ and $\lt$ are strict crystal embeddings.
\end{cor}

\begin{proof}
Since $\ls$ is the identity map on the rigged configurations and preserves the weights, it must be a strict crystal embedding by 
Theorem~\ref{thm:hw_crystals_nsl}. Since $\lt$ adds singular strings of length 1 to $\nu^{(a)}$ for $a < r$ and preserves the 
vacancy numbers, a straightforward check shows the resulting rigged configuration also preserves the weight. Additionally $\lt$ 
commutes with $e_a$ and $f_a$ by similar argument to~\cite[Lemma~C.3]{DS06}. Therefore by Theorem~\ref{thm:hw_crystals_nsl}, 
it must be a strict crystal embedding.
\end{proof}

\section{The filling map}
\label{sec:filling_map}

In this section we describe the filling map for all non-exceptional types on a case-by-case basis, extending the results
for type $D_n^{(1)}$ in~\cite{OSS13}. Philosophically, the map $\Phi$ between tensor products of KR crystals and
rigged configurations consists of a sequence of splitting maps and $\delta$. For each factor $B^{r,s}$ in the tensor product,
the map $\delta$ is applied $rs$ times and results in $rs$ letters in $B^{1,1}$. However, if an element $b\in B^{r,s}$ is in the classical 
component $B(\lambda) \subseteq B^{r,s}$, its KN tableaux representation has shape $\lambda$.
The filling map $\fillmap \colon B^{r,s} \to (B^{1,1})^{\otimes rs}$ makes the link between the KN tableaux and the map $\Phi$
by effectively ``filling in'' the shape $\lambda$ to an $r \times s$ rectangle.
Here we describe the explicit image of $\fillmap$ for the classically highest weight elements and then extend it as a classical crystal morphism. 
We consider $(B^{1,1})^{\otimes rs}$ as an $r \times s$ rectangle where the classical crystal structure is given by the column reading 
word. We denote this (classical) crystal by $T^{r,s}$ and call the resulting column-strict tableaux \emph{Kirillov-Reshetikhin (KR) tableaux} 
following~\cite{OSS13}. We will also show that the filling map corresponds to the exact image of $\Phi$ on highest weight elements 
for a single factor $B^{r,s}$.

\begin{remark}
\label{remark:spin_tableaux}
In~\cite{KN94,HK02}, the tableaux for the spin cases of $B(\clfw_n)$ in type $B_n$ and $B(\clfw_{n-1})$ and $B(\clfw_n)$ in 
type $D_n$ are given in terms of half-width columns of height $n$ filled with $\{+, -\}$. These tableaux can be identified
with full-width columns where $+$ (resp. $-$) at height $i$ goes to $i$ (resp. $\overline{i}$) sorted to be strictly increasing.
This is a virtualization map with $\gamma_a=2$ for all $a\in I_0$, so that in particular $e^v_a = \virtual{e}_a^2$ and 
$f^v_a = \virtual{f}_a^2$ for all $a \in I_0$. This corresponds to the natural embedding of $B(\lambda)$ in 
$B(2\lambda)$ (with $\virtual{e}_a$ and $\virtual{f}_a$ as in~\cite{HK02}).
\end{remark}

Note that for type $A_n^{(1)}$, the filling map is the identity since $B^{r,s} \cong B(s\Lambda_r)$ classically in this case. We begin in 
Section~\ref{sec:typeD_fill} by recalling the filling map in type $D_n^{(1)}$~\cite{OSS13} and then proceed onto all other 
non-exceptional non-simply-laced types.

\subsection{Filling map for type $D_n^{(1)}$}
\label{sec:typeD_fill}

We first recall from~\cite{Kleber98} the structure of the Kleber tree for $B^{r,s}$ in type $D_n^{(1)}$ as it will be needed later. We begin by 
considering the spinor cases, i.e. $r = n-1,n$, in which case the Kleber tree is trivial -- it consists only of the root. Thus the only highest weight 
rigged configuration for both spinor cases is the empty rigged configuration (where $\nu^{(a)} = \emptyset$ for all $a \in I_0$). Hence the unique 
highest weight tableau is given by $s$ columns of the form $\column{1, \dotsc, n-1, \overline{n}}$ for $r = n-1$ and of the form 
$\column{1, \dotsc, n}$ for $r = n$ and the filling map is the identity map (recall that we are using doubled spin columns, see 
Remark~\ref{remark:spin_tableaux}).

Next we consider the case $r < n-1$. The Kleber tree structure was originally described in~\cite[Sec.~3.5]{Kleber98} and the 
resulting rigged configurations were given in~\cite[Prop.\ 3.3]{OSS13}. We give a proof here for completeness as some of the details
are used for other types.
Let $\overline{\lambda}$ denote the complement shape of $\lambda$ in an $r \times s$ rectangle and let $\mu^{[m]}$ denote the partition 
$\mu$ but with the first $m$ rows removed. Let $\mu^{1/2}$ be the partition with multiplicities $m_i(\mu) / 2$.

\begin{lemma}\cite[Prop.\ 3.3]{OSS13}
\label{lemma:hw_typeD}
Let $B^{r,s}$ be a KR crystal of type $D_n^{(1)}$ with $r < n$. We have
\[
\RC(B^{r,s}) = \bigoplus_{\lambda} \RC(B^{r,s}; \lambda),
\]
where $\lambda$ is obtained by removing vertical dominoes from an $r \times s$ rectangle.
Moreover, the highest weight rigged configuration in $\RC(B^{r,s}; \lambda)$ is
\begin{equation}
\label{eq:highest_weight_typeD}
\nu^{(a)} = \begin{cases}
\overline{\lambda}^{[r-a]} & 1 \leq a < r, \\
\overline{\lambda} & r \leq a < n - 1, \\
\overline{\lambda}^{1/2} & a = n-1, n,
\end{cases}
\end{equation}
with all riggings 0.
\end{lemma}

\begin{proof}
For $T_1$ in Definition~\ref{def:kleber_algorithm}, we have one node $t_0 := \clfw_r$. Next, to obtain other dominant weights, we can only subtract
\[
	\alpha^{(k_1)} := \alpha_{r-2k_1+1} + 2 \alpha_{r-2k_1+2} + \cdots + 2k_1 \alpha_r + \cdots + 2k_1 \alpha_{n-2} + k_1 \alpha_{n-1} 
	+ k_1 \alpha_n = \clfw_r - \clfw_{r-2k_1}
\]
by (K2) of Definition~\ref{def:kleber_algorithm} where $1 \leq k_1 \leq r/2$, resulting in $\clfw_{r-2k_1}$. Pictorially, this removes 
$k_1$ vertical dominoes from the single column of height $r$. This yields all possible children of $t_0$.

Next we add $\clfw_r$ to all nodes of $T_1$ to get $T_2'$. We now consider a particular leaf $x$ that was obtained from its parent using $\alpha^{(k_1)}$. To obtain all children of $x$, we can only subtract $\alpha^{(k_2)}$ where $k_1 \geq k_2 > 0$ by the additional conditions in Step~(K2) of Definition~\ref{def:kleber_algorithm}. Thus this changes the newly added $\clfw_r$ to $\clfw_{r-2k_2}$, and this is the only possibility because otherwise we would have to subtract $\alpha_a$ for some $a < r-2k_1+1$, violating~(K2). Ranging over all leaves we obtain $T_2$. We can iterate the above to see that for any leaf in $T(B^{r,s})$, we must have a sequence $r/2 \geq k_1 \geq k_2 \geq \cdots \geq k_s > 0$. Note that there are exactly $s$ steps needed to construct $T(B^{r,s})$ since we can only change the newly added weight at each step. Furthermore, each sequence gives rise to a unique dominant weight.

Fix a node $x \in T(B^{r,s})$ at depth $p$ and a sequence $r/2 \geq k_1 \geq k_2 \geq \cdots \geq k_p > 0$ which denotes the path to $x$ and $\alpha^{(k_i)} = x^{(i-1)} - x^{(i)}$. Following the convention given in Definition~\ref{def:kleber_algorithm}, we have $k_q = 0$ for all $q > p$. Now fix some $1 \leq i \leq p$, and let $d_i = k_i - k_{i+1}$. We have
\begin{align*}
\alpha^{(k_i)} - \alpha^{(k_{i+1})} = & \alpha_{r-2k_i+1} + 2 \alpha_{r-2k_i+2} + \cdots + 2d_i \alpha_{r-2k_{i+1}} + \cdots + 2d_i \alpha_{n-2} + d_i \alpha_{n-1} + d_i \alpha_n.
\end{align*}
Therefore, we have
\[
m_i^{(a)} =
\begin{cases}
0 & 1 \leq a \leq r-2k_i \\
2 \bigl( a - (r-2k_i) \bigr) & r-2k_i < a \leq r - 2k_{i+1} \\
2d_i & r - 2k_{i+1} < a < n-1, \\
d_i & a = n-1, n,
\end{cases}
\]
since $(\alpha_a \mid \clfw_b) = \delta_{ab}$. It is straightforward to see this is our desired $\nu$ (you can also consider the term $c_a \alpha_a$ in $\alpha^{(k_i)}$ corresponding to adding a column of height $2c_a$ to $\nu^{(a)}$). Since all of the resulting vacancy numbers are $0$, the only possible riggings are all $0$. Thus we have the desired rigged configuration.
\end{proof}

Note that since we remove vertical dominoes here, $\overline{\lambda}^{1/2}$ is well-defined.

\begin{ex}
Consider $T(B^{12, 8})$ in type $D_{25}^{(1)}$ and the sequence
\[
k_1 = 5 \quad \geq \quad k_2 = 5 \quad \geq \quad k_3 = 3 \quad \geq \quad k_4 = 2 \quad \geq \quad k_5 = 2 \quad \geq \quad k_6 = 1.
\]
Therefore $\lambda = 2\clfw_{12} + \clfw_{10} + 2\clfw_8 + \clfw_6 + 2\clfw_2$, and we have
\[
\begin{tikzpicture}[baseline]
\draw (-4,0) node {$\nu^{(12)} =$};
\draw (-2.6,2) node {$d_1 = 0$};
\draw (-2.6,0.97) node {$d_2 = 2$};
\draw (-2.6,-0.5) node {$d_3 = 1$};
\draw (-2.6,-1) node {$d_4 = 0$};
\draw (-2.6,-1.5) node {$d_5 = 1$};
\draw (-2.6,-2.5) node {$d_6 = 1$};
\draw (3,0) node {$= \overline{\lambda},$};
\matrix [matrix of math nodes,column sep=-.4, row sep=-.5,text height=8,text width=8,align=center,inner sep=3] 
{
\node[draw,fill=gray!50]{}; & \node[draw,fill=gray!50]{}; & \node[draw,fill=gray!50]{}; & \node[draw,fill=gray!50]{}; & \node[draw,fill=gray!50]{}; & \node[draw,fill=gray!50]{}; & \node[draw,fill=gray!50]{}; & \node[draw,fill=gray!50]{}; \\
\node[draw,fill=gray!50]{}; & \node[draw,fill=gray!50]{}; & \node[draw,fill=gray!50]{}; & \node[draw,fill=gray!50]{}; & \node[draw,fill=gray!50]{}; & \node[draw,fill=gray!50]{}; & \node[draw,fill=gray!50]{}; & \node[draw,fill=gray!50]{}; \\
\node[draw]{}; & \node[draw]{}; & \node[draw,fill=gray!50]{}; & \node[draw,fill=gray!50]{}; & \node[draw,fill=gray!50]{}; & \node[draw,fill=gray!50]{}; & \node[draw,fill=gray!50]{}; & \node[draw,fill=gray!50]{}; \\
\node[draw]{}; & \node[draw]{}; & \node[draw,fill=gray!50]{}; & \node[draw,fill=gray!50]{}; & \node[draw,fill=gray!50]{}; & \node[draw,fill=gray!50]{}; & \node[draw,fill=gray!50]{}; & \node[draw,fill=gray!50]{}; \\
\node[draw]{}; & \node[draw]{}; & \node[draw,fill=gray!50]{}; & \node[draw,fill=gray!50]{}; & \node[draw,fill=gray!50]{}; & \node[draw,fill=gray!50]{}; & \node[draw,fill=gray!50]{}; & \node[draw,fill=gray!50]{}; \\
\node[draw]{}; & \node[draw]{}; & \node[draw,fill=gray!50]{}; & \node[draw,fill=gray!50]{}; & \node[draw,fill=gray!50]{}; & \node[draw,fill=gray!50]{}; & \node[draw,fill=gray!50]{}; & \node[draw,fill=gray!50]{}; \\
\node[draw]{}; & \node[draw]{}; & \node[draw]{}; & \node[draw,fill=gray!50]{}; & \node[draw,fill=gray!50]{}; & \node[draw,fill=gray!50]{}; & \node[draw,fill=gray!50]{}; & \node[draw,fill=gray!50]{}; \\
\node[draw]{}; & \node[draw]{}; & \node[draw]{}; & \node[draw,fill=gray!50]{}; & \node[draw,fill=gray!50]{}; & \node[draw,fill=gray!50]{}; & \node[draw,fill=gray!50]{}; & \node[draw,fill=gray!50]{}; \\
\node[draw]{}; & \node[draw]{}; & \node[draw]{}; & \node[draw]{}; & \node[draw]{}; & \node[draw,fill=gray!50]{}; & \node[draw,fill=gray!50]{}; & \node[draw,fill=gray!50]{}; \\
\node[draw]{}; & \node[draw]{}; & \node[draw]{}; & \node[draw]{}; & \node[draw]{}; & \node[draw,fill=gray!50]{}; & \node[draw,fill=gray!50]{}; & \node[draw,fill=gray!50]{}; \\
\node[draw]{}; & \node[draw]{}; & \node[draw]{}; & \node[draw]{}; & \node[draw]{}; & \node[draw]{}; & \node[draw,fill=gray!50]{}; & \node[draw,fill=gray!50]{}; \\
\node[draw]{}; & \node[draw]{}; & \node[draw]{}; & \node[draw]{}; & \node[draw]{}; & \node[draw]{}; & \node[draw,fill=gray!50]{}; & \node[draw,fill=gray!50]{}; \\
};
\end{tikzpicture}
\]
where the grey shaded region corresponds to $\lambda$.
\end{ex}

Consider the classical component $B(\lambda) \subseteq B^{r,s}$ corresponding to some shape 
$\lambda = k_r \clfw_r + k_{r-2} \clfw_{r-2} + \cdots$. We want to describe the image of the highest weight 
element of highest weight $\lambda$ under the map $\fillmap \colon B(\lambda) \hookrightarrow T^{r,s}$. The resulting tableau $t$ 
was described in~\cite{OSS13}
 and is constructed as follows. Let $k_c$ be the first odd integer in the sequence $(k_{r-2}, k_{r-4}, \ldots)$ and if $k_c$ does 
 not exist, then set $c = -1$. The process proceeds by induction on the columns of $t$ from left to right.
\begin{enumerate}[(1)]
\item The first $k_r$ columns of $t$ are filled with $\column{1, 2, \dotsc, r-1, r}$.

\item For $k_h$ where $r > h \geq c$, add $\lfloor k_h / 2 \rfloor$ times the pair of columns 
$\column{1, \dotsc, h, \overline{r}, \dotsc, \overline{h+1}}$ then $[1, 2, \dotsc, r-1, r]$.

\item \label{step:repeat} Let $h$ be a column of $\lambda$ and $x = c + 1$. Add the column
\[
\column{1, \dotsc, h-1, h, r - (x-h-2), \dotsc, r-1, r, \overline{r}, \dotsc, \overline{x+1}, \overline{x}}.
\]
Now set $x$ to be the $(h+1)$-th letter of the previously added column and repeat~(\ref{step:repeat}) for all columns of height $h < c$.

\item \label{step:final} If $c > -1$, let $x$ be the final value we obtained from~(\ref{step:repeat}). The rightmost column is filled with
\[
\column{1, 2, \dotsc, (r+x-1)/2, \overline{(r+x-1)/2}, \dotsc, \overline{x+1}, \overline{x}}.
\]
\end{enumerate}

Recall the (affine) crystal isomorphism $\iota \colon B^{r,s}\to \RC(B^{r,s})$ given in~\cite{OSS13}. This map is natural in the 
sense that it maps classically highest weight elements to their unique corresponding classically highest weight rigged 
configurations of the same (classical) weight. The uniqueness comes from the fact that the classical decomposition is multiplicity free.

\begin{dfn}
\label{def:fill_D}
The (classical) crystal morphism $\fillmap \colon B^{r,s} \to T^{r,s}$ is given by the filling procedure above on highest weight 
elements and extending it as a crystal morphism.
\end{dfn}

\begin{thm}[{\cite[Thm.~5.9]{OSS13}}]
Let $B^{r,s}$ be a KR crystal of type $D_n^{(1)}$. Then
\[
\Phi = \fillmap \circ \iota^{-1}
\]
on highest weight elements with $\fillmap$ as in Definition~\ref{def:fill_D}.
\end{thm}

\subsection{Filling map for type $C_n^{(1)}$}

\subsubsection{$r < n$}

Let $\frac{1}{2} \mu$ denote the partition by scaling each row by $1/2$.

\begin{lemma}
\label{lemma:hw_typeC}
Let $B^{r,s}$ be a KR crystal of type $C_n^{(1)}$ with $r < n$. We have
\begin{equation}
\label{equation.RC decomp}
\RC(B^{r,s}) = \bigoplus_{\lambda} \RC(B^{r,s}; \lambda),
\end{equation}
where $\lambda$ is obtained by removing horizontal dominoes from an $r \times s$ rectangle. Moreover, the highest weight 
rigged configuration in $\RC(B^{r,s}; \lambda)$ is given by
\[
\nu^{(a)} = \begin{cases}
\overline{\lambda}^{[r-a]} & 1 \leq a < r, \\
\overline{\lambda} & r \leq a < n, \\
\frac{1}{2} \overline{\lambda} & r = n,
\end{cases} 
\]
with all riggings 0.
\end{lemma}

\begin{proof}
Recall that we realize $B^{r,s}$ of type $C_n^{(1)}$ as a virtual crystal in $B^{r,s} \otimes B^{2n-r,s}$ of type $A_{2n-1}^{(1)}$.
We will prove the assertion by constructing the ambient Kleber tree, selecting the virtual nodes, and then pulling back to the 
type $C_n^{(1)}$ setting. Let us begin by constructing the ambient Kleber tree. For $\virtual{T}_1^{\prime}$ in 
Definition~\ref{def:kleber_algorithm} we have one node 
$t_0 := \virtual{\clfw}_r + \virtual{\clfw}_{2n-r}$. Next to obtain other dominant weights, we first consider moving to the 
``nearest'' dominant weight $\virtual{\clfw}_{r-1} + \virtual{\clfw}_{2n-r+1}$ by subtracting the root $\virtual{\alpha}_{r,2n-r} 
:= \virtual{\alpha}_r + \virtual{\alpha}_{r+1} + \cdots + \virtual{\alpha}_{2n-r}$. Pictorially,
this is moving a box from the column of height $r$ to the column of height $2n-r$. Now to obtain the next sibling weight, 
we add $\virtual{\alpha}_{r-1,2n-r+1}$. In general, to get to all possible children of $t_0$, we add
\begin{equation}
\label{equation.alpha}
\begin{split}
\virtual{\alpha}^{(k_1)} & := \virtual{\alpha}_{r-k_1,2n-r+k_1} + \virtual{\alpha}_{r-k_1+1,2n-r+k_1-1} + \cdots 
+ \virtual{\alpha}_{r,2n-r}
\\ & = \virtual{\alpha}_{r-k_1} + 2 \virtual{\alpha}_{r-k_1+1} + \cdots + (k_1 - 1) \virtual{\alpha}_{r-1} + k_1 \virtual{\alpha}_r 
+ k_1 \virtual{\alpha}_{r+1}
\\ & \hspace{20pt} + \cdots + k_1 \virtual{\alpha}_{2n-r} + (k_1 - 1) \virtual{\alpha}_{2n-r+1} + \cdots + \virtual{\alpha}_{2n-r+k_1}
\end{split}
\end{equation}
for some $r \geq k_1 > 0$ to $t_0$ to obtain the other weights in $\virtual{T}_1$. Next we add 
$\virtual{\clfw}_r + \virtual{\clfw}_{2n-r}$ to all weights of $\virtual{T}_1$ to get $\virtual{T}_2^{\prime}$. We consider a particular 
leaf $x$ that was obtained from its parent using $\virtual{\alpha}^{(k_1)}$. To 
obtain all children of $x$, we can only subtract $\virtual{\alpha}^{(k_2)}$ where $k_1 \geq k_2 > 0$ by the additional conditions 
in Step~(K2) of Definition~\ref{def:kleber_algorithm}. This is because from~(K2), we can only change the newly added 
$\virtual{\clfw}_r$ and
$\virtual{\clfw}_{2n-r}$ to $\virtual{\clfw}_{r-k_2}$ and $\virtual{\clfw}_{2n-r+k_2}$, respectively, as otherwise we would subtract 
$\virtual{\alpha}_a$ for some $a \leq k_1$ and/or $a \geq 2n-k_1$ (i.e. move boxes from the recently added column of height 
$r$ to the one of height $2n-r$ and keeping the partition shape), violating~(K2). Ranging over all leaves, we obtain 
$\virtual{T}_2$. We can iterate the above to see that for any leaf in $\virtual{T}(B)$, we must have a sequence 
$r \geq k_1 \geq k_2 \geq \cdots \geq k_s > 0$. Note that there are exactly $s$ steps needed to construct $\virtual{T}(B)$ 
since we can only change the newly added weights at each step. Furthermore, each sequence gives rise to a unique 
dominant weight.

For the virtual Kleber tree, condition~(V1) of Definition~\ref{def:virtual_kleber} is satisfied by the symmetry of the 
$\virtual{\alpha}^{(k_i)}$ under the folding; namely by Equation~\eqref{equation.alpha} the coefficient of $\virtual{\alpha}_a$
in $\virtual{\alpha}^{(k_i)}$ is the same as the coefficient of $\virtual{\alpha}_{2n-a}$.
In order to satisfy~(V2) of Definition~\ref{def:virtual_kleber}, we must have $k_i = k_{i-1}$ for all $i \in 2\ZZ$ using the 
convention that $k_j = k_s$ for all $j > s$. This comes from the fact that $\gamma_n=2$ and $\gamma_a=1$
for $1\le a<n$, so that (V2) requires the coefficient of $\virtual{\alpha}_n$ in $\virtual{\alpha}^{(k_i)}$ to be equal to the same coefficient in $\virtual{\alpha}^{(k_{i-1})}$, which means $k_i=k_{i-1}$. Thus we can only select nodes at even depth and the 
selected weights after devirtualization correspond to removing horizontal dominoes. We note that $k_i$ is the number 
of horizontal dominoes removed from the $(r-i)$-th row of an $r \times s $ rectangle. From the description of 
$\virtual{\alpha}^{(k_i)}$, we have the desired rigged configuration.
\end{proof}

Recall that the classical decomposition of $B^{r,s}$ for $r < n$ is given by weights obtained from $s\clfw_r$ by removing horizontal 
dominoes. We note that Lemma~\ref{lemma:hw_typeC} and Theorem~\ref{thm:hw_crystals_nsl} imply that there exists a natural (classical) 
crystal morphism $\iota$, which maps the classical component $\RC(B^{r,s}; \lambda)$ to the unique 
corresponding classical component $B(\lambda) \subseteq B^{r,s}$. This morphism is a (classical) crystal isomorphism, in analogy 
to type $D_n^{(1)}$ given in Section~\ref{sec:typeD_fill}. In subsequent sections, we prove lemmas analogous to 
Lemma~\ref{lemma:hw_typeC} in other types. Hence we obtain crystal isomorphisms $\iota$ in all such cases.

\begin{dfn}[Type $C_n^{(1)}$ filling map]
\label{def:fill_C}
Let $r < n$ and consider a dominant weight $\lambda = \sum_{i \in I_0} k_i \clfw_i$ in the decomposition~\eqref{equation.RC decomp}
(note that only $k_r$ can be odd since we are removing horizontal dominoes and $k_i=0$ for $i>r$) and define $k_0 := s - \sum_{i \in I_0} k_i$. 
The image under $\fillmap$ of the (unique) classically highest weight element $u_{\lambda} \in B^{r,s}$ of classical weight $\lambda$ is given 
on columns from right to left by filling in $\column{1, \dotsc, r}$ followed by $\column{1, \dotsc, h, \overline{r}, \dotsc, \overline{h+1}}$, repeating this 
$\lfloor k_h / 2 \rfloor$ times, for $h = 0, 1, \dotsc, r$. If $k_r$ is odd, we add a leftmost column of $\column{1, \dotsc, r}$.
\end{dfn}

Note that Definition~\ref{def:fill_C} is a special case of Definition~\ref{def:fill_D}; namely only cases (1) and (2) occur.
Alternatively, consider a classically highest weight element $u_{\lambda} \in B(\lambda) \subseteq B^{r,s}$ of classical weight 
$\lambda$. Then $u_\lambda$ can be regarded as a tableau of shape $\lambda$ whose $k$-th row is filled by the letter $k$. The filling map 
on $u_{\lambda}$ is obtained by adding pairs $[\bar{k} \mid k]$ into the $k$-th row of $u_\lambda$ (for each horizontal domino 
removed) and then sorting elements increasingly from bottom to top within each column as necessary.

\begin{ex}
Consider $\lambda = 2\clfw_2 + \clfw_3$ for type $C_4^{(1)}$ and $B^{3,5}$. Then $\fillmap(u_{\lambda})$ is:
\[
\begin{tikzpicture}[baseline]
\matrix [matrix of math nodes,column sep=-.4, row sep=-.5,text height=8,text width=8,align=center,inner sep=3] 
{
\node[draw]{3}; & \node[draw,fill=gray!30]{\overline{3}}; & \node[draw,fill=gray!30]{3}; & \node[draw,fill=gray!30]{\overline{1}}; & \node[draw,fill=gray!30]{3}; \\
\node[draw]{2}; & \node[draw]{2}; & \node[draw]{2}; & \node[draw,fill=gray!30]{\overline{2}}; & \node[draw,fill=gray!30]{2}; \\
\node[draw]{1}; & \node[draw]{1}; & \node[draw]{1}; & \node[draw,fill=gray!30]{\overline{3}}; & \node[draw,fill=gray!30]{1}; \\
};
\end{tikzpicture}
\ ,
\]
where the filled in portion is shaded in grey.
\end{ex}

\begin{ex}
Consider $\lambda = 2\clfw_1 + 2 \clfw_2 + 2\clfw_4$ for type $C_{128}^{(1)}$ and $B^{5,6}$. Then $\fillmap(u_{\lambda})$ is:
\[
\begin{tikzpicture}[baseline]
\matrix [matrix of math nodes,column sep=-.4, row sep=-.5,text height=8,text width=8,align=center,inner sep=3] 
{
\node[draw,fill=gray!30]{\overline{5}}; & \node[draw,fill=gray!30]{5}; & \node[draw,fill=gray!30]{\overline{3}}; & \node[draw,fill=gray!30]{5}; & \node[draw,fill=gray!30]{\overline{2}}; & \node[draw,fill=gray!30]{5}; \\
\node[draw]{4}; & \node[draw]{4}; & \node[draw,fill=gray!30]{\overline{4}}; & \node[draw,fill=gray!30]{4}; & \node[draw,fill=gray!30]{\overline{3}}; & \node[draw,fill=gray!30]{4}; \\
\node[draw]{3}; & \node[draw]{3}; & \node[draw,fill=gray!30]{\overline{5}}; & \node[draw,fill=gray!30]{3}; & \node[draw,fill=gray!30]{\overline{4}}; & \node[draw,fill=gray!30]{3}; \\
\node[draw]{2}; & \node[draw]{2}; & \node[draw]{2}; & \node[draw]{2}; & \node[draw,fill=gray!30]{\overline{5}}; & \node[draw,fill=gray!30]{2}; \\
\node[draw]{1}; & \node[draw]{1}; & \node[draw]{1}; & \node[draw]{1}; & \node[draw]{1}; & \node[draw]{1}; \\
};
\end{tikzpicture}
\ .
\]
\end{ex}

We recall the algorithm for $\delta$ given in~\cite{OSS03} (with appropriate modifications due to our convention, see 
Remark~\ref{rem:rc_convention}). Suppose the leftmost factor is $B^{r,1}$. Set $\ell^{(r-1)} = 0$
and repeat the following process for $a=r,r+1,\ldots,n-1$ or until stopped. Find the minimal index $i\ge
\ell^{(a-1)}$ such that $(\nu,J)^{(a)}$ has a singular string of length $i$. If no such $i$ exists, return $a$ and stop. 
Otherwise set $\ell^{(a)}=i$ and continue. If the process has not terminated at $a=n-1$, find the minimal $i \geq \ell^{(n-1)}/2$ such that $(\nu,J)^{(n)}$ has a singular string of length $i$. If no such $i$ exists, return $n$ and stop. Otherwise set $\overline{\ell}^{(n)} = i$ and continue.
Next find the smallest index $i\ge 2 \overline{\ell}^{(n)}$ such that $(\nu,J)^{(n-1)}$ has a singular string of length $i$
and set $\overline{\ell}^{(n-1)}=i$; if no such singular string exists return $\overline{n}$ and stop.
If the process has not stopped continue as follows for $a=n-2,n-3,\ldots,1$ or until stopped. 
Find the minimal index $i\ge \overline{\ell}^{(a+1)}$ such that $(\nu,J)^{(a)}$ has a singular string
of length $i$. If no such $i$ exists, return $\overline{a+1}$ and stop.
Otherwise set $\overline{\ell}^{(a)}=i$ and continue. If the process does not
stop for $a\ge 1$ return $\overline{1}$.

Next we modify the rigged configuration by removing a box from the singular string of length 
$\ell^{(a)}$ for $a = r, \dotsc, n-1$ and $\overline{\ell}^{(a)}$ for $a = n, \dotsc, 1$ if such values are defined 
(if $\ell^{(a)} = \overline{\ell}^{(a)}$, then we remove 2 boxes from the same string). We then make the affected 
rows singular.

\begin{prop}
\label{prop:crystal_iso_C}
Let $B^{r,s}$ be a KR crystal of type $C_n^{(1)}$ and $r < n$. Then
\[
\Phi = \fillmap \circ \iota^{-1}
\]
on highest weight elements with $\fillmap$ as in Definition~\ref{def:fill_C}.
\end{prop}

\begin{proof}
We show the claim by induction on $s$ by noting that removing the leftmost column gives a highest element in $B^{r,s-1}$ (resp.
$\RC(B^{r,s-1})$) corresponding to the partition obtained from $\overline{\lambda}$ by replacing all rows of length $s$ (if they exist) 
by rows of length $s - 2$.

Consider a highest weight $\lambda$ and the corresponding rigged configuration $(\nu, J)$ given by Lemma~\ref{lemma:hw_typeC}. 
Suppose $\overline{\lambda}_1 < s$, so that the desired leftmost column of the filled tableau is $\column{1, \dotsc, r}$. We now check that we 
obtain this after removing the leftmost column by $\Phi$. Note that all rows in $\nu^{(r)}$ have length less than $s$, so after we split off the 
leftmost column, all of the vacancy numbers for $\nu^{(r)}$ are 1. However, all of the riggings are 0, so we stop and $\delta$ returns $r$. 
We note that this implies that $\nu$ is unchanged and we are now in $\RC(B^{r-1,1} \otimes B^{r,s-1})$.
This implies that all vacancy numbers for 
$\nu^{(r-1)}$, which again has only rows of length less than $s$, are 1. We are in a similar case to before, so $\delta$ returns $r-1$. 
We can repeat this for the entire column to obtain $\column{1, \dotsc, r}$. Note that the resulting rigged configuration is what we started with.

Now suppose that $\overline{\lambda}_1 = s$, so our desired leftmost column is $\column{1, \dotsc, h, \overline{r}, \dotsc, \overline{h+1}}$
for $h < r$. Note that $h$ is maximal such that $k_h\neq 0$.
Thus the shortest column in $\overline{\lambda}$ has height $r - h$ and
\[
m_s^{(a)} =
\begin{cases}
r- h & a \geq r, \\
a -h & r > a \geq h, \\
0 & h >a,
\end{cases}
\]
by Lemma~\ref{lemma:hw_typeC}.
Now splitting off the leftmost column increases $p_i^{(r)}$ for $i<s$ by 1 and leaves all other vacancy numbers unchanged. 
Hence after applying $\delta$, we remove 2 boxes from the same row in each of the longest rows of $\nu^{(a)}$ for $a \geq r$ and
1 box from $\nu^{(a)}$ for $h<a<r$. Thus the algorithm returns $\overline{h+1}$. The resulting riggings on the selected 
rows will be $0$. Next we are in $\RC(B^{r-1,1} \otimes B^{r,s-1})$, and now $\delta$ removes a box from the (unique) row of length 
$s-1$ in $\nu^{(r-1)}$, two boxes from the same row of length $s$ in $\nu^{(a)}$ for $a \geq r$, and a single box from a row of length 
$s$ in $\nu^{(a)}$ for $h+1 < a < r$. Hence $\delta$ returns $\overline{h+2}$ and the riggings on the selected rows will be $0$. Therefore 
by using a similar procedure, we can continue until we return $\overline{r}$, in which case all strings in $(\nu,J)^{(h)}$ are now non-singular 
because there are no rows of length (at least) $\overline{\lambda}_1 = s$ in $(\nu,J)^{(a)}$ for all $a \in I_0$ (and we are in 
$\RC(B^{h,1} \otimes B^{r,s-1})$). Then we fall back into the case when the 
column was $\column{1, \dotsc, r}$ since there are no singular strings. Thus we have removed two boxes from all strings 
of length $s$ in $(\nu,J)^{(a)}$ (if they exist), specifically from $\nu^{(r)} = \overline{\lambda}$ to obtain the desired highest weight element.
\end{proof}

\subsubsection{$r = n$}

Recall that $B^{n,s} = B(s\clfw_n)$.

\begin{prop}
\label{prop:crystal_iso_Cn}
Let $B^{n,s}$ be a KR crystal of type $C_n^{(1)}$. Then
\[
\Phi = \fillmap \circ \iota^{-1}
\]
on highest weight elements with $\fillmap$ the trivial filling map (i.e. the identity map on the unique highest weight tableau) 
and $\iota \colon B^{r,s} \to \RC(B^{r,s})$ is the natural crystal isomorphism.
\end{prop}

\begin{proof}
The ambient Kleber tree is $\virtual{T}(B^{n,2s})$ in type $A_{2n-1}^{(1)}$ and consists of a single node of weight $2s \clfw_n$.
Thus the only highest weight rigged configuration is the empty rigged configuration. Hence all columns of the highest weight 
tableau are filled with $[1, \dotsc, n]$, and the filling map is trivial on the highest weight tableau.
\end{proof}

\begin{remark}
\label{remark:fillmap_identity}
The filling map is actually the identity on \emph{all} elements of $B^{n,s}$.
\end{remark}


\subsection{Filling map for type $A_{2n-1}^{(2)}$}

The analysis of this type is  similar to type $D_n^{(1)}$.

\begin{lemma}
\label{lemma:hw_typeA2odd}
Let $B^{r,s}$ be a KR crystal of type $A_{2n-1}^{(2)}$. We have
\[
\RC(B^{r,s}) = \bigoplus_{\lambda} \RC(B^{r,s}; \lambda),
\]
where $\lambda$ is obtained by removing vertical dominoes from an $r\times s$ rectangle. Moreover, the highest weight rigged 
configuration in $\RC(B^{r,s}; \lambda)$ is given by
\[
\nu^{(a)} = \begin{cases}
\overline{\lambda}^{[r-a]} & 1 \leq a < r,\\
\overline{\lambda} & r \leq a < n, \\
\overline{\lambda}^{1/2} & a = n,
\end{cases} 
\]
with all riggings 0.
\end{lemma}

\begin{proof}
We split the proof into two cases: $r < n$ and $r = n$. We note that the virtualization of type $A_{2n-1}^{(2)}$ is into type $D_{n+1}^{(1)}$ 
and $\gamma_a = 1$ for all $a$.

\case{Case $r<n$}

\noindent
In this case we have $\virtual{B}^{r,s} = B^{r,s}$ of type $D_{n+1}^{(1)}$. Hence the ambient Kleber tree is exactly the same as the usual 
type $D_{n+1}^{(1)}$ Kleber tree, and we select all of the nodes for the virtual Kleber tree. The
only modification needed to Lemma~\ref{lemma:hw_typeD} is that $\nu^{(n)} = \virtual{\nu}^{(n)} = \virtual{\nu}^{(n+1)} = \overline{\lambda}^{1/2}$. 
Therefore we have the desired rigged configurations.

\case{Case $r = n$}

\noindent
We note that in this case we have $\virtual{B}^{n,s} = B^{n,s} \otimes B^{n+1,s}$ of type $D_{n+1}^{(1)}$. The roots we can subtract in the 
Kleber tree are of the form
\[
\alpha^{(k)} = k \alpha_{n+1} + k \alpha_n + (2k-1) \alpha_{n-1} + (2k-2) \alpha_{n-2} + \cdots + \alpha_{n+1-2k}
\]
by condition~(K2) of Definition~\ref{def:kleber_algorithm}. This sends the weight $\clfw_n + \clfw_{n+1}$ to $\clfw_{n-2k}$ We note that this 
corresponds to removing $k$ vertical dominoes. Now in building $\virtual{T}(B^{n,s})$, we get a sequence 
$\lfloor n / 2 \rfloor \geq k_1 \geq k_2 \geq \cdots \geq k_n > 0$ and this determines a unique dominant weight. (We note that this is the same 
as in Lemma~\ref{lemma:hw_typeD} where we consider $\clfw_n + \clfw_{n+1}$ as a full column of height $n$.)

It is easy to see that $\alpha^{(k)}$ satisfies the conditions of Definition~\ref{def:virtual_kleber}. Therefore all nodes of $\virtual{T}(B^{r,s})$ are 
constructed and selected, and from the description of $\alpha^{(k)}$, we have the desired rigged configurations.
\end{proof}

\begin{prop}
\label{prop:filling_A2odd}
Let $B^{r,s}$ be a KR crystal of type $A_{2n-1}^{(2)}$. Then
\[
\Phi = \fillmap \circ \iota^{-1}
\]
on highest weight elements with $\fillmap$ being the same as given in Definition~\ref{def:fill_D} and $\iota \colon B^{r,s} \to \RC(B^{r,s})$ is the 
natural crystal isomorphism.
\end{prop}

\begin{proof}
From Lemma~\ref{lemma:hw_typeA2odd}, we have the same highest weight rigged configurations given by 
Equation~\eqref{eq:highest_weight_typeD} except with $\nu^{(n)} = \virtual{\nu}^{(n)} = \virtual{\nu}^{(n+1)}$. Now in type $D_{n+1}^{(1)}$, the 
map $\delta_D$ selects the same singular string in $\virtual{\nu}^{(n)}$ and $\virtual{\nu}^{(n+1)}$ as $\delta$ (in type $A_{2n-1}^{(2)}$) for $\nu^{(n)}$. 
Therefore the return value of $\delta$ agrees with $\delta_D$ and the resulting rigged partitions agree (up to the equivalence of the spinor rigged 
partitions). Hence the proof of~\cite[Thm.~5.9]{OSS13} holds for type $A_{2n-1}^{(2)}$, and so $\Phi = \fillmap \circ \iota^{-1}$.
\end{proof}

\subsection{Filling map for type $B_n^{(1)}$}

\subsubsection{$r < n$}

We note that this is similar to type $D_n^{(1)}$.

\begin{lemma}
\label{lemma:hw_typeB}
Let $B^{r,s}$ be a KR crystal of type $B_n^{(1)}$ with $r < n$. We have
\[
\RC(B^{r,s}) = \bigoplus_{\lambda} \RC(B^{r,s}; \lambda),
\]
where $\lambda$ is obtained by removing vertical dominoes from the $r\times s$ rectangle. Moreover, the highest weight rigged 
configuration in $\RC(B^{r,s}; \lambda)$ is given by
\[
\nu^{(a)} = \begin{cases}
\overline{\lambda}^{[r-a]} & 1 \leq a < r, \\
\overline{\lambda} & r \leq a < n, \\
2\overline{\lambda}^{1/2} & a = n,
\end{cases} 
\]
with all riggings 0.
\end{lemma}

\begin{proof}
Recall that the KR crystal of type $B_n^{(1)}$ can be modeled by a virtual crystal of type $D_{n+1}^{(1)}$.
Note that as a virtual $D_{n+1}^{(1)}$ crystal, we have $\gamma_a = 2$ and so $\virtual{\nu}^{(a)} = 2 \nu^{(a)}$ for all $a < n$. 
Additionally $\gamma_n = 1$ and $\virtual{\nu}^{(n)} = \virtual{\nu}^{(n+1)} = \nu^{(n)}$ under $v$. Now from the proof of 
Lemma~\ref{lemma:hw_typeD} and condition~(V2) of the virtual Kleber tree, we must remove another vertical domino at each 
even step, so $2 \clfw_k$ goes to $2 \clfw_{k-2}$ for all $k$. Next since $\gamma_a = \gamma = 2$ for all $a \neq n$ 
by conditions~(A1) and~(A2) of Definition~\ref{def:virtual_kleber} we only select nodes at even levels, so this corresponds to selecting 
nodes by removing $2 \times 2$ boxes. Therefore when converting back to a $B_n^{(1)}$ rigged configuration, we get a removal of vertical 
dominoes. Hence the resulting highest weight rigged configuration is as desired after devirtualization.
\end{proof}

\begin{prop}
Let $B^{r,s}$ be a KR crystal of type $B_n^{(1)}$ and $r < n$. Then
\[
\Phi = \fillmap \circ \iota^{-1}
\]
on highest weight elements with $\fillmap$ being the same as given in Definition~\ref{def:fill_D} and $\iota \colon B^{r,s} \to \RC(B^{r,s})$ 
is the natural crystal isomorphism.
\end{prop}

\begin{proof}
From Lemma~\ref{lemma:hw_typeB}, we have $\virtual{\nu}^{(a)} = 2\nu^{(a)}$ for all $1 \leq a < n$ and 
$\nu^{(n)} = \virtual{\nu}^{(n)} = \virtual{\nu}^{(n+1)}$, so $(\virtual{\nu}, \virtual{J})$ is a highest weight rigged configuration in 
type $D_{n+1}^{(1)}$. Now in type $D_{n+1}^{(1)}$, the map $\delta_D$ selects the same singular strings 
in $\virtual{\nu}^{(n)}$ and $\virtual{\nu}^{(n+1)}$, and in type $B_n^{(1)}$, we only have even length strings and $\delta_B$ selects 
the corresponding single singular string. Therefore the return value of $\delta_B$ agrees with $\delta_D$ and 
the resulting rigged partitions agree (up to the equivalence of the spinor rigged partitions). Hence the proof of~\cite[Thm.~5.9]{OSS13} 
holds for type $B_n^{(1)}$, and so $\Phi = \fillmap \circ \iota^{-1}$.
\end{proof}

\subsubsection{$r = n$}

We are representing the spinor in this case with doubled columns as well. As such, the classical decomposition corresponding to removing 
$2 \times 2$ boxes as opposed to a vertical domino. Thus the proof is similar to type $C_n^{(1)}$, but by removing $2 \times 2$ boxes.

\begin{lemma}
\label{lemma:hw_typeB_spin}
Let $B^{n,s}$ be a KR crystal type $B_n^{(1)}$. We have
\[
\RC(B^{n,s}) = \bigoplus_{\lambda} \RC(B^{n,s}; \lambda)
\]
where $\lambda$ is obtained by removing vertical dominoes from an $n \times (s/2)$ rectangle. 
Moreover, the highest weight rigged configuration in $\RC(B^{n,s}; \lambda)$ is given by
\[
\nu^{(a)} = \begin{cases}
\overline{\lambda}^{[r-a]} & 1 \leq a < n, \\
2 \overline{\lambda}^{1/2} & a = n,
\end{cases} 
\]
with all riggings 0.
\end{lemma}

\begin{proof}
Recall the construction of the ambient Kleber tree, which is of type $D_{n+1}^{(1)}$, from Lemma~\ref{lemma:hw_typeA2odd} for $r = n$. 
However, here we must have $k_i = k_{i+1}$ for all $i \in 2\ZZ$ since $\gamma_a = 2$ for $a\neq n$
(for reasons similar to the type $C_n^{(1)}$ case given in Lemma~\ref{lemma:hw_typeC}). 
Moreover we only select nodes in $\virtual{T}(B^{r,s})$ in the even levels. Therefore after devirtualization (note that $\gamma_n = 1$, which 
implies the factor of $2$ for $\nu^{(n)}$), we get the desired rigged configurations.
\end{proof}

\begin{prop}
Let $B^{n,s}$ be a KR crystal of type $B_n^{(1)}$. Then
\[
\Phi = \fillmap \circ \iota^{-1}
\]
on highest weight elements with $\fillmap$ the same as given in Definition~\ref{def:fill_C} for double columns and $\iota \colon B^{n,s} \to \RC(B^{n,s})$ 
is the natural crystal isomorphism.
\end{prop}

\begin{proof}
We first recall the doubling map for $B^{n,s}$ in the paragraph after Equation~\eqref{eq:doubling_map}, which is a virtualization map into $B^{n,s}$ of type $A_{2n-1}^{(2)}$ with $\gamma_a = 2$ for all $a \neq n$ and $\gamma_n = 1$, so
\[
\virtual{\nu}^{(a)} = \begin{cases} 2\nu^{(a)} & a < n, \\ \nu^{(a)} & a = n, \end{cases}
\]
and similarly for the riggings. Thus the removal of vertical dominoes translates into removing $2 \times 2$ boxes. Recall that the bijection for $B^{n,s}$ of type $B_n^{(1)}$ is given by
first applying the doubling map, then following the algorithm given in Definition~\ref{def:fill_D}, and then taking the halving map, which on the KR 
tableaux is the identity map. Since we are removing $2 \times 2$ boxes, we must have $c = -1$. Thus from Proposition~\ref{prop:filling_A2odd} 
(or the proof of Theorem~5.9 in~\cite{OSS13}), we see that $\Phi = \fillmap \circ \iota^{-1}$.
\end{proof}

We note that our convention choice is visible here. Specifically, if we used half width boxes then we would have $\nu^{(n)} = \overline{\lambda}^{1/2}$ 
and $\virtual{\nu}^{(n)} = 2\nu^{(n)}$.

\subsection{Filling map for type $A_{2n}^{(2)}$}

\begin{lemma}
\label{lemma:hw_typeA2_even}
Consider type $A_{2n}^{(2)}$ and $B^{r,s}$ be a KR crystal. We have
\[
\RC(B^{r,s}) = \bigoplus_{\lambda} \RC(B^{r,s}; \lambda),
\]
where $\lambda$ is obtained by removing single boxes from an $r \times s$ rectangle. 
Moreover, the highest weight rigged configuration in $\RC(B^{r,s}; \lambda)$ is given by
\[
\nu^{(a)} = \begin{cases}
\overline{\lambda}^{[r-a]} & 1 \leq a < r, \\
\overline{\lambda} & r \leq a \leq n,
\end{cases} 
\]
with all riggings 0.
\end{lemma}

\begin{proof}
Similar to Lemma~\ref{lemma:hw_typeC} except we select all nodes in the ambient Kleber tree.
\end{proof}

\begin{dfn}
\label{def:fill_A2even}
The crystal morphism $\fillmap \colon B^{r,s} \to T^{r,s}$ is as in Definition~\ref{def:fill_D} except the final column for step~(\ref{step:final}) 
(when $c > -1$) is $\column{1, \dotsc, x-1, \emptyset, \dotsc, \emptyset}$ (recall $\emptyset$ is the unique element in $B(0) \subseteq B^{1,1}$).
\end{dfn}

\begin{ex}
Consider $\lambda = 2\clfw_1 + \clfw_2 + 2\clfw_3 + \clfw_4$ for type $A_{20}^{(2)}$ and $B^{4,6}$. Then we have
\[
\fillmap(u_{\lambda}) = 
\begin{tikzpicture}[baseline]
\matrix [matrix of math nodes,column sep=-.4, row sep=-.5,text height=8,text width=8,align=center,inner sep=3,nodes={anchor=south,minimum height=15}] 
{
\node[draw]{4}; & \node[draw,fill=gray!30]{\overline{4}}; & \node[draw,fill=gray!30]{4}; & \node[draw,fill=gray!30]{\overline{3}}; & \node[draw,fill=gray!30]{\overline{4}}; & \node[draw,fill=gray!30]{\emptyset}; \\
\node[draw]{3}; & \node[draw]{3}; & \node[draw]{3}; & \node[draw,fill=gray!30]{\overline{4}}; & \node[draw,fill=gray!30]{4}; & \node[draw,fill=gray!30]{\emptyset}; \\
\node[draw]{2}; & \node[draw]{2}; & \node[draw]{2}; & \node[draw]{4}; & \node[draw,fill=gray!30]{3}; & \node[draw,fill=gray!30]{2}; \\
\node[draw]{1}; & \node[draw]{1}; & \node[draw]{1}; & \node[draw]{1}; & \node[draw]{1}; & \node[draw]{1}; \\
};
\end{tikzpicture}
\ .
\]
\end{ex}

We recall some pertinent facts about $\delta$ in type $A_{2n}^{(2)}$. Specifically for $\nu^{(n)}$, if the selected singular string 
has length 1, we terminate and return $\emptyset$, otherwise we remove 2 boxes from the selected string. In all other cases,
it behaves as in type $C_n^{(1)}$.

\begin{prop}
\label{prop:morphism_A2even}
Let $B^{r,s}$ be a KR crystal of type $A_{2n}^{(2)}$. Then
\[
\Phi = \fillmap \circ \iota^{-1}
\]
on highest weight elements with $\fillmap$ as in Definition~\ref{def:fill_A2even} and $\iota \colon B^{r,s} \to \RC(B^{r,s})$ is the natural crystal isomorphism.
\end{prop}

\begin{proof}
This is the similar to the proof as type $D_n^{(1)}$ given in~\cite{OSS13}, but we must make the following changes. When we remove pairs of 
columns (Step~1 in~\cite{OSS13}), this behaves as in the proof of Proposition~\ref{prop:crystal_iso_C} (type $C_n^{(1)}$ case). When there is a 
single column remaining (i.e. we are in the final step, Step~3, and $k_c = 1$ with all other $k_i = 0$), the resulting rigged configuration 
is given by Lemma~\ref{lemma:hw_typeA2_even} with $\lambda = \Lambda_x$ and $s = 1$, where $x$ in given by the algorithm for 
Definition~\ref{def:fill_D}. Here $\delta$ starts at $\nu^{(r)}$ and goes to $\nu^{(n)}$. Since 
$\nu^{(n)}$ consists only of a single column with a singular string, the map $\delta$ returns $\emptyset$.
Thus we remove a single box from each 
$\nu^{(k)}$ for $r \leq k \leq n$. This can be repeated this $r-x$ times, at which point we obtain the empty rigged configuration. Hence we obtain the 
final column as $\column{1, \dotsc, x-1, \emptyset, \dotsc, \emptyset}$.
\end{proof}

\subsection{Filling map for type $A_{2n}^{(2)\dagger}$}

This is the same as type $C_n^{(1)}$ except at $r = n$, in which case we have $v \colon B^{n,s} \to B^{n,s} \otimes B^{n,s}$ under the virtualization 
map into type $A_{2n-1}^{(1)}$. However this makes the behavior uniform with the proof for type $C_n^{(1)}$ for $r < n$.

\begin{lemma}
\label{lemma:hw_typeA2even_dual}
Let $B^{r,s}$ be a KR crystal of type $A_{2n}^{(2)\dagger}$. We have
\[
\RC(B^{r,s}) = \bigoplus_{\lambda} \RC(B^{r,s}; \lambda),
\]
where $\lambda$ is obtained by removing horizontal dominoes from an $r \times s$ rectangle. Moreover, the highest weight rigged configuration 
in $\RC(B^{r,s}; \lambda)$ is given by
\[
\nu^{(a)} = \begin{cases}
\overline{\lambda}^{[r-a]} & 1 \leq a < r, \\
\overline{\lambda} & r \leq a \leq n,
\end{cases} 
\]
with all riggings 0.
\end{lemma}

\begin{proof}
Same as Lemma~\ref{lemma:hw_typeC} for $r < n$. For $r = n$, we have $\virtual{B}^{n,s} = B^{n,s} \otimes B^{n,s}$ in type $A_{2n-1}^{(1)}$. 
So a similar proof as Lemma~\ref{lemma:hw_typeC} holds here.
\end{proof}

\begin{prop}
Let $B^{r,s}$ be a KR crystal of type $A_{2n}^{(2)\dagger}$. Then
\[
\Phi = \fillmap \circ \iota^{-1}
\]
on highest weight elements with $\fillmap$ is the same as given in Definition~\ref{def:fill_C} and $\iota \colon B^{r,s} \to \RC(B^{r,s})$ is the 
natural crystal isomorphism.
\end{prop}

\begin{proof}
The proof is similar to Proposition~\ref{prop:crystal_iso_C}.
\end{proof}

\subsection{Filling map for type $D_{n+1}^{(2)}$}

We note that in this case, the filling map is similar to type $A_{2n}^{(2)}$ except for $r = n$.

\subsubsection{$r < n$}

\begin{lemma}
\label{lemma:hw_typeD_twisted}
Let $B^{r,s}$ be a KR crystal type $D_{n+1}^{(2)}$ for $r < n$. We have
\[
\RC(B^{r,s}) = \bigoplus_{\lambda} \RC(B^{r,s}; \lambda),
\]
where $\lambda$ is obtained by removing single boxes from an $r \times s$ rectangle.
Moreover, the highest weight rigged configuration in $\RC(B^{r,s}; \lambda)$ is given by
\[
\nu^{(a)} = \begin{cases}
\overline{\lambda}^{[r-a]} & 1 \leq a < r, \\
\overline{\lambda} & r \leq a \leq n,
\end{cases} 
\]
with all riggings 0.
\end{lemma}

\begin{proof}
Similar to Lemma~\ref{lemma:hw_typeC} except we select all nodes in the ambient Kleber tree.
\end{proof}

\begin{prop}
Let $B^{r,s}$ be a KR crystal of type $D_{n+1}^{(2)}$ and $r < n$. Then
\[
\Phi = \fillmap \circ \iota^{-1}
\]
on highest weight elements with $\fillmap$ as in Definition~\ref{def:fill_A2even} and $\iota \colon B^{r,s} \to \RC(B^{r,s})$ is the natural crystal isomorphism.
\end{prop}

\begin{proof}
Similar to Proposition~\ref{prop:morphism_A2even} as all vacancy numbers and riggings of $\nu^{(n)}$ are 0 (i.e. there are no 
quasisingular strings so we cannot go into case~(Q) when performing $\delta$~\cite{OSS03}).
\end{proof}

\subsubsection{$r = n$}

This is similar to $C_n^{(1)}$ when $r = n$.

\begin{prop}
Let $B^{n,s}$ be a KR crystal of type $D_{n+1}^{(2)}$. Then
\[
\Phi = \fillmap \circ \iota^{-1}
\]
on highest weight elements with $\fillmap$ the trivial filling map (i.e. the identity map on the unique highest weight tableau) 
and $\iota \colon B^{r,s} \to \RC(B^{r,s})$ is the natural crystal isomorphism.
\end{prop}

\begin{proof}
This is the same as Proposition~\ref{prop:crystal_iso_Cn}.
\end{proof}

Moreover $\fillmap$ is the identity map on the tableaux as noted in Remark~\ref{remark:fillmap_identity}.

\subsection{Summary}

We have given an explicit description of the highest weight rigged configurations in all (non-exceptional) types for single tensor factors. 
Furthermore, we have shown the following.
\begin{thm}
\label{thm:isomorphism}
Let $\g$ be a non-exceptional affine type. We have
\[
\Phi = \fillmap \circ \iota^{-1}
\]
on highest weight elements in $\RC(B^{r,s})$ and $\iota \colon B^{r,s} \to \RC(B^{r,s})$ is the natural crystal isomorphism.
\end{thm}
In other words, as classical crystals (and hence as sets) $\RC(B^{r,s}) \iso T^{r,s} \iso B^{r,s}$ and classically highest weight elements are 
mapped by $\Phi$ and $\fillmap$, respectively. We also note that the filling map for general $r$ depends only on how the affine 
node attaches to the classical type, analogous to the classical decompositions.

We can define $\ls$ and $\lt$ on (a tensor product of) $T^{r,s}$ by splitting off the left column (of the leftmost factor) and the top box (of the 
leftmost factor), respectively.

Moreover we can show that $\iota$ sends cocharge to energy.
\begin{thm}
\label{thm:statistics}
Let $B^{r,s}$ be a KR crystal of non-exceptional type with $s\ge 1$ and $1\le r\le n$.
For all $b \in B^{r,s}$ we have 
\[
D(b) = \cc\bigl( \theta \circ \iota(b) \bigr).
\]
\end{thm}

\begin{proof}
Our proof is similar to~\cite[Thm.~4.10]{OSS13}. Since the energy function is constant on classical components, and by 
Proposition~\ref{prop:constant_cocharge}, cocharge is as well, it suffices to prove the statement for highest weight elements $b\in B^{r,s}$ with 
the unique weight $\lambda_b$. Since all riggings and vacancy numbers are 0 for highest weight elements in $\RC(B^{r,s})$, the map $\theta$ 
is the identity. We rewrite the cocharge in terms of the vacancy numbers:
\begin{equation}
\label{eq:cocharge_vac}
\cc(\nu) = \frac{1}{2} \sum_{(a,i)\in \HH_0} t_a^{\vee} p_i^{(a)} m_i^{(a)} + \frac{1}{2} 
\sum_{\substack{ a \in I_0 \\ i,j \in \ZZ_{>0} }} t_a^{\vee} \min(i,j) L_i^{(a)} m_j^{(a)}.
\end{equation}
Since the vacancy numbers are zero, the first term of Equation~\eqref{eq:cocharge_vac} is 0. Also in our case $L_i^{(a)} = \delta_{a,r} \delta_{i,s}$, 
and so for the unique highest weight rigged configuration $(\nu, J)$ of weight $\lambda_b$, we have
\begin{equation}
\label{eq:cocharge_vac_nums}
\cc(\nu, J) = \cc(\nu) = \frac{1}{2} \sum_{j \in \ZZ_{>0}} t_r^{\vee} \min(s,j) m_j^{(r)} = \frac{t_r^{\vee}}{2} \absval{\nu^{(r)}}.
\end{equation}

We note that for $r \neq n$, we have $\nu^{(r)} = \overline{\lambda}$ by Lemma~\ref{lemma:hw_typeC} and equivalent lemmas in 
Section~\ref{sec:filling_map} (depending on the type). Furthermore, $t_r^{\vee} = 2$ in types $A_{2n}^{(2)}$ and $D_{n+1}^{(2)}$, both of which 
have a classical decomposition given by removing single boxes. In all other types we have $t_r^{\vee} = 1$ and we are removing dominoes 
(or no boxes are removed) to obtain the classical decomposition. By Definition~\ref{def:energy}, this shows the desired claim.

Now consider $r = n$. We have $t_n^{\vee} = 2$ in types $A_{2n}^{(2)}$ and $A_{2n-1}^{(2)}$, and $t_n^{\vee} = 1$ in all other types. For 
type $A_{2n}^{(2)}$, we are removing single boxes and $\nu^{(n)} = \overline{\lambda}$. For type $A_{2n}^{(2)\dagger}$, we are removing 
horizontal dominoes and $\nu^{(n)} = \overline{\lambda}$. For type $B_n^{(1)}$, we are removing vertical dominoes and 
$\nu^{(n)} = 2 \overline{\lambda}$, in particular $\absval{\overline{\lambda}} = 2 \absval{\nu^{(n)}}$. For type $A_{2n-1}^{(2)}$, we are 
removing vertical dominoes and $\nu^{(n)} = \overline{\lambda}^{1/2}$, in particular $\absval{\overline{\lambda}} = \frac{1}{2} \absval{\nu^{(n)}}$. 
In all other types the classical decomposition is trivial (i.e., no boxes are removed). From Equation~\eqref{eq:cocharge_vac_nums} and 
Definition~\ref{def:energy}, this shows the desired claim.
\end{proof}

From Theorem~\ref{thm:isomorphism}, we can give an affine crystal structure on rigged configurations and KR tableaux. We do so by 
mapping to the KN tableaux model, where we know how to explicitly compute $e_0$ and $f_0$ by~\cite{FOS09}, under the natural (classical) crystal 
isomorphism $\iota$ and mapping back. Together with Theorem~\ref{thm:statistics}, this implies that Conjecture~\ref{conj:bijection} holds for $B^{r,s}$ 
in non-exceptional types on classically highest weight elements.

However this definition of the affine crystal structure is somewhat unsatisfactory as it is not a direct description of $e_0$ and $f_0$ 
on rigged configurations nor on KR tableaux in general. In the next section, we will give an explicit description of $e_0$ and 
$f_0$ on rigged configurations of types $B_n^{(1)}$ and $A_{2n-1}^{(2)}$ for $B^{r,s}$ where $r < n$. For type $A_n^{(1)}$ this was 
done in~\cite{SW10} (for general factors) and for type $D_n^{(1)}$ in~\cite{OSS13}.

We also have the following conjecture related to Conjecture~\ref{conj:bijection} about the filling map for arbitrary number of factors.

\begin{conj}
\label{conj:filling_map}
Let $B = \bigotimes_{i=1}^N B^{r_i, s_i}$ and let $T = \bigotimes_{i=1}^N (B^{1,1})^{\otimes r_i s_i}$ (organized into a $r_i \times s_i$ 
rectangle as with $T^{r_i,s_i}$), then the filling map $\fillmap \colon B \to T$ is given by
\[
	\fillmap(B) = \bigotimes_{i=1}^N \fillmap(B^{r_i, s_i}).
\]
\end{conj}

In other words, we have $T = \bigotimes_{i=1}^N T^{r_i,s_i}$. This has been verified by computer for tensor products for non-exceptional 
types up to rank 4, up to 2 factors, and $s \leq 2$.

\section{Affine crystal strucutre}
\label{sec:affine_structure}

In this section we give the explicit affine crystal structure for $\RC(B^{r,s})$ for all $1\le r\le n$ of type $B_n^{(1)}$ and $A_{2n-1}^{(2)}$.
In addition, we show that $B^{r,s}$ of type $B_n^{(1)}$ and $A_{2n-1}^{(2)}$ virtualizes into $\virtual{B}^{r,s}$ of type $D_{n+1}^{(1)}$
for $r<n$. This proves~\cite[Conj. 3.7]{OSS03II} (see also Conjecture~\ref{conj:KR_virtualization}) in these cases.

\subsection{Affine crystal operators}
Here we give an explicit description of the affine crystal operators $e_0$ and $f_0$ on rigged configurations for types $B_n^{(1)}$ and 
$A_{2n-1}^{(2)}$. In general for types $D_n^{(1)}$, $B_n^{(1)}$, and $A_{2n-1}^{(2)}$, we define the affine crystal operators by
\begin{subequations}
\label{eq:affine_operators_diagram_map}
\begin{align}
e_0 & = \sigma \circ e_1 \circ \sigma\;,
\\ f_0 & = \sigma \circ f_1 \circ \sigma\;,
\end{align}
\end{subequations}
where $\sigma$ is the crystal involution that is induced by the Dynkin diagram automorhpism which interchanges node $0$ and node $1$. 
Therefore to describe $e_0$ and $f_0$, we need to define the crystal automorphism $\sigma$. This is done by first defining the
map on $\{2,3,\ldots,n\}$-highest weight elements, which are in bijection with so-called $\pm$-diagrams and then extending to all
crystal elements.

A \emph{$\pm$-diagram} is a sequence of shapes $\tau \subseteq \mu \subseteq \lambda$ such that $\lambda / \mu$ and $\mu / \tau$ 
are horizontal strips (i.e. every column contains at most one box). We depict this as a skew shape $\lambda / \tau$ in which the cells 
of $\mu / \tau$ are filled with a $+$ and $\lambda / \mu$ are filled with a $-$. The partitions $\lambda$ and $\tau$ are called the outer 
and inner shapes, respectively.
In type $B_n$, the $\pm$-diagrams with columns of height $n$ can also contain at most one $0$ between a $+$ and $-$ at height $n$, 
can have at most one half-width spin column of height $n$ with either a $+$ or a $-$, and must have all columns of height $n$ being 
non-empty. In type $C_n$, there are no empty columns of height $n$. We will only consider $\pm$-diagrams in type $D_n$ whose outer 
shape does not contain any columns of height $n-1$ nor $n$.

\begin{prop}[{\cite{S08}, \cite[Sec.~3.2]{FOS09}}]
\label{prop:pm_diagram_classification}
Let $\g_0$ be of type $B_n,C_n$, or $D_n$.
There is a bijection $\zeta$ from $\pm$-diagrams of outer shape $\lambda$ to $\{2, \dotsc, n\}$-highest weight elements in the highest 
weight crystal $B(\lambda)$ of type $\g_0$. The $\pm$-diagram which has a $+$ in every column and no $-$ corresponds to the highest 
weight vector. Given a $\pm$-diagram $P$, we obtain the $\{2, \dotsc, n\}$-highest weight element $\zeta(P) = b$ inductively as follows:
\begin{itemize}
\item[Case 1:] $P$ has a column where a $+$ can be added. Let $P'$ be the $\pm$-diagram obtained from $P$ by adding a $+$ in the 
rightmost possible column at height $h$. If there is a column of height $n$ containing $0$, then $b = f_1 f_2 \cdots f_h f_1 f_2 \cdots f_n \zeta(P')$,
otherwise $b = f_1 f_2 \cdots f_h \zeta(P')$. Note that we cannot add a $+$ to a spin column.
\item[Case 2:] $P$ has no column where a $+$ can be added and at least one $-$. Let $P'$ be the $\pm$-diagram obtained from $P$ by 
removing the leftmost $-$ at height $h$ and either moving the $+$ in the same column up if $h > 1$ or adding a $+$ if $h = 1$. Then 
\[
b = \begin{cases}
f_1 f_2 \cdots f_n f_{n-2} f_{n-3} \cdots f_h \zeta(P') & \g_0 = D_n,\\ 
f_1 f_2 \cdots f_{n-1} f_n \zeta(P') & \g_0 = B_n \text{ and $-$ is in the spin column},\\ 
f_1 f_2 \cdots f_{n-1} f_n f_n f_{n-1} \cdots f_h \zeta(P') & \g_0 = B_n \text{ otherwise},\\ 
f_1 f_2 \cdots f_{n-1} f_n f_{n-1} f_{n-2} \cdots f_h \zeta(P') & \g_0 = C_n.
\end{cases}
\]
\end{itemize}
\end{prop}

Next we recall the bijection $\zeta_{rc}$ from $\pm$-diagrams to rigged configurations given in~\cite{OSS13} for type $D_n^{(1)}$. 
Consider the classically highest weight component $\RC(B^{r,s}; \lambda)$. We construct all $\{2,\dotsc,n\}$-highest weight rigged 
configurations in $\RC(B^{r,s}; \lambda)$ from the $\pm$-diagrams of outer shape $\lambda$ as follows (note they are in bijection). 
Consider a single column $\pm$-diagram $P$ of height $x$, and let $y = r - x$ (this will always be even). We describe the rigged 
configuration based on which type of column $P$ is:
\begin{itemize}
\item $P$ does not contain any sign:
\begin{equation}
\label{equation.rc1}
\begin{split}
\nu & = \bigl( \phantom{-} \overbrace{(1), (1), \dotsc, (1)}^x, \overbrace{(1), (1^2), \dotsc, (1^y)}^y, (1^y), \dotsc, (1^y), (1^{\frac{y}{2}}), (1^{\frac{y}{2}}) \bigr),
\\ J & = \bigl( \overbrace{(-1), (0), \dotsc, (0)}^x, \overbrace{(1), (0^2), \dotsc, (0^y)}^y, (0^y), \dotsc, (0^y), (0^{\frac{y}{2}}), (0^{\frac{y}{2}}) \bigr).
\end{split}
\end{equation}

\item $P$ contains $+$:
\begin{equation}
\label{equation.rc2}
\begin{split}
\nu & = \bigl( \overbrace{\emptyset, \emptyset, \dotsc, \emptyset}^x, \overbrace{(1), (1^2), \dotsc, (1^y)}^y, (1^y), \dotsc, (1^y), (1^{\frac{y}{2}}), 
(1^{\frac{y}{2}}) \bigr),
\\ J & = \bigl( \overbrace{\emptyset, \emptyset, \dotsc, \emptyset}^x, \overbrace{(0), (0^2), \dotsc, (0^y)}^y, (0^y), \dotsc, (0^y), (0^{\frac{y}{2}}), 
(0^{\frac{y}{2}}) \bigr).
\end{split}
\end{equation}

\item $P$ contains $-$:
\begin{equation}
\label{equation.rc3}
\begin{split}
\nu & = \bigl( \phantom{-} \overbrace{(2), (2), \dotsc, (2)}^{x-1}, \overbrace{(1^2), (1^3), \dotsc, (1^{y+2})}^{y+1}, (1^{y+2}), \dotsc, (1^{y+2}), 
(1^{\frac{y+2}{2}}), (1^{\frac{y+2}{2}}) \bigr),
\\ J & = \bigl( \overbrace{(-2), (0), \dotsc, (0)}^{x-1}, \overbrace{(0^2), (0^3), \dotsc, (0^{y+2})}^{y+1}, (0^{y+2}), \dotsc, (0^{y+2}), (0^{\frac{y+2}{2}}), 
(0^{\frac{y+2}{2}}) \bigr),
\end{split}
\end{equation}
except when $x = 1$, where we take $(\nu, J)^{(1)} = \bigl((1,1), (-1, -1)\bigr)$.

\item $P$ contains $\pm$:
\begin{equation}
\label{equation.rc4}
\begin{split}
\nu & = \bigl( \phantom{-} \overbrace{(1), (1), \dotsc, (1)}^{x-1}, \overbrace{(1^2), (1^3), \dotsc, (1^{y+2})}^{y+1}, (1^{y+2}), \dotsc, (1^{y+2}), 
(1^{\frac{y+2}{2}}), (1^{\frac{y+2}{2}}) \bigr),
\\ J & = \bigl( \overbrace{(-1), (0), \dotsc, (0)}^{x-1}, \overbrace{(0^2), (0^3), \dotsc, (0^{y+2})}^{y+1}, (0^{y+2}), \dotsc, (0^{y+2}), (0^{\frac{y+2}{2}}), 
(0^{\frac{y+2}{2}}) \bigr).
\end{split}
\end{equation}
\end{itemize}

An arbitrary $\pm$-diagram $P$ is the concatenation of columns described above. The corresponding rigged configurations is
obtained by summing together all partitions (padding with $0$ as necessary) and riggings over all columns of $P$. We can invert this map 
as follows. Fix a $\pm$-diagram $P$. Let $c_{\bullet}(h)$, $c_+(h)$, $c_-(h)$, and $c_{\pm}(h)$ denote the number of columns of $P$ with 
outer height $h$ with no sign, $+$, $-$, and $\pm$ respectively. These values uniquely determine the $\pm$-diagram and can be computed 
(inductively) from $h = 0$ ($r$ even) or $h = 1$ ($r$ odd) to $h = r$ as follows:
\begin{align*}
c_{\bullet}(h) & = \begin{cases}
  J^{(h+1)}_1 + \delta_{h0} \nu_1^{(1)} & 0 \leq h < r \\
  \nu_1^{(r)} - \nu_1^{(r+1)} & h = r
\end{cases}
\\ c_+(h) & = \nu_1^{(h+1)} - \nu_1^{(h)} \qquad 1\le h <r
\\ c_-(h) & = \begin{cases}
  \nu_2^{(1)} & h = 1 \\
  \nu_1^{(h-1)} - \nu_1^{(h)} & 1 < h \leq r
\end{cases}
\\ c_{\pm}(h) & = \sum_{j=1}^2 \bigl( \nu_j^{(h)} - \nu_j^{(h-1)} \bigr) - \bigl( c_{\bullet}(h-2) + c_+(h-2) \bigr)
\end{align*}
where we set $c_+(0) = 0$ and $\delta_{h0}$ is the Kronecker delta. Note that $c_+(r)$ is not determined by the above formula, but
rather by the fat that the total number of columns is $s$.

\begin{prop}[\cite{OSS13} Proposition~4.3]
\label{prop:rc_pm_bijection_typeD}
Let $\zeta$ be the map from $\pm$-diagrams to $\{2,\dotsc,n\}$-highest weight elements in $B^{r,s}$ in type $D_n^{(1)}$. Then we have
\[
\zeta_{rc} = \iota \circ \zeta.
\]
\end{prop}

In order to define the diagram involution map, we now need an involution on $\pm$-diagrams.

\begin{dfn}[\cite{S08}]
\label{def:pm_map}
Let $P$ be a $\pm$-diagram of outer shape $\Lambda$, where the columns of $\Lambda$ are either 
all even or all odd height. Then $\mathfrak{S}(P)$ is the $\pm$-diagram, where compared to $P$ the values $c_+(h)$ and $c_-(h)$ are 
interchanged for $r \ge h \geq 1$, and the values of $c_{\bullet}(h-2)$ and $c_{\pm}(h)$ are interchanged for $r \ge h \geq 2$.
\end{dfn}

We can now define the diagram involution on $\RC(B^{r,s})$.

\begin{dfn}
\label{def:promotion_type_D}
Let $(\nu,J) \in \RC(B^{r,s})$ with $B^{r,s}$ a KR crystal of type $D_n^{(1)}$ with $1 \leq r \leq n-2$. Choose a sequence $b = (b_1, b_2, \dotsc, b_k)$ 
with $b_i \in \{2, \dotsc, n\}$ such that $e_b(\nu, J) := e_{b_1} \cdots e_{b_k}(\nu, J)$ is $\{2, \dotsc, n\}$-highest weight. Then define
\begin{equation}
\label{eq:diagram_map}
\sigma_{rc}(\nu, J) = f_{b^r} \circ \zeta_{rc} \circ \mathfrak{S} \circ \zeta_{rc}^{-1} \circ e_b,
\end{equation}
where $b^r$ is the reverse of $b$.
\end{dfn}

\begin{thm}[{\cite[Thm.~4.9]{OSS13}}]
\label{thm:affine_rc_typeD}
Let $B^{r,s}$ be a KR crystal of type $D_n^{(1)}$ with $1\le r \le  n-2$. Then $\RC(B^{r,s})$ is a $U_q^{\prime}(\g)$-crystal with
\begin{align*}
e_0 & = \sigma_{rc} \circ e_1 \circ \sigma_{rc},
\\ f_0 & = \sigma_{rc} \circ f_1 \circ \sigma_{rc}.
\end{align*}
Moreover the natural classical crystal isomorphism $\iota$ is an affine crystal isomorphism.
\end{thm}

Now we show an analogous result to Proposition~\ref{prop:rc_pm_bijection_typeD} for types $B_n^{(1)}$ with $r < n$ and $A_{2n-1}^{(2)}$ 
for all $r\le n$. We will use this to show the analogous result to Theorem~\ref{thm:affine_rc_typeD} for $r < n$.

Consider $\RC(B^{r,s})$ of type $B_n^{(1)}$ for $r < n$ or type $A_{2n-1}^{(2)}$ for $r \le n$. We begin by showing that $\pm$-diagrams are 
in bijection with $\{2, \dotsc, n\}$-highest weight rigged configurations. 
By inspection of~\eqref{equation.rc1}--\eqref{equation.rc4} observe that the $n$-th and $(n+1)$-th rigged partition in the $\{2,3,\ldots,n+1\}$-highest 
weight rigged configurations of type $D_{n+1}^{(1)}$ are equal. Hence we can define $\zeta_{rc}$ in type $A_{2n-1}^{(2)}$ similar to type 
$D_{n+1}^{(1)}$ except we identify the last two rigged partitions (which in effect drops $\nu^{(n+1)}$).
For type $B_n^{(1)}$, we define $\zeta_{rc}$ by also identifying the last two rigged partitions of type $D_{n+1}^{(1)}$ (which in effect drops $\nu^{(n+1)}$) 
in addition to doubling $\nu^{(n)}$ to keep with our convention. For example, if a column in the $\pm$-diagram does not contain any sign, in 
type $A_{2n-1}^{(2)}$ we add
\begin{align*}
\nu & = \bigl( \phantom{-} \overbrace{(1), (1), \dotsc, (1)}^x, \overbrace{(1), (1^2), \dotsc, (1^y)}^y, (1^y), \dotsc, (1^y), (1^{\frac{y}{2}}) \bigr),
\\ J & = \bigl( \overbrace{(-1), (0), \dotsc, (0)}^x, \overbrace{(1), (0^2), \dotsc, (0^y)}^y, (0^y), \dotsc, (0^y), (0^{\frac{y}{2}}) \bigr),
\end{align*}
and in type $B_{n}^{(1)}$ we add
\begin{align*}
\nu & = \bigl( \phantom{-} \overbrace{(1), (1), \dotsc, (1)}^x, \overbrace{(1), (1^2), \dotsc, (1^y)}^y, (1^y), \dotsc, (1^y), (2^{\frac{y}{2}}) \bigr),
\\ J & = \bigl( \overbrace{(-1), (0), \dotsc, (0)}^x, \overbrace{(1), (0^2), \dotsc, (0^y)}^y, (0^y), \dotsc, (0^y), (0^{\frac{y}{2}}) \bigr).
\end{align*}

\begin{prop}
\label{prop:rc_pm_bijection}
Let $\zeta$ be the map from $\pm$-diagrams to $\{2,\dotsc,n\}$-highest weight elements in $B^{r,s}$ in type $A_{2n-1}^{(2)}$ for all 
$r$ or type $B_n^{(1)}$ for $r < n$. Then we have
\[
\zeta_{rc} = \iota \circ \zeta.
\]
\end{prop}

\begin{proof}
The proof is similar to~\cite[Prop.~4.3]{OSS13}. We note that in type $B_n^{(1)}$ from Lemma~\ref{lemma:act_same_string}, $f_n^2$ acts on 
the same string and keeps $m_i^{(n)} = 0$ for all $i \notin 2\ZZ$.
\end{proof}

Next to extend this to $r = n$ for $B_n^{(1)}$, we need the following lemma and need to reformulate the doubling map given 
in the definition of $\Phi$ in Section~\ref{sec:background} (in the paragraph after Equation~\eqref{eq:doubling_map}) from type 
$B_n^{(1)}$ to $A_{2n-1}^{(2)}$ as a classical virtualization map with $\gamma_r = 2 - \delta_{rn}$ 
and the trivial folding $\phi(r) = r$ for all $r \in I_0$.

\begin{lemma}[{\cite[Lemma~3.5]{FOS09}}]
\label{lemma:double_pm_diagrams}
Let $d \colon B^{n,s} \to \virtual{B}^{n,s}$ denote the doubling map from $B_n^{(1)} \to A_{2n-1}^{(2)}$. Let $\lambda = \sum_{i=1}^n k_i \clfw_i$
be a classical weight of type $B_n$, let $b \in B(\lambda) \subseteq B^{n,s}$ be a $\{2, \dotsc, n\}$-highest weight element,
and $P$ be the corresponding $\pm$-diagram. The $\pm$-diagram corresponding 
to $\virtual{b}$ in $B\bigl( \Psi(\lambda) \bigr) \subseteq \virtual{B}^{n,s}$ 
is obtained by doubling each column of $P$ together with its signs for non-spin columns. For a spin 
column, it becomes a usual full width column with the same sign. For a column with $0$, we replace it with a column containing a $+$ and 
a column containing a $-$.
\end{lemma}

We recall that the $\pm$-diagrams fit inside a $n \times (s/2)$ box with possibly one half-width spin column. Therefore we need to 
describe a map from a column of height $n$ containing a $0$ or a spin column to rigged configurations.
If  it is a spin column with a $-$ or a full column containing a 0, we add
\begin{align*}
\nu & = \bigl( \phantom{-}(1), (1), (1), \dotsc, (1) \bigr),
\\ J & = \bigl( (-1), (0), (0), \dotsc, (0) \bigr).
\end{align*}
For a spin column with a $+$, we do not add anything. Thus we have the following.
\begin{prop}
\label{prop:rc_pm_doubling_commute}
Let $d_{rc} \colon \RC(B^{n,s}) \to \RC(B^{n,s})$ be the doubling map from type $B_n^{(1)}$ to type $A_{2n-1}^{(2)}$ defined in 
Section~\ref{sec:background} (equivalently by Equation~\eqref{eq:rc_virtualization_map} with the scaling factors given above). 
Let $d_{\pm}$ denote the doubling map defined on $\pm$-digrams given by Lemma~\ref{lemma:double_pm_diagrams}.
Then we have
\[
\zeta_{rc} \circ d_{\pm} = d_{rc} \circ \zeta_{rc}.
\]
\end{prop}

\begin{proof}
This follows from the definition of $d_{rc}$, $d_{\pm}$, and $\zeta_{rc}$.
\end{proof}

We therefore can extend Proposition~\ref{prop:rc_pm_bijection} to $r = n$.
\begin{prop}
\label{prop:rc_pm_spinor_bijection}
Let $\zeta$ be the bijection from $\{2,\dotsc,n\}$-highest weight elements in $B^{n,s}$ to $\pm$-diagrams in type $B_n^{(1)}$. Then we have
\[
\zeta_{rc} = \iota \circ \zeta.
\]
\end{prop}

\begin{proof}
This follows from the fact that the doubling map is a virtualization map~\cite[Lemma~4.2]{FOS09}, Proposition~\ref{prop:rc_pm_bijection}, and 
Proposition~\ref{prop:rc_pm_doubling_commute}.
\end{proof}

\begin{thm}
Consider $\RC(B^{r,s})$ in type $B_n^{(1)}$ or $A_{2n-1}^{(2)}$. The natural classical crystal isomorphism $\iota \colon \RC(B^{r,s}) \to B^{r,s}$ 
is an affine crystal isomorphism.
\end{thm}

\begin{proof}
The classical crystal isomorphism $\iota$ intertwines with $\sigma$ and $\sigma_{rc}$ by construction. Therefore 
Proposition~\ref{prop:rc_pm_bijection} implies that $\iota$ is an affine crystal isomorphism.
\end{proof}

\subsection{Virtualization as affine crystals}

By constructing the virtualization map on $\pm$-diagrams, we can show Conjecture~\ref{conj:KR_virtualization} for $B^{r,s}$ of types 
$B_n^{(1)}$ and $A_{2n-1}^{(2)}$ (i.e., those that virtualize in $D_{n+1}^{(1)}$) for $r < n$ (which we assume in this subsection). 
We first must describe the action of $e_0$ and $f_0$ on $\pm$-diagrams.

Define a virtualization map $v$ on $\pm$-diagrams of outer shape $\lambda$ to $\pm$-diagrams of outer shape $\Psi(\lambda)$ 
by $c_*(r) \mapsto \gamma_r c_*(r)$ where $* = \bullet, -, +, \pm$.

\begin{lemma}
\label{lemma:virtual_pm_diagrams}
Consider $B^{r,s}$ of type $B_n^{(1)}$ or type $A_{2n-1}^{(2)}$. The virtualization map $v$ restricted to $\{2, \dotsc, n\}$-highest weight 
elements in $\RC(B^{r,s})$ commutes with $\zeta_{rc}$ and $\zeta$.
\end{lemma}

\begin{proof}
It is clear that $\virtual{\zeta}_{rc}^{-1} \circ v = v \circ \zeta_{rc}^{-1}$ from the definition of $\zeta_{rc}$, which proves our first claim. Next 
since $\zeta_{rc} = \iota \circ \zeta$, $\virtual{\zeta}_{rc} = \virtual{\iota} \circ \virtual{\zeta}$ and the fact that Theorem~\ref{thm:virtual_highest_weight} 
implies that $v \circ \iota = \virtual{\iota} \circ v$, we have
\[
v \circ \zeta = v \circ \iota^{-1} \circ \zeta_{rc} = \virtual{\iota}^{-1} \circ v \circ \zeta_{rc} = \virtual{\iota}^{-1} \circ \virtual{\zeta}_{rc} \circ v 
= \virtual{\zeta} \circ v.
\]
\end{proof}

\begin{lemma}
\label{lemma:virtual_pm_involution}
The virtualization map $v$ commutes with $\mathfrak{S}$.
\end{lemma}

\begin{proof}
Since $\mathfrak{S}$ can be reformulated as acting column by column 
and $\gamma_a = \gamma_b$ for all $a,b < n$, it is clear that $\virtual{\mathfrak{S}} \circ v = v \circ \mathfrak{S}$.
\end{proof}

\begin{prop}
\label{prop:virtual_sigma}
The virtualization map on $\pm$-diagrams commutes with $\sigma_{rc}$ and $\sigma$.
\end{prop}

\begin{proof}
This follows from Theorem~\ref{thm:virtual_highest_weight}, Lemma~\ref{lemma:virtual_pm_diagrams}, and Lemma~\ref{lemma:virtual_pm_involution}.
\end{proof}

Thus we can show the following case of Conjecture~\ref{conj:KR_virtualization}.

\begin{thm}
\label{thm:virtual_KR}
Let $B^{r,s}$ be a KR crystal of type $B_n^{(1)}$ or $A_{2n-1}^{(2)}$. Then $B^{r,s}$ virtualizes in $B^{r,\gamma_r s}$ of type $D_{n+1}^{(1)}$ as $U_q'(\g)$-crystals.
\end{thm}

\begin{proof}
This follows from Equation~\eqref{eq:affine_operators_diagram_map}, Proposition~\ref{prop:virtual_sigma}, and Theorem~\ref{thm:virtual_highest_weight}.
\end{proof}

\begin{remark}
Theorem~\ref{thm:virtual_KR} implies that the doubling map for $B^{n,s}$ in type $B_n^{(1)}$ into type $A_{2n-1}^{(2)}$ can be extended to a virtualization 
map given by Equation~\eqref{eq:doubling_map} for any $B^{r,s}$ with $r < n$ into $B^{r,\gamma_r s}$ in type $A_{2n-1}^{(2)}$ with 
$\gamma_r = 2$ for all $r < n$ and $\gamma_n = 1$. This can be seen by the composition of virtualization maps
\[
	B_n^{(1)} \xrightarrow[\hspace{30pt}]{v} D_{n+1}^{(1)} \xrightarrow[\hspace{30pt}]{v^{-1}} A_{2n-1}^{(2)}.
\]
\end{remark}

We note that the result of this section cannot be easily extended to types $\g = C_n^{(1)}, D_{n+1}^{(2)}, A_{2n}^{(2)}$ because the 
construction of $e_0$ and $f_0$ in type $\g$ as given in~\cite[Sec.~4]{FOS09} use a different virtual construction than the one
discussed here, and this other virtual construction is not well-behaved with respect to rigged configurations because the folding does not 
preserve the affine node $0$.

\subsection{Extension to $r = n$}

Recall that in type $D_{n+1}$, we can represent $B(\clfw_n) \otimes B(\clfw_{n+1})$ as a usual KN column of height $n$. From the construction of the ambient Kleber tree for $r = n$ in the proof of Lemma~\ref{lemma:hw_typeA2odd}, we know that $B^{n,s} \otimes B^{n+1,s}$ of type $D_{n+1}^{(1)}$ has a classical decomposition
given by removing vertical dominoes from an $n \times s$ rectangle, analogous to the usual type $D_{n+1}^{(1)}$ case of $r<n$. We note that 
there is also an extension of Proposition~\ref{prop:pm_diagram_classification} to spin columns in type $D_{n+1}^{(1)}$.

Thus we can define an affine crystal $\widetilde{B}^{n,s}$ of type $D_{n+1}^{(1)}$ by having a classical decomposition 
$\bigoplus_{\lambda} B(\lambda)$, where $\lambda$ is obtained by removing vertical dominoes from an $n \times s$ rectangle 
and the affine structure by Equation~\ref{eq:affine_operators_diagram_map}. From this definition and the preceeding paragraph, we have 
$\widetilde{B}^{n,s} \iso B^{n,s} \otimes B^{n+1,s}$ as classical crystals. Moreover, we can define a filling map 
$\fillmap \colon \widetilde{B}^{n,s} \to \widetilde{T}^{n,s}$ as in the usual type $D_{n+1}^{(1)}$ case and can extend 
Theorem~\ref{thm:isomorphism} to this case as well by extending $\Phi$ in a natural way by using $\widetilde{\delta}^{(n)}$, see~\cite{S05}.

Additionally, we can extend the virtualization map $v \colon B^{n,s} \to \widetilde{B}^{n,s}$ from type $A_{2n-1}^{(2)}$ to type $D_{n+1}^{(1)}$
as the identity map (on rigged configurations, it is almost the identity except for $\nu^{(n)} = \virtual{\nu}^{(n)} = \virtual{\nu}^{(n+1)}$). From 
this we can see that Conjecture~\ref{conj:KR_virtualization} in type $A_{2n-1}^{(2)}$ for $B^{n,s}$ is equivalent to the following conjecture.

\begin{conj}
\label{conj:KR_typeD}
Let $\g$ be of type $D_{n+1}^{(1)}$. We have $\widetilde{B}^{n,s} \iso B^{n,s} \otimes B^{n+1,s}$ as affine crystals.
\end{conj}

This conjecture was proven for $s = 1$ in~\cite[Thm.~3.3]{S05}.

\section{The virtualization map and $\Phi$}
\label{sec:single_factors}

In this section, we show that the virtualization map commutes with the bijection $\Phi$ on highest weight elements of a single
tensor factor for $\g$ of non-exceptional affine type. In addition, for type $A_{2n-1}^{(2)}$ we prove that the virtualization
map in general (multiple tensor factors and not necessarily highest weight) commutes with $\Phi$.

\subsection{Single tensor factors}

In this subsection $\g$ is of non-exceptional type. Recall that $B^{r,s} \cong T^{r,s}$ are related by the filling map and are 
isomorphic as crystals. Hence the virtualization maps on $B^{r,s}$ of Section~\ref{subsection.virtual crystals} can be lifted to $T^{r,s}$.
We define the crystal morphism $v \colon T^{r,s} \to \virtual{T}^{r,s}$ by sending $u_{\lambda} \in B(\lambda) \subseteq T^{r,s}$ to 
$u_{\Psi(\lambda)} \in B(\Psi(\lambda)) \subseteq \virtual{T}^{r,s}$ and extending as a virtual classical crystal. It is not a priori 
clear that $B(\Psi(\lambda))$ is indeed a component in $\virtual{T}^{r,s}$, so it needs to be shown that $v$ is well-defined.

\begin{lemma}
\label{lemma:virtualization_KR_crystals}
The map $v$ is well-defined and virtualizes $T^{r,s}$ in $\virtual{T}^{r,s}$ as a classical virtual crystal.
\end{lemma}

\begin{proof}
When $\virtual{\g}$ is of type $D_{n+1}^{(1)}$ with $r < n$, the claim follows from Theorem~\ref{thm:virtual_KR}. For $r = n$, this follows 
from the proofs of Lemma~\ref{lemma:hw_typeA2odd} and Lemma~\ref{lemma:hw_typeB_spin}.
 
Now assume that $\virtual{\g}$ is of type $A_{2n-1}^{(1)}$. We note that the decomposition of a KR crystal of non-exceptional type 
into classical crystals is multiplicity free. The tensor product of two rectangles in type $A_{2n-1}^{(1)}$ is multiplicity free~\cite{Stembridge01} 
(we have also shown this in the proof of Lemma~\ref{lemma:hw_typeC} during the construction of the ambient Kleber tree). From the construction of the ambient Kleber tree in the proof of 
Lemma~\ref{lemma:hw_typeC}, we have shown that for every shape $\lambda$ in a $r \times s$ rectangle, the crystal $B(\lambda)$ virtualizes 
into the decomposition of $T^{r,s} \otimes T^{2n-r,s}$. Thus there exists a unique classical crystal $B(\Psi(\lambda)) \subseteq \virtual{T}^{r,s}$ 
corresponding to $B(\lambda) \subseteq T^{r,s}$. Hence the map $v$ is well-defined. That $T^{r,s}$ virtualizes (as a classical crystal) in 
$\virtual{T}^{r,s}$ under $v$ follows from Theorem~\ref{thm:virtual_highest_weight}.
\end{proof}

Consider a weight $\lambda = \sum_{i \in I_0} k_i \clfw_i$ of type $\g$. Suppose $\virtual{\g}$ is of type $D_{n+1}^{(1)}$. For $T^{r,s}$ with 
$r < n$, the corresponding classically highest weight element is $u_{\Psi(\lambda)} \in \virtual{T}^{r,s} = T^{r,\gamma_r s}$ given in 
Section~\ref{sec:filling_map}. For $T^{n,s}$, the classically highest weight element $u_{\Psi(\lambda)} \in \virtual{T}^{n,s} = T^{n,s} \otimes T^{n+1,s}$ 
is given by filling the right tableau by trivial columns of $\column{1, \dotsc, n, n+1}$ and the left tableau with 
$\column{1, \dotsc, k, \overline{n+1}, \dotsc, \overline{k+1}}$, where $k$ is the height of the corresponding column in $\lambda$.
Now suppose $\virtual{\g}$ is of type $A_{2n-1}^{(1)}$. For $T^{r,s}$ with $r < n$, the classically highest weight element 
$u_{\Psi(\lambda)} \in \virtual{T}^{r,s} = T^{r,s} \otimes T^{2n-r,s}$ is given by filling the right tableau with trivial columns of 
$\column{1, 2, \dotsc, 2n-r}$ and the left tableau with $\column{1, \dotsc, k, 2n-r+1, \dotsc, 2n-k}$, where $k$ is the height of the corresponding 
column in $\lambda$. For $r = n$ in types $A_{2n}^{(2)}$ and $A_{2n}^{(2)\dagger}$, the image $u_{\Psi(\lambda)} \in \virtual{T}^{n,s}$ is 
the same as above. For $r = n$ in types $C_n^{(1)}$ and $D_{n+1}^{(2)}$, the corresponding 
$u_{\Psi(\lambda)} \in \virtual{T}^{n,s} = T^{n, \gamma_n s}$ is the tableau with trivial columns $\column{1, 2, \dotsc, n}$.

Now we can prove the main result of this section.

\begin{thm}
\label{thm:bij_virtual_hw}
Consider a single Kirillov-Reshetikhin crystal $B^{r,s}$. The virtualization map $v$ commutes with the bijection $\Phi$ on highest weight elements.
\end{thm}

\begin{proof}
We consider a highest weight $\lambda$ in the classical decomposition of $T^{r,s} \iso B^{r,s}$. The corresponding type $\g$ rigged configuration is 
generally $\nu^{(a)} = \overline{\lambda}$ for all $r \leq a < n$ and $\nu^{(a)} = \overline{\lambda}^{[r-a]}$ for all $a < r$ (recall that $\overline{\lambda}$ 
is the complement of $\lambda$ in an $r \times s$ box and $\lambda^{[i]}$ denotes $\lambda$ with the first $i$ rows removed) with all riggings 
and vacancy numbers are $0$ from the results in Section~\ref{sec:filling_map}. Let $k$ be the largest index such that 
$\inner{\alpha_k^{\vee}}{\lambda} \neq 0$ (i.e., the height of $\lambda$).

\case{$\virtual{\g}$ of type $D_{n+1}^{(1)}$ and $r < n$}

\noindent
Note that we double $\lambda$ under the virtualization map if $\gamma_r = 2$. Hence by weight considerations (recall $\Phi$ is a bijection 
on classical highest weight elements in $T^{r,s}$ 
and the classical decomposition is multiplicity free) and the fact that the virtual rigged configuration corresponds to the highest weight in $\virtual{\g}$ 
for $\Psi(\lambda)$, the bijection $\Phi$ must commute with $v$ (on classically highest weight elements).

\case{$\virtual{\g}$ of type $D_{n+1}^{(1)}$ and $r = n$}

\noindent
We note that in type $B_n^{(1)}$, the spinor lifts to the type $A_{2n-1}^{(2)}$ case (albeit with $T^{n,2s}$). Thus without loss of generality, 
assume $\g$ is of type $A_{2n-1}^{(2)}$. We begin by splitting off the leftmost column, which increases the vacancy numbers of rows 
smaller than $2\lambda_1$ in $\nu^{(n-1)}$. Therefore when we apply the doubling map, it keeps the strings of length smaller than 
$2\lambda_1$ non-singular. Thus $\delta^{(n)}$
(recall $\delta^{(n)}$ and $\widetilde{\delta}^{(n)}$ were defined as slightly modified versions of $\delta$ in Section~\ref{sec:background})
selects the row corresponding to $2\lambda_1$ and must terminate at $k$ since 
$\nu^{(k)}$ does not have any rows of length $2\lambda_1$, 
and thus $\delta$ returns $\overline{k+1}$. Next applying $\widetilde{\delta}^{(n)}$ 
selects $2\lambda_1$ from $\nu^{(n+1)}$, skips $\nu^{(n)}$, and proceeds down until $\nu^{(k+1)}$ and returns a $\overline{k+2}$.

Now we are applying the usual $\delta$ where we select a string of length $2 \lambda_1 - 1$ from $\nu^{(n-1)}, \nu^{(n)}, \nu^{(n+1)}$, and then 
another string of length $2 \lambda_1 - 1$ down until $\nu^{(k+2)}$ and returns a $\overline{k+3}$. A similar process holds for each of the 
remaining $k-3$ rows of length $2\lambda_1 - 1$. At this point, all the strings in $\nu^{(k)}$ are not singular, thus $\delta$ returns $k$. This 
process repeats until we remove the entire column. A similar process occurs for the remaining columns in the left factor until it is completely 
removed. Once we are at the right factor of $T^{n,s}$, we have the empty rigged configuration. Therefore $\Phi$ returns the letter $a$ for each 
entry at height $a$ and we have the desired filling.

\case{$\virtual{\g}$ of type $A_{2n-1}^{(1)}$}

\noindent
To see that the image under $v$ of the KR tableaux corresponds to the virtual rigged configuration, we first split off the leftmost column. This 
increases vacancy numbers of rows smaller than $\lambda_1$ in $\nu^{(r)}$. Therefore $\delta$ selects the row corresponding to $\lambda_1$ 
and must terminate at $2n-k$ since $\nu^{(2n-k)}$ does not have any rows of length $\lambda_1$ since $k$ corresponds to the number of 
rightmost columns.
A similar procedure occurs except using $k^{\prime} = k+1$ and repeating until all rows of length $\lambda_1$ are removed. 
At this point, all the strings in $\nu^{(k)}$ are not singular, thus $\delta$ returns $k$. This process repeats until we remove the entire column. 
We then repeat this for the next column, and a similar situation holds. This process is repeated until the left factor is removed, and we are 
left with the empty rigged configuration. Hence the right tableau must by filled by $\column{1, \dotsc, 2n-k}$.
\end{proof}

\subsection{General case}

Now we consider the general case, where we start by giving an extended version of~\cite[Conj.~7.2]{OSS03III}.

\begin{conj}
\label{conj:virtual_bijection}
Let $\g$ be of affine type and $B = \bigotimes_{i=1}^N B^{r_i, s_i}$ with virtualization map $v$ into type $\virtual{\g}$. Then we have
\[
v \circ \Phi = \virtual{\Phi} \circ v.
\]
\end{conj}

Conjecture~\ref{conj:virtual_bijection} was shown for $\bigotimes_{i=1}^N B^{r_i,1}$ in types $C_n^{(1)}$, $D_{n+1}^{(2)}$, and 
$A_{2n}^{(2)}$ in~\cite{OSS03III} and for $\bigotimes_{i=1}^N B^{1,s_i}$ in all non-exceptional affine types in~\cite{SS2006}. 
Also Theorem~\ref{thm:bij_virtual_hw} is Conjecture~\ref{conj:virtual_bijection} for classically highest weight elements for 
$B = B^{r,s}$ of non-exceptional affine type. We show that this reduces Conjecture~\ref{conj:crystal_isomorphism} to showing 
it holds in simply-laced types.

\begin{prop}
Let $\g$ be of affine type. Suppose Conjecture~\ref{conj:virtual_bijection} holds and Conjecture~\ref{conj:crystal_isomorphism} 
holds in type $\virtual{\g}$, then Conjecture~\ref{conj:crystal_isomorphism} holds in type $\g$.
\end{prop}

\begin{proof}
Let $V^*$ be the set of classically highest weight rigged configurations in $\RC(\virtual{B})$ that satisfy Equation~\eqref{eq:rc_virtualization_map}. 
From Equation~\eqref{eq:vacancy_virtualization}, it is easy to see that $V^*$ is in bijection with $\RC^*(B)$. We have
\[
v \circ \Phi \circ f_a = \virtual{\Phi} \circ v \circ f_a = \virtual{\Phi} \circ f_a^v \circ v
= f_a^v \circ \virtual{\Phi} \circ v = f_a^v \circ v \circ \Phi = v \circ f_a \circ \Phi,
\]
so $f_a \circ \Phi = \Phi \circ f_a$ and similarly for $e_a$. Therefore $\Phi$ is a crystal isomorphism.
\end{proof}

\begin{thm}
\label{thm:virtual_bijection_A2odd}
Conjecture~\ref{conj:virtual_bijection} holds for $\g$ of type $A_{2n-1}^{(2)}$.
\end{thm}

\begin{proof}
Let $b, \virtual{b}$ be the elements returned under $\delta, \virtual{\delta}$, respectively. It was shown in~\cite{SS2006} that 
$\virtual{\delta} \circ v = v \circ \delta$ and $v(b) = \virtual{b}$. Since $v$ is the identity map on tableaux and is essentially the identity 
map on rigged configurations (recall that $\nu^{(n)} \mapsto \virtual{\nu}^{(n)} = \virtual{\nu}^{(n+1)}$), we have $\virtual{\ls} \circ v = v \circ \ls$ 
and $\virtual{\lt} \circ v = v \circ \lt$ on both $\RC(B)$ and $B$. Therefore by the definition of $\Phi$, we have
\[
	\virtual{\Phi} \circ v = v \circ \Phi.
\]
\end{proof}

Therefore Conjecture~\ref{conj:crystal_isomorphism} holds for type $A_{2n-1}^{(2)}$ any time $\Phi$ is known to be a bijection 
in type $D_n^{(1)}$. Alternatively, recall that the algorithm $\Phi$ for the $A_{2n-1}^{(2)}$ case is essentially identical to the type $D_n^{(1)}$ case. Thus the property that $\Phi$ is a classical crystal isomorphism is an immediate consequence of~\cite{Sakamoto13}. If we use the argument after Conjecture~\ref{conj:crystal_isomorphism}, we are not necessarily to use Theorem~\ref{thm:virtual_bijection_A2odd}.

\appendix
\section{Proof of Theorem~\ref{thm:virtual_highest_weight}}
\label{appendix:extension}

Theorem~\ref{thm:virtual_highest_weight} was proved by Baker~\cite{baker2000} when $\virtual{\g}_0$ is of type
$A_{2n-1}$. We provide details for the other cases here. Our proof follows the proof of Baker~\cite{baker2000}, in that we show that 
$B(\clfw_a)$ virtualizes in $B\bigl(\sum_{b \in \phi^{-1}(a)} \gamma_a \clfw_b \bigr)$ and then use Proposition~\ref{prop:virtual_tensor} 
to extend this to general shapes $\lambda$.

\begin{table}[t]
\[
\begin{array}{|c|c|c|c|c|}\hline
\g & \virtual{B}^{1,1} & \virtual{B}^{2,1} & \virtual{B}^{3,1} & \virtual{B}^{4,1}
\\ \hline
E_6^{(2)} & B^{2,1} & B^{4,1} & B^{3,1} \otimes B^{5,1} & B^{1,1} \otimes B^{6,1}
\\ \hline
F_4^{(1)} & B^{2,2} & B^{4,2} & B^{3,1} \otimes B^{5,1} & B^{1,1} \otimes B^{6,1}
\\ \hline
G_2^{(1)} & B^{1,1} \otimes B^{3,1} \otimes B^{4,1} & B^{2,3} & &
\\ \hline
D_4^{(3)} & B^{2,1} & B^{1,1} \otimes B^{3,1} \otimes B^{4,1} & &
\\ \hline
\end{array}
\]
\caption{Virtualizations given in Proposition~\ref{prop:single_column_fg}.}
\label{table:single_column_fg}
\end{table}

\begin{prop}
\label{prop:single_column_fg}
Consider one of the foldings
\begin{align*}
E_6^{(2)}, F_4^{(1)} & \lhook\joinrel\longrightarrow E_6^{(1)},
\\ G_2^{(1)}, D_4^{(3)} & \lhook\joinrel\longrightarrow D_4^{(1)}.
\end{align*}
Unless $a=2$ and $\g$ is of type $F_4^{(1)}$, $B^{a,1}$ virtualizes in 
$\virtual{B}^{a,1} = \bigotimes_{b \in \phi^{-1}(a)} B^{b,\gamma_a}$ as affine crystals.
\end{prop}

\begin{proof}
This was be done by (computer) computation in Sage~\cite{sage} using the results from~\cite{JS10,LSSSS14,LNSSS14}. 
The algorithm is to start from any node in the image of $v$ (usually this is the unique node of (classical) weight $\Psi(\lambda)$), 
apply all possible $f_i^v$ to build the crystal graph of the virtual crystal inside of $B(\Psi(\lambda))$, and compare it to the crystal 
graph of $B(\lambda)$.

We note that the results from~\cite{LSSSS14,LNSSS14} only give a model for the single column KR crystal $B^{r, 1}$. In the remaining 
cases, the resulting KR crystal is not a single column (for example in type $G_2^{(1)}$, the virtualization of $B^{2,1}$ is $B^{2,3}$), 
but other models exist for these cases~\cite{FOS09, JS10}.
\end{proof}

A similar check could be made for $B^{2,1}$ in type $F_4^{(1)}$ once the crystal graph for $B^{4,2}$ in type $E_6^{(1)}$ is computed.

\begin{lemma}
\label{lemma:folding_F}
Consider the folding $F_4 \lhook\joinrel\longrightarrow E_6$. Then $B(\clfw_2)$ virtualizes in $B(2\clfw_4)$.
\end{lemma}

\begin{proof}
This was done by (computer) computation using well-known models for type $F_4$ and $E_6$ crystals (for example, LS paths or 
Nakajima monomials) following the algorithm in Proposition~\ref{prop:single_column_fg}.
\end{proof}

Combining Proposition~\ref{prop:single_column_fg} and Lemma~\ref{lemma:folding_F} using the (virtual) Kleber algorithm shows that
$B^{2,1}$ of type $F_4^{(1)}$ classically virtualizes in $B^{4,2}$ of type $E_6^{(1)}$ (as opposed to as affine crystals).

\begin{lemma}
\label{lemma:folding_CD}
Consider the folding $C_n  \lhook\joinrel\longrightarrow D_{n+1}$. Then $B(\clfw_r)$ virtualizes in $B(\clfw_r)$ for all $r \neq n$ and 
$B(\clfw_n + \clfw_{n+1})$ for $r = n$.
\end{lemma}

\begin{proof}
Recall that in type $D_{n+1}$ that $B(\clfw_n + \clfw_{n+1})$ is represented by a single column of height $n$. We claim that the virtualization 
map $v$ in both cases is given by the identity map on tableaux. That $v$ commutes with $f_a$ for all $a \neq n$ is clear since $\gamma_a = 1$. 
For $a = n$, we have that $f_n^v = f_n f_{n+1}$ sends $n$ to $\overline{n}$. Therefore if we start with something that does not contain an 
$n+1$ or an $\overline{n+1}$, we cannot obtain an $n+1$ or $\overline{n+1}$. Hence neither $n+1$ nor $\overline{n+1}$ can appear in 
the image and all properties of Definition~\ref{dfn:virtual_crystal} can be checked.
\end{proof}

\begin{lemma}
\label{lemma:folding_BD}
Consider the folding $B_n \lhook\joinrel\longrightarrow D_{n+1}$. Then $B(\clfw_r)$ virtualizes in $B(2\clfw_r)$ for all $r \neq n$ 
and $B(\clfw_n + \clfw_{n+1})$ for $r = n$.
\end{lemma}

\begin{proof}
Consider first $r = n$. Recall that $B(\clfw_n)$ virtualizes in $B(2\clfw_n)$ (both of these are type $B_n$ crystals) by taking 
$\gamma_a = 2$ for all $a \in I_0$ as mentioned in Remark~\ref{remark:spin_tableaux}. Thus we can represent the elements of 
$B(2\clfw_n)$ by single column tableaux, and we claim that the desired virtualization map is the composition
\[
v \colon B(\clfw_n) \xrightarrow[\hspace{30pt}]{d} B(2\clfw_n) \xrightarrow[\hspace{30pt}]{v'} B(\clfw_n + \clfw_{n+1}),
\] 
where the second map is the identity map on tableaux. For $a \neq n$, we have $f^v_a = \virtual{f}_a^2$ and so $v \circ f_a = f^v_a \circ v$
since the embedding $d$ doubles everything (in particular we apply $\virtual{f}_a$ twice) and $v'$ is the identity map. 
For $r = n$ in $B(2\clfw_n)$ of type $B_n$, the crystal operator $f_n^d = \virtual{f}_n^2$ 
sends $n \mapsto 0 \mapsto \overline{n}$ and in type $D_{n+1}$, we have $f^{v}_n = \virtual{f}_n \virtual{f}_{n+1}$ sending $n \mapsto n+1 \mapsto \overline{n}$ (alternatively going through 
$\overline{n+1}$). Therefore $v$ is the desired virtualization map.

Next assume $r < n$ and let $t$ be a single column tableau in $B(\clfw_r)$ of type $B_n^{(1)}$. 
Let $t_+ = \{a \mid a \in t, a\neq 0\}$ and $t_-=\{a \mid \overline{a} \in t, a\neq 0\}$ 
denote the set of non-zero unbarred and barred letters in $t$, respectively. Let
\begin{align*}
K & := t_+ \cap t_-, 
\\ J_0 & := \max\{ A \subseteq (t_+ \cup t_-)^c \mid \absval{A} = k_0 \},
\\ J & := \max\{ A \subseteq (t_+ \cup t_- \cup J_0)^c \mid \absval{A} = \absval{K}, \text{ and } A < K \},
\end{align*}
where $k_0$ is the number of times $0$ appears in $t$, the maxima are taken with respect to lexicographic order $<$, and 
$p^c = \{1, \dotsc, n\} \setminus p$. We note that $J$ is well-defined from the one column condition on $B_n$ tableaux 
(see~\cite{KN94,HK02}). Let $v_{\pm} := t_{\pm} \setminus K$. We claim that the image of $t$ under the virtualization map $v$ is
\[ v(t)=
\begin{array}{c|c}
\overline{t}_- & \overline{v}_- \\
J & \overline{J}_0 \\
J_0 & \overline{J} \\
v_+ & t_+
\end{array},
\]
where $\overline{L}$ denotes the set $L$ but with barred letters and reordering within each column as necessary.

Since $f_n^v=\virtual{f}_n \virtual{f}_{n+1}$, we have by similar arguments to the proof of Lemma~\ref{lemma:folding_CD} that neither 
$n+1$ nor $\overline{n+1}$ can appear in the image. Also note that $v_+, v_-, J_0, J, K$ are pairwise disjoint. From the construction, 
it is clear that the image has the correct weight. Next we let $t' = f_a t$ and $\virtual{t} = v(t)$, We also let $K'$, $J_0'$, $J'$, and $v'_{\pm}$ denote 
the above constructions with $t'$. We proceed by doing a case-by-case analysis to show that $v \circ f_a = f_a^v \circ v$. The main cases split 
according to which of these disjoint sets $a$ belongs to. Within each case, we also have to consider whether $a < n$ or $a = n$ and what set 
$a + 1$ belongs to.

\case{Case $a \in v_+$}

We split this into two subcases, when $a < n$ and when $a = n$. Note that $\overline{a} \notin t, \virtual{t}$ and
$a$ appears in both columns of $\virtual{t}$ by construction.

\case{Subcase $a < n$}

We begin by assuming $a+1 \notin J_0, J$. Now for the next two (sub)subcases, we assume that $a+1 \notin t$. If $\overline{a+1} \notin t$, 
then $t'$ differs from $t$ by replacing $a$ with $a+1$ and $f_a^v \virtual{t}$ replaces both instances of $a$ with $a+1$. Diagrammatically, 
writing only the entries that contribute to the computation of $f_a$ (i.e., all of the other entries do not change under $f_a$), we have
\[
\virtual{t} = \begin{array}{c|c}
\vdots & \vdots
\\ a & a
\\ \vdots & \vdots
\end{array}
\xrightarrow[\hspace{30pt}]{\virtual{f}_a}
\begin{array}{c|c}
\vdots & \vdots
\\ a & a+1
\\ \vdots & \vdots
\end{array}
\xrightarrow[\hspace{30pt}]{\virtual{f}_a}
\begin{array}{c|c}
\vdots & \vdots
\\ a+1 & a+1
\\ \vdots & \vdots
\end{array} = f_a^v \virtual{t} = v(t').
\]
If $\overline{a+1} \in t$, then we have
\[
\virtual{t} = \begin{array}{c|c}
\overline{a+1} & \overline{a+1}
\\ \vdots & \vdots
\\ a & a
\end{array}
\xrightarrow[\hspace{30pt}]{\virtual{f}_a}
\begin{array}{c|c}
\overline{a+1} & \overline{a+1}
\\ \vdots & \vdots
\\ a & a+1
\end{array}
\xrightarrow[\hspace{30pt}]{\virtual{f}_a}
\begin{array}{c|c}
\overline{a+1} & \overline{a}
\\ \vdots & \vdots
\\ a & a+1
\end{array} = f_a^v \virtual{t}.
\]
As before, we replaced $a$ with $a+1$ in $t'$, but now we have $a+1 \in K'$. Thus either $a \in J_0'$ or $a \in J'$ by the one column 
condition and the construction of $J_0'$ and $J'$, and hence $f_a^v \virtual{t} = v(t')$.

Now we assume $a+1 \in t$ (note that by construction $a+1 \notin J_0, J$). So if $\overline{a+1} \notin t$, then $f_a t = 0$ and $f_a^v \virtual{t} = 0$. 
Otherwise if $\overline{a+1} \in t$, then we replace $\overline{a+1} \mapsto \overline{a}$ to obtain $t'$, and we have
\[
\virtual{t} = \begin{array}{c|c}
\overline{a+1} & \vdots
\\ \vdots & a+1
\\ a & a
\end{array}
\xrightarrow[\hspace{50pt}]{f_a^v}
\begin{array}{c|c}
\overline{a} & \vdots
\\ \vdots & a+1
\\ a+1 & a
\end{array} = f_a^v \virtual{t}.
\]
Note that $\absval{K} = \absval{K'}$, and so $a \in K'$. Additionally $a+1 \in v_+'$, and hence we have $f_a^v \virtual{t} = v(t')$.

Next we consider the case when $a + 1 \in J$. Therefore we have
\[
\virtual{t} = \begin{array}{c|c}
\vdots & \overline{a+1}
\\ a + 1 & \vdots
\\ a & a
\end{array}
\xrightarrow[\hspace{50pt}]{f_a^v}
\begin{array}{c|c}
\vdots & \overline{a}
\\ a+1 & \vdots
\\ a & a + 1
\end{array} = f_a^v \virtual{t},
\]
and $f_a$ sends the (unique) $a \mapsto a+1$ to obtain $t'$. Thus since $a \notin t'_+ \cup t'_-$, we have that $a \in J'$.
Hence $v(t') = f_a^v \virtual{t}$. The case $a + 1 \in J_0$ is similar to the case $a+1\in J$.

\case{Subcase $a = n$}

By our assumption of $a \in v_+$, we note that applying $f_n$ sends the $n \mapsto 0$ to obtain $t'$. Recall that $f_n^v = \virtual{f}_n \virtual{f}_{n+1}$, and so we have
\[
\virtual{t} = \begin{array}{c|c}
\vdots & \vdots
\\ n & n
\end{array}
\xrightarrow[\hspace{50pt}]{f_n^v}
\begin{array}{c|c}
\vdots & \overline{n}
\\ n & \vdots
\end{array} = f_a^v \virtual{t}.
\]
Since $n \notin t_{\pm}'$, we must have $n \in J'_0$, and therefore we have $f_n^v \virtual{t} = v(t')$.

\case{Case $a \in v_-$}

We must have $f_a t = 0$ and $f_a^v \virtual{t} = 0$ since $a \notin t, \virtual{t}$ and any $\overline{a+1}$ (resp. $n,0$ if $a = n$) would pair with $\overline{a}$ (resp. $\overline{n}$).

\case{Case $a \in J_0$}

We recall that $a, \overline{a} \notin t$ from the construction of $J_0$. We now proceed into subcases.

\case{Subcase $a < n$}

By construction of $J_0$ and $J$, we must have that $a+1$ is in $v_+$, $v_-$, $K$, or $J_0$. So we start by assuming $a+1 \in v_+$. Then we have $f_a t = 0$ and
\[
\virtual{t} = \begin{array}{c|c}
\vdots & \overline{a}
\\ a+1 & \vdots
\\ a & a+1
\end{array},
\]
and so $f_a^v \virtual{t} = 0$.

If $a+1 \in v_-$, then $t'$ is obtained by sending the $\overline{a+1} \mapsto \overline{a}$. Thus we have
\[
\virtual{t} = \begin{array}{c|c}
\overline{a+1} & \overline{a}
\\ \vdots & \overline{a+1}
\\ a & \vdots
\end{array}
\xrightarrow[\hspace{50pt}]{f_a^v}
\begin{array}{c|c}
\overline{a} & \overline{a}
\\ \vdots & \overline{a+1}
\\ a+1 & \vdots
\end{array} = f_a^v \virtual{t}.
\]
Note that $a \in v'_-$, and we have $a+1 \in J'_0$ because $\absval{J'_0} = \absval{J_0}$ with $a+1 \in (t'_+ \cup t'_-)^c$ and $a$ was the previous maximum. Therefore $f_a^v \virtual{t} = v(t')$.

If $a+1 \in K$, then we have
\[
\virtual{t} = \begin{array}{c|c}
\overline{a+1} & \overline{a}
\\ \vdots & \vdots
\\ a & a+1
\end{array}
\xrightarrow[\hspace{50pt}]{f_a^v}
\begin{array}{c|c}
\overline{a} & \overline{a}
\\ \vdots & \vdots
\\ a+1 & a+1
\end{array} = f_a^v \virtual{t}.
\]
$t'$ is given by sending the $\overline{a+1} \mapsto \overline{a}$, so that
$a \in v'_-$ and $a+1 \in v'_+$. Now since $\absval{J_0'} = \absval{J_0}$ but $\absval{K'} = \absval{K} - 1$, 
we have moved a letter $b \in J$ into $J_0'$ (i.e., $b \notin J'$). Hence $f_a^v \virtual{t} = v(t')$.

If $a+1 \in J$, then we have $f_a t = 0$ and 
\[
\virtual{t} = \begin{array}{c|c}
\vdots & \overline{a}
\\ \vdots & \overline{a+1}
\\ a+1 & \vdots
\\ a & \vdots
\end{array},
\]
and hence $f_a^v \virtual{t} = 0$.

\case{Subcase $a = n$}

We obtain $t'$ from mapping $0 \mapsto \overline{n}$, and hence $n \in v'_-$. So we have
\[
\virtual{t} = \begin{array}{c|c}
\vdots & \overline{n}
\\ n & \vdots
\end{array}
\xrightarrow[\hspace{50pt}]{f_n^v}
\begin{array}{c|c}
\overline{n} & \overline{n}
\\ \vdots & \vdots
\end{array} = f_a^v \virtual{t} = v(t')
\]
since $\absval{J_0'} = 0$ and $J_0 = \{n\}$.

\case{Case $a \in J$}

Unlike in the previous cases, we cannot have $n \in J$, so we only need to consider $a < n$. From the construction of $J$, we have that $a, \overline{a} \notin t$.

\case{Subcase $a+1 \in J_0$ or $J$}
In this case $a+1, \overline{a+1}\not \in t$, so that we have $f_a t = 0$. On the other hand
\[
\virtual{t} = \begin{array}{c|c}
\vdots & \overline{a}
\\ \vdots & \overline{a+1}
\\ a+1 & \dots
\\ a & \vdots
\end{array},
\]
and hence $f_a^v \virtual{t} = 0$.

\case{Subcase $a+1 \in v_+$}

We have $f_a t = 0$ and
\[
\virtual{t} = \begin{array}{c|c}
\vdots & \overline{a}
\\ a + 1 & a + 1
\\ a & \vdots
\end{array},
\]
so that $f_a^v \virtual{t} = 0$.

\case{Subcase $a+1 \in v_-$}

$t'=f_at $ is obtained by $\overline{a+1} \mapsto \overline{a}$, and so we have 
\[
\virtual{t} = \begin{array}{c|c}
\overline{a+1} & \overline{a}
\\ \vdots & \overline{a+1}
\\ a & \vdots
\end{array}
\xrightarrow[\hspace{50pt}]{f_a^v}
\begin{array}{c|c}
\overline{a} & \overline{a}
\\ \vdots & \overline{a+1}
\\ a+1 & \vdots
\end{array} = f_a^v \virtual{t} = v(t')
\]
since $a \in v'_-$ and $a+1 \in J'$.

\case{Subcase $a+1 \in K$}

We obtain $t'$ by sending $\overline{a+1} \mapsto \overline{a}$ and
\[
\virtual{t} = \begin{array}{c|c}
\overline{a+1} & \overline{a}
\\ \vdots & \vdots
\\ a & a + 1
\end{array}
\xrightarrow[\hspace{50pt}]{f_a^v}
\begin{array}{c|c}
\overline{a} & \overline{a}
\\ \vdots & \vdots
\\ a + 1 & a + 1
\end{array} = f_a^v \virtual{t} = v(t')
\]
since $a \in v'_-$ and $a+1 \in v'_+$.

\case{Case $a \in K$}

We break this into subcases. We generally assume that $a < n$ since $a = n$ is the same as the first subcase below.

\case{Subcase $a+1 \notin t_{\pm}, J_0, J$}

We have $f_a t = 0$ and
\[
\virtual{t} = \begin{array}{c|c}
\overline{a} & \vdots
\\ \vdots & a
\end{array},
\]
and so $f_a^v \virtual{t} = 0$.

\case{Subcase $a+1 \in v_+$}

We have $f_a t = 0$ and
\[
\virtual{t} = \begin{array}{c|c}
\overline{a} & \vdots
\\ \vdots & a + 1
\\ a + 1 & a
\end{array},
\]
which implies $f_a^v \virtual{t} = 0$.

\case{Subcase $a + 1 \in v_-$}

Here $f_a$ sends the $a \mapsto a+1$, and so we have
\[
\virtual{t} = \begin{array}{c|c}
\overline{a} & \overline{a+1}
\\ \overline{a+1} & \vdots
\\ \vdots & a
\end{array}
\xrightarrow[\hspace{50pt}]{f_a^v}
\begin{array}{c|c}
\overline{a} & \overline{a}
\\ \overline{a+1} & \vdots
\\ \vdots & a+1
\end{array} = f_a^v \virtual{t} = v(t')
\]
since $a \in v'_-$ and $a+1 \in K'$.

\case{Subcase $a+1 \in K$}

We have $f_a t = 0$ and 
\[
\virtual{t} = \begin{array}{c|c}
\overline{a} & \vdots
\\ \overline{a+1} & \vdots
\\ \vdots & a+1
\\ \vdots & a
\end{array},
\]
which implies $f_a^v \virtual{t} = 0$.

\case{Subcase $a+1 \in J_0$ or $J$}

We have $f_a t = 0$ and
\[
\virtual{t} = \begin{array}{c|c}
\overline{a} & \overline{a+1}
\\ \vdots & \vdots
\\ a + 1 & a
\end{array},
\]
which implies $f_a^v \virtual{t} = 0$.

\hspace{12pt}

Therefore from the cases above, we have $v \circ f_a = f_a^v \circ v$. Similarly, one can show that $v \circ e_a = e_a^v \circ v$. 
The cases also all show that
\begin{align*}
\gamma_a \varepsilon_a = \virtual{\varepsilon}_b
\qquad \text{and} \qquad
\gamma_a \varphi_a = \virtual{\varphi}_b
\end{align*}
for any $b \in \phi^{-1}(a)$.
Therefore the map $v$ is a virtualization map.
\end{proof}

\begin{prop}
Let $\g_0$ be of classical type with foldings given by Equation~\eqref{eq:classical_embeddings}. The highest weight 
crystal $B(\lambda)$ virtualizes in $B(\Psi(\lambda))$ with the virtualization map $v$ given by $v(u_{\lambda}) \mapsto u_{\Psi(\lambda)}$ 
(recall $u_{\lambda}$ is the unique highest weight element in $B(\lambda)$).
\end{prop}

\begin{proof}
Let $\lambda = \sum_{a \in I_0} c_a \clfw_a$. From Proposition~\ref{prop:single_column_fg} restricted to the 
(unique) classical component $B(\clfw_a) \subseteq B^{a,1}$, Lemma~\ref{lemma:folding_F}, 
Lemma~\ref{lemma:folding_CD}, Lemma~\ref{lemma:folding_BD}, and 
Proposition~\ref{prop:virtual_tensor}, we know that there exists a virtualization map
\[
v \colon \bigotimes_{a \in I_0} B(\clfw_a)^{\otimes c_a} \lhook\joinrel\longrightarrow \bigotimes_{a \in I_0} 
B\bigl(\Psi(\clfw_a) \bigr)^{\otimes c_a}.
\]
If we restrict $v$ to the unique classical component 
$v' \colon B(\lambda) \to B\bigl( \Psi(\lambda) \bigr)$, then $v'$ is the desired virtualization map.
\end{proof}

\bibliographystyle{alpha}
\bibliography{filling}

\end{document}